\documentclass[10pt,a4paper]{amsart} 
\usepackage{amsmath,amsthm,amssymb}
\usepackage{mathabx}
\usepackage[margin=1.5in]{geometry}
\usepackage[initials,msc-links]{amsrefs}
\usepackage{url}
\usepackage{amsaddr}
\usepackage[dvipsnames]{xcolor}
\usepackage[english=american]{csquotes}
\usepackage{enumerate}
\usepackage{textcomp}
\usepackage{mathrsfs}
\usepackage{mathtools}
\usepackage{color, colortbl}
\definecolor{Gray}{gray}{0.9}
\usepackage{bbm}
\usepackage[multiple]{footmisc}
\usepackage{hyperref}
\usepackage{stmaryrd}
\providecommand{\keyword}[1]{{\footnotesize\textbf{\textit{Key words---}} #1}}
\hypersetup{
    bookmarks=true,         
    unicode=false,          
    pdftoolbar=true,        
    pdfmenubar=true,        
    pdffitwindow=false,     
    pdfstartview={FitH},    
    pdftitle={My title},    
    pdfauthor={Author},     
    pdfsubject={Subject},   
    pdfcreator={Creator},   
    pdfproducer={Producer}, 
    pdfkeywords={keyword1, key2, key3}, 
    pdfnewwindow=true,      
    colorlinks=true,       
    linkcolor=blue ,          
    citecolor=blue ,        
    filecolor=magenta,      
    urlcolor=Aquamarine           
}
\usepackage{caption}
\usepackage{tikz}
\usepackage{fancyhdr}
\usepackage{siunitx}
\usepackage{xcolor}
\usepackage{booktabs,colortbl, array}
\usepackage{pgfplotstable}
\pgfplotsset{compat=1.8}

\definecolor{rulecolor}{RGB}{0,71,171}
\definecolor{tableheadcolor}{gray}{0.92}
\colorlet{tableheadcolor}{gray!25} 
\colorlet{tablerowcolor}{gray!10} 
\makeatletter
\DeclareFontFamily{OMX}{MnSymbolE}{}
\DeclareSymbolFont{MnLargeSymbols}{OMX}{MnSymbolE}{m}{n}
\SetSymbolFont{MnLargeSymbols}{bold}{OMX}{MnSymbolE}{b}{n}
\DeclareFontShape{OMX}{MnSymbolE}{m}{n}{
    <-6>  MnSymbolE5
    <6-7>  MnSymbolE6
    <7-8>  MnSymbolE7
    <8-9>  MnSymbolE8
    <9-10> MnSymbolE9
    <10-12> MnSymbolE10
    <12->   MnSymbolE12
}{}
\DeclareFontShape{OMX}{MnSymbolE}{b}{n}{
    <-6>  MnSymbolE-Bold5
    <6-7>  MnSymbolE-Bold6
    <7-8>  MnSymbolE-Bold7
    <8-9>  MnSymbolE-Bold8
    <9-10> MnSymbolE-Bold9
    <10-12> MnSymbolE-Bold10
    <12->   MnSymbolE-Bold12
}{}
\let\llangle\@undefined
\let\rrangle\@undefined
\DeclareMathDelimiter{\llangle}{\mathopen}%
{MnLargeSymbols}{'164}{MnLargeSymbols}{'164}
\DeclareMathDelimiter{\rrangle}{\mathclose}%
{MnLargeSymbols}{'171}{MnLargeSymbols}{'171}
\makeatother

\newtheorem{theorem}{Theorem}[section]
\newtheorem{lemma}[theorem]{Lemma}
\newtheorem{proposition}[theorem]{Proposition}
\newtheorem{corollary}[theorem]{Corollary}

\theoremstyle{definition}
\newtheorem{definition}[theorem]{Definition}
\newtheorem{remark}[theorem]{Remark}

\newtheorem{assumption}[theorem]{Assumption}

\newcommand{\Crm}{\mathrm{C}}
\newcommand{\Drm}{\mathrm{D}}

\newcommand{\Lrm}{\mathrm{L}}

\newcommand{\Trm}{\mathrm{T}}

\newcommand{\Wrm}{\mathrm{W}}

\newcommand{\Acal}{\mathcal{A}}
\newcommand{\Bcal}{\mathcal B}

\newcommand{\Fcal}{\mathcal{F}}
\newcommand{\Gcal}{\mathcal{G}}
\newcommand{\Hcal}{\mathcal{H}}
\newcommand{\Ical}{\mathcal{I}}

\newcommand{\Kcal}{\mathcal{K}}
\newcommand{\Lcal}{\mathcal{L}}
\newcommand{\Mcal}{\mathcal{M}}
\newcommand{\Ncal}{\mathcal{N}}

\newcommand{\Rcal}{\mathcal{R}}
\newcommand{\Scal}{\mathcal{S}}

\newcommand{\Vcal}{\mathcal{V}}
\newcommand{\Wcal}{\mathcal{W}}

\newcommand{\Ascr}{\mathscr{A}}

\newcommand{\Cscr}{\mathscr{C}}

\newcommand{\Fscr}{\mathscr{F}}

\newcommand{\Lscr}{\mathscr{L}}

\newcommand{\Pscr}{\mathscr{P}}

\newcommand{\Sscr}{\mathscr{S}}

\newcommand{\Uscr}{\mathscr{U}}

\newcommand{\Wscr}{\mathscr{W}}

\newcommand{\Abf}{\mathbf{A}}
\newcommand{\Bbf}{\mathbf{B}}
\newcommand{\Cbf}{\mathbf{C}}
\newcommand{\Dbf}{\mathbf{D}}
\newcommand{\Ebf}{\mathbf{E}}
\newcommand{\Fbf}{\mathbf{F}}
\newcommand{\Gbf}{\mathbf{G}}
\newcommand{\Hbf}{\mathbf{H}}
\newcommand{\Ibf}{\mathbf{I}}

\newcommand{\Tbf}{\mathbf{T}}
\newcommand{\Ubf}{\mathbf{U}}

\newcommand{\Nbb}{\mathbb{N}}

\DeclareMathOperator{\id}{id}
\DeclareMathOperator{\Dir}{Dir}
\DeclareMathOperator{\Reg}{Reg}

\DeclareMathOperator{\diam}{diam}

\DeclareMathOperator{\dist}{dist}

\newcommand{\set}[2]{\left\{\, #1 \  \textup{\textbf{:}}\  #2 \,\right\}}

\newcommand{\dpr}[1]{\langle #1 \rangle}

\newcommand{\cl}[1]{\overline{#1}}

\newcommand{\dd}{\;\mathrm{d}}

\newcommand{\N}{\mathbb{N}}
\newcommand{\R}{\mathbb{R}}
\newcommand{\C}{\mathbb{C}}

\newcommand{\Z}{\mathbb{Z}}
\newcommand{\loc}{\mathrm{loc}}

\newcommand{\spt}{\mathrm{spt}}

\newcommand{\Sing}{\mathrm{Sing}}

\newcommand{\toweak}{\rightharpoonup}
\newcommand{\toweakstar}{\overset{*}\rightharpoonup}

\newcommand{\toup}{\uparrow}
\newcommand{\todown}{\downarrow}

\newcommand{\eps}{\epsilon}

\newcommand{\proofstep}[1]{\textit{#1}}

\renewcommand{\eps}{\varepsilon}
\newcommand{\vphi}{\varphi}

\makeatletter
\renewcommand*\env@matrix[1][*\c@MaxMatrixCols c]{%
    \hskip -\arraycolsep
    \let\@ifnextchar\new@ifnextchar
    \array{#1}}
\makeatother

\DeclareMathOperator{\Gr}{Gr}

\DeclareMathOperator{\B}{\mathcal B}

\DeclareMathOperator{\Span}{Span}

\DeclareMathOperator{\Lip}{Lip}

\newcommand{\mres}{\mathbin{\vrule height 1.6ex depth 0pt width
        0.13ex\vrule height 0.13ex depth 0pt width 1.3ex}}

\def\vint_#1{\mathchoice%
    {\mathop{\kern 0.2em\vrule width 0.6em height 0.69678ex depth -0.58065ex
            \kern -0.8em \intop}\nolimits_{\kern -0.4em#1}}%
    {\mathop{\kern 0.1em\vrule width 0.5em height 0.69678ex depth -0.60387ex
            \kern -0.6em \intop}\nolimits_{#1}}%
    {\mathop{\kern 0.1em\vrule width 0.5em height 0.69678ex depth -0.60387ex
            \kern -0.6em \intop}\nolimits_{#1}}%
    {\mathop{\kern 0.1em\vrule width 0.5em height 0.69678ex depth -0.60387ex
            \kern -0.6em \intop}\nolimits_{#1}}}
\makeatletter
\newcommand*{\RangeX}{%
    {%
        \mathpalette\@RangeOf{X}%
    }%
}
\newcommand*{\@RangeOf}[2]{%
    \sbox0{$\m@th#1\mathsf{#2}$}%
    \mathsf{#2}%
    \kern-\wd0 %
    \mkern2.75mu\relax
    \nonscript\mkern.25mu\relax
    \mathsf{#2}%
}
\makeatother

\newcommand{\aveint}[2]{\mathchoice%
    {\mathop{\kern 0.2em\vrule width 0.6em height 0.69678ex depth -0.58065ex
            \kern -0.8em \intop}\nolimits_{\kern -0.45em#1}^{#2}}%
    {\mathop{\kern 0.1em\vrule width 0.5em height 0.69678ex depth -0.60387ex
            \kern -0.6em \intop}\nolimits_{#1}^{#2}}%
    {\mathop{\kern 0.1em\vrule width 0.5em height 0.69678ex depth -0.60387ex
            \kern -0.6em \intop}\nolimits_{#1}^{#2}}%
    {\mathop{\kern 0.1em\vrule width 0.5em height 0.69678ex depth -0.60387ex
            \kern -0.6em \intop}\nolimits_{#1}^{#2}}}

\title[]{An upper Minkowski bound for the interior singular set of area minimizing currents}

\author[A. Skorobogatova]{Anna Skorobogatova}
\address{Department of Mathematics, Fine Hall, Princeton University, Washington Road, Princeton NJ 08540, USA}
\email{\href{mailto:As110@math.princeton.edu}{As110@math.princeton.edu}}

\makeindex

\begin{document}
	
\maketitle

	\begin{abstract}
        We show that for an area minimizing $m$-dimensional integral current $T$ of codimension at least 2 inside a sufficiently regular Riemannian manifold, the upper Minkowski dimension of the interior singular set is at most $m-2$. This provides a strengthening of the existing $(m-2)$-dimensional Hausdorff dimension bound due to Almgren and De Lellis \& Spadaro. As a by-product of the proof, we establish an improvement on the persistence of singularities along the sequence of center manifolds taken to approximate $T$ along blow-up scales.
	\end{abstract}

\keyword{minimal surfaces, area minimizing currents, regularity theory, multiple valued functions, blow-up analysis, center manifold}

\section{Introduction and Main Results}
Integral currents are a natural generalization of $m$-dimensional integer multiplicity smooth oriented manifolds with boundary, and are realized as the weak-$*$ closure of the space of such manifolds, when treated as dual to the space of differential $m$-forms. They allow one to solve the oriented Plateau problem over this class and provide a suitable setting in which to analyze the singular structure of minimizers. We will be working with the following core assumption throughout:
\begin{assumption}\label{asm:curr}
    $T$ is an $m$-dimensional integral current in $\Sigma$, where $\Sigma$ is an $(m + \bar{n})$-dimensional embedded submanifold of class $\Crm^{3,\eps_0}$ in $\R^{m+n} \equiv \R^{m+\bar{n}+l}$, where $\eps_0 \in (0,1)$ is a fixed constant. Assume that $T$ is area-minimizing in $\Sigma$ and that $\bar{n} \geq 2$.
\end{assumption}

We define the interior \emph{regular set} of $T$ to be all points around which $\spt T$ can locally be expressed as the graph of a sufficiently regular map. Namely,
\[
    \Reg T \coloneqq \set{p \in \spt T\setminus \spt(\partial T)}{\substack{\text{$\spt T \cap \Bbf_R(p)$ is a $\Crm^{1,\alpha}$ submanifold of $\R^{m+n}$ for} \\ \text{some $R, \alpha > 0$}}}.
\]

Consequently, we define the (interior) \emph{singular set} of $T$ to be
\[
    \Sing T \coloneqq \spt T \setminus (\spt(\partial T) \cup \Reg T).
\]
A sharp Hausdorff dimension estimate for the interior singular set of $T$ in this setting has been determined in all possible cases. When the codimension $\bar{n}$ is one, the Hausdorff dimension of this set was shown to be at most $m-7$ (with purely isolated singularities when $m=7$) by Federer~\cite{Federer_the_singular_sets_of}, with important contributions from Fleming, De~Giorgi, Almgren and Simons for the lower dimensional cases (see~\cite{Fleming_On_the_oriented_Plateau,De_Giorgi_frontiere,De_Giorgi_una_estensione,Almgren_some_interior_regularity,Simons_minimal_varieties}). It was further shown by Simon in~\cite{Simon_rectifiability} that the singular set is $(m-7)$-rectifiable. The sharpness of these results in codimension one were verified by Bombieri, De Giorgi \& Guisti in~\cite{Bombieri_De_Giorgi_Giusti}. There, the authors demonstrate that Simons' cone
\[
    S \coloneqq \set{x \in \R^8}{x_1^2 + x_2^2 + x_3^2 + x_4^2 = x_5^2 + x_6^2 + x_7^2 + x_8^2} \subset \R^8,
\]
which has an isolated singularity at the origin, is area minimizing.

When $\bar{n} \geq 2$, the seminal work of Almgren~\cite{Almgren_regularity} provides an upper Hausdorff dimension bound of $m-2$ for the interior singular set:
\begin{equation}\label{eq:Almgren}
    \dim_\Hcal(\Sing T) \leq m-2.
\end{equation}
The sharpness of this can be easily seen from considering any holomorphic variety of complex codimension one (real codimension two) with a singular branch point at the origin, such as
\[
    \Gamma = \set{(z,w) \in \C^2}{z^2 = w^3}.
\]
Almgren's original proof has since been simplified and made more transparent by De Lellis \& Spadaro in the series of papers~\cite{DLS14Lp, DLS16centermfld, DLS16blowup}. Furthermore, these authors were also able to relax the a priori assumptions on the ambient manifold $\Sigma$ from $\Crm^5$ to merely $\Crm^{3,\eps_0}$.

An important first step towards classifying the types of singularities of $T$ comes from Almgren's \emph{stratification} (see, for example,~\cite{WhiteStrat}), which tells us that
\begin{equation}\label{eq:Almgrenstrat}
    \dim_{\Hcal}\big(\Scal^k(T)\big) \leq k,
\end{equation}
where
\begin{equation}\label{eq:stratum}
    \Scal^{k}(T) \coloneqq \set{p \in \spt T\setminus\spt\partial T}{\substack{\text{any tangent cone of T at $p$ splits off} \\ \text{no more than a $k$-dimensional subspace}}}.
\end{equation}
If the codimension of $T$ is one, then any point that has a flat tangent cone is necessarily regular. This is due to the local characterization of codimension one integral currents in terms of boundaries of Caccioppoli sets, see~\cite[Theorem~27.6,~Corollary~27.8]{Simon_GMT}, combined with De Giorgi's $\eps$-Regularity Theorem. In addition, we always have $\Scal^{m-1}(T)\setminus\Scal^{m-2}(T) = \emptyset$ for classical area-minimizers. Combining these facts with~\eqref{eq:Almgrenstrat}, the codimension two bound on the singular set follows immediately (the sharper codimension $7$ bound due to Federer~\cite{Federer_the_singular_sets_of} comes from a more elaborate \emph{dimension reduction} argument). This is no longer true in higher codimension; flat singular points $p \in \Scal^m(T)\setminus\Scal^{m-1}(T)$ exist, as exemplified by holomorphic subvarieties of the form
\[
    \Gamma' = \set{(z,w) \in \C^2}{(z-w^2)^3 = w^{1000}}.
\]
This is a $3$-valued perturbation of the regular graph $z = w^2$ locally around the origin, and $0 \in \Scal^m(T)\setminus\Scal^{m-1}(T)$. 

Thus, more work needs to be done in the higher codimension setting, to estimate the size of those singular points at which $T$ is close to being flat. An important step is the removal of the contribution of the `regular part' of the current around singular points, before analyzing the singular behaviour. This is precisely the role of the \emph{center manifold}.

In contrast to the codimension one case, the fine structure of the interior singular set is still unknown in general. In the case $m=2$, Chang~\cite{SXChang} and De Lellis, Spadaro \& Spolaor~\cite{DLSS1, DLSS2, DLSS3} established a structure theorem for the singular set: all singularities are isolated. Chang's result relies heavily on the existence of a \emph{branched center manifold}. The rigorous proof of the existence of such an object is due to the latter named authors, who also generalize the structure result to the wider class of semicalibrated currents and spherical cross sections of area minimizing cones.

See also the recent work of Liu~\cite{liu2021finite}, where the author demonstrates that it is possible for any finite (combinatorial) graph to arise as the singular set of a 3-dimensional area minimizing current in a 7-dimensional closed, compact Riemannian manifold. The singular set produced there, however, consists only of conical and cylindrical singularities; namely, points in $\Scal^{1}(T)$.

The results of this article are an initial step towards better understanding the problem of how to determine the structure of the singular set; in particular, the `flat' singular points $p \in \Scal^{m}(T)\setminus\Scal^{m-1}(T)$.
Before we state our main result, let us introduce some common notation. We will use $\widebar{\dim}_M(E)$ to denote the upper Minkowski dimension of a set $E \subset \R^{m+n}$. Namely, 
\[
    \widebar{\dim}_M(E) \coloneqq \inf\set{s \geq 0}{\limsup_{\delta \todown 0} \delta^{s}N(\delta,E) < \infty},
\]
where $N(\delta,E)$ denotes the smallest number of open balls of radius $\delta$ required to cover $E$. 

For $Q \in \Nbb$, we let
\[
    \Sing_{\geq Q}(T) \coloneqq \Sing T \cap \Drm_{\geq Q}(T) 
\]
where
\[
    \Drm_{\geq Q}(T) \coloneqq \set{p \in \spt T\setminus\spt(\partial T)}{\Theta(T,p) \geq Q}.
\]
When we are referring to singular points of density exactly $Q$, we will use the notation $\Sing_Q T$. We are now ready to state our main result, which is the following:

\begin{theorem}\label{thm:main}
    For $T$ and $\Sigma$ as in Assumption~\ref{asm:curr} and $Q \in \Nbb \setminus  \{0\}$,
    \[
        \widebar{\dim}_M(\Sing_{\geq Q}T\cap\widebar{\Omega}) \leq m-2
    \]
    for any $\Omega \Subset \R^{m+n} \setminus (\partial T \cup \Sing_{\geq Q+1}T)$.
\end{theorem}
This provides a more refined dimension bound on the singular set than the existing one~\eqref{eq:Almgren} due to Almgren. However, we are not able to control the $(m-2)$-dimensional Minkowski content via the methods here. See the recent work~\cite{DLMSV_Rectifiability_and_upper_Minkowski} of De Lellis, Marchese, Spadaro \& Valtorta, where the authors successfully get an $(m-2)$-dimensional Minkowski content bound and establish $(m-2)$-rectifiability for the set of $Q$-points of (non-trivial) multi-valued harmonic maps. There, the authors use \emph{rectifiable Reifenberg} methods to establish this; it is currently unclear as to whether the same approach can be adopted in this setting. There is also the alternative approach in~\cite{KW2018fine} by Krummel \& Wickramasekera, but this seems to heavily rely on the explicit structure of Dir-minimizers on two-dimensional domains and the classificaiton of their frequency values.

Notice that in order to prove Theoreom~\ref{thm:main}, it suffices to establish the estimate
\[
    \widebar{\dim}_M(\Sing_{[Q,Q+\eps]}T\cap\widebar{\Omega}) \leq m-2, \qquad \text{for any $\eps > 0$},
\]
where
\[
    \Sing_{[Q,Q+\eps]}T \coloneqq \Sing T \cap D_{[Q,Q+\eps]} T,
\]
and
\[
     \Drm_{[Q,Q+\eps]}T \coloneqq \set{p \in \spt T \setminus \spt(\partial T)}{Q \leq \Theta(T,p) \leq Q + \eps}.
\]
This is because we have the following consequence of the quantitative stratification of Naber \& Valtorta in~\cite{naber2017singular}:
\begin{lemma}\label{lem:nonintsing}
    For any $\eps > 0$, the following holds. For $T$ and $\Sigma$ as in Assumption~\ref{asm:curr} and $Q \in \Nbb \setminus \{0\}$, the set
    \[
        \Sing_{\geq Q+\eps}T\cap\widebar{\Omega}
    \]
    is $(m-3)$-rectifiable and has finite upper $(m-3)$-Minkowski content, for any $\Omega \Subset \R^{m+n} \setminus (\partial T \cup \Sing_{\geq Q+1}(T))$. 
\end{lemma}
This allows us to ignore contribution from singularities of all non-integer densities above an arbitrarily small threshold. We defer the proof of this to Section~\ref{sct:nonintdensity}. By virtue of the Allard-Almgren $\eps$-Regularity Theorem, in the case $Q = 1$ we have
\[
    \Sing_1 T \subset \Sing_{\geq 1+ \eps}T
\]
for any $\eps > 0$. Thus, in the special case of singular points with multiplicities $Q \in [1,2)$, we immediately deduce the following stronger result:

\begin{proposition}
    For any $\Omega \Subset \R^{m+n} \setminus (\partial T \cup \Sing_{\geq 2}(T))$, the set
    \[
        \Sing_{\geq 1}(T)\cap \widebar{\Omega}           
    \]
    is $(m-3)$-rectifiable and has finite upper $(m-3)$-Minkowski content.
\end{proposition}

Notice that throughout this article, we restrict to the study of interior singularities for our area minimizing current $T$. Boundary singularities are much less clearly understood; Hardt \& Simon demonstrated in~\cite{Hardt_Simon_boundary} that the boundary singular set is empty when the codimension $\bar{n} = 1$, while Allard~\cite{Allard_boundary} demonstrated boundary regularity in any dimension and codimension under a multiplicity one assumption on the boundary and with a convex barrier. More recently in the works~\cite{HM} and~\cite{DLNS_unique, DLNS_Allard} the uniqueness of tangent cones and an Allard-type regularity result at the boundary is shown when $m=2$. However, such a regularity result only holds at \emph{all} boundary points under a convex barrier assumption on the boundary, and to the author's knowledge, the best known result to date in general is density of regular boundary points. See~\cite{DLDPHM_boundary} for a proof of this.

\section*{Acknowledgments} The author is incredibly grateful to Professor Camillo De Lellis, who introduced her to the problem of determining the fine structure of the singular set of high codimension area minimizing currents, and whose unpublished ideas formed the basis of this article. Professor De Lellis also read a preliminary draft of the paper, providing many helpful suggestions, and engaged in several insightful discussions with the author. Without the help and support of Prof. De Lellis and all those that have participated in the interior regularity reading seminars at Princeton, she would have found it near impossible to understand the background work of Almgren and many others in this field.

\section{Notation}
Let us now introduce the frequent notation that we will be using:

\begin{align*}
    & \Acal_Q(\R^k) && \text{the space of $Q$-tuples of vectors in $\R^k$ (see \cite{DLS_MAMS} for more details);} \\
    & \Bbf_r(p) && \text{the $(m+n)$-dimensional Euclidean ball of radius $r$ centered at $p$};\\
    & \Bcal_r(z) && \text{the geodesic ball of radius $r$ centered at $z$ on a given center} \\
    & \ && \text{manifold (see~\cite{DLS16centermfld} for more details);}\\
    &\Hcal^s &&\text{the} \ s \text{-dimensional Hausdorff measure}, \ s \geq 0; \\
    & d_H && \text{the Hausdorff metric on the space of compact subsets of $\R^{m+n}$}; \\
    &\Wrm^{k,p}(\Omega;\Acal_Q) &&\text{the space of} \ Q\text{-valued} \ p\text{-integrable Sobolev maps with} \\ 
    & \ && \text{$p$-integrable distributional derivatives up to order $k \in \Nbb$} \ \text{on} \ \Omega; \\
    &B_r(z,\pi) &&\text{the $m$-dimensional Euclidean ball of radius $r$ and center $z$ in the} \\
    & \ &&\text{$m$-dimensional plane $\pi$. If it is clear from context, we will just} \\
    & \ &&\text{write $B_r(z)$;} \\
    & E^\perp &&\text{The orthogonal complement to the set $E$ with respect to the} \\
    & \ &&\text{standard Euclidean inner product;} \\
    &\Cbf_r(z,\pi) &&\text{the infinite $(m+n)$-dimensional Euclidean cylinder $B_r(z,\pi) + \pi^\perp$}\\
    & \ &&\text{with center $z$, radius $r$ in direction $\pi^\perp$;} \\
    & \iota_{z,r} &&\text{the scaling map $w \mapsto \frac{w-z}{r}$ around the center $z$;} \\
\end{align*}
\begin{align*}
    & \tau_z &&\text{the translation map $w \mapsto w+z$;} \\
    & f_\sharp &&\text{the push-forward under the map $f$;} \\
    & E_{z,r} &&\text{the blow-up $(\iota_{z,r})_\sharp E$ of the set $E$;} \\
    &T_p \Ncal &&\text{the tangent plane to the manifold $\Ncal$ at the point $p \in \Ncal$;} \\
    &\Tbf_{F} &&\text{the current $\sum_{i\in\N}\sum_{j=1}^Q (f^j_i)_\sharp \llbracket M_i \rrbracket$ induced by the push-forward of a}\\
    & \ &&\text{$Q$-valued map $F:M \to \Acal_Q(\R^{m+n})$ on a Borel set $M \subset \R^m$ with} \\
    & \ &&\text{decomposition $F|_{M_i} = \sum_{j=1}^Q \llbracket f^j_i \rrbracket$, $M = \sqcup_i M_i$ as in~\cite[Lemma~1.1]{DLS_multiple_valued}} \\
    & \ &&\text{(see~\cite[Section~1.1]{DLS_multiple_valued} for a more detailed definition);} \\
    & E \Subset F &&\text{The set $E$ is compactly contained in the set $F$, namely, $\cl{E} \subset F$;} \\
    & A \simeq B &&\text{The quantities $A$ and $B$ are comparable, namely, $c_1A \leq B \leq c_2B$} \\
    & \ &&\text{for some $c_1, c_2 > 0$}; \\
    &\Theta(T,p) &&\text{the $m$-dimensional Hausdorff density of $T$ at a given point $p$;} \\
    &\mathbf{p}_\pi &&\text{the orthogonal projection to the $m$-plane $\pi \subset \R^{m+n}$;} \\
    &\Ical^m(X) &&\text{the space of integral $m$-dimensional varifolds on $X$;} \\
    &\Fcal(V_1,V_2) &&\text{the flat distance between the varifolds $V_1$ and $V_2$;} \\
    &\Fbf(V_1,V_2) &&\text{the varifold distance (induced by the weak-$*$ topology) between $V_1$} \\
    &\ &&\text{and $V_2$}. \\
\end{align*}
We refer the reader to a standard text such as~\cite{Federer},~\cite{Simon_GMT} or~\cite{FMorganGMT} for the relevant geometric measure theory background. For further details on $Q$-valued maps and related concepts, we recommend~\cite{DLS_MAMS} and~\cite{DLS_multiple_valued}.

We will often be extracting subsequences for compactness arguments throughout this article. None of the subsequences will be relabelled, unless otherwise stated. \emph{Geometric constants}, which only depend only on $m,n,\bar{n}$ and $Q$, will usually be denoted by $C_0$. When it is not important for the overall argument, dependencies of constants such as $c$ and $C$ will be omitted. Such constants will typically also be geometric; otherwise, the precise dependencies can often be inferred from the text.
\section{Key preliminary results}\label{sct:prelim}

Before stating the important preliminary results needed to prove Theorem~\ref{thm:Minkowskibd}, we remind ourselves that the goal is to prove that for any given $\alpha > 0$, the upper $(m-2+\alpha)$-Minkowski content of $\Sing T$ vanishes:

\[
    \limsup_{r \todown 0} N(\Sing T, r)r^{m-2+\alpha} = 0.
\]
We argue by contradiction; suppose this is not true. As in the proof of the Hausdorff dimension bound, we would like to blow up around an arbitrary $\Hcal^{m-2+\alpha}$-density point $p \in \Sing T$ and combine this with the existence of a flat tangent cone there to reach a contradiction. The problem is that the Minkowski contents do not behave desirably under Hausdorff convergence of sets, and in fact are not even measures so are not particularly easy to work with in general. The Hausdorff measure itself is also not necessarily lower-semicontinuous with respect to Hausdorff convergence, but to avoid issues with this one can relax it to the premeasure $\Hcal^s_\infty$. Null sets are preserved under this replacement, so this does not affect our contradiction argument.

Thus, we first reduce the problem of estimating the upper Minkowski content of $\Sing T$ to one of estimating the Hausdorff dimension. Of course, since the latter is a cruder dimension estimate, we pay the price of having to bound the Hausdorff dimension uniformly for every set in an appropriate compact family. A crucial tool that we will require is Almgren's Work Raccoon Lemma, which allows us to work with the Hausdorff measure instead of the Minkowski content. To the author's knowledge, this result is originally due to Almgren and a version of it is likely contained within~\cite{AlmgrenQvalued}, but first formally appeared in~\cite[Theorem~5.1]{White86}. However, in both of these references, it was used to establish Hausdorff dimension estimates for blow-up limits and as far as the author is aware, it has not yet been formally observed that this result also gives a Minkowski dimension bound.

\begin{lemma}[Work Raccoon Lemma]\label{lem:workraccoon}
    Suppose that we have a non-trivial class $\Cscr$ of compact subsets of $\widebar{\Bbf}_1 \subset \R^{m+n}$ such that
    \begin{enumerate}[(a)]
        \item\label{itm:closedrescale} $\Cscr$ is closed under rescalings, namely, for each $K \in \Cscr$, $x \in \widebar{\Bbf}_1$ and $0 < r \leq 1$,
        \[
            K_{x,r} \cap \widebar{\Bbf}_1 \coloneqq (\iota_{x,r})_\sharp K \cap \widebar{\Bbf}_1 \in \Cscr, 
        \]
        \item\label{itm:closedconv} $\Cscr$ is closed under Hausdorff convergence.
    \end{enumerate}
    Then the family of exponents
    \[
        \Ascr(\Cscr) \coloneqq \set{\alpha \geq 0}{\Hcal^\alpha(K) = 0 \ \text{for every $K \in \Cscr$}}
    \]
    is an open half-line $(\alpha_0,\infty)$.
    
    Moreover, for every $0 < r \leq 1$, the minimal number $N(K,r)$ of balls of radius $r$ required to cover $K$ satisfies
    \begin{equation}\label{eq:Minkowskicpct}
        N(K,r)r^\alpha \leq C(\alpha) \qquad \text{for every $\alpha \in \Ascr(\Cscr)$}.
    \end{equation}
\end{lemma}

We postpone the proof of this Lemma to Section~\ref{sct:workraccoon}, and will henceforth assume its validity.

We will henceforth fix an arbitrary set $\Omega \Subset \R^{m+n}\setminus (\partial T \cup \Sing_{\geq Q + 1} T)$. We now seek an appropriate family of sets to which the Work Raccoon Lemma applies. Recall that, as for the Hausdorff dimension estimate, we are interested in blowing up around points $x \in \Sing_{[Q,Q+\eps]} T \cap \Omega$, in order to understand the size of this singular set. This motivates choosing the compact family of sets to be the smallest one that contains blow-up sequences of the singular set and is closed under Hausdorff convergence and rescalings of the form in~\eqref{itm:closedrescale} of the lemma. 
We formulate this more precisely as follows.

Given $\eps \in (0,1)$, let
\begin{equation}\label{eq:Qpts}
    A(\eps) \coloneqq \Omega\setminus \Sing_{\geq Q + \eps}T.
\end{equation}
By the upper-semicontinuity of the density, this is an open set, and so can be written as an (increasing) countable union of compact sets. Fix any one of these compact sets $\tilde K = \tilde K(\eps)$ and let
\[
    K^0(\eps, \tilde K) \coloneqq \tilde K(\eps) \cap\Sing_{\geq Q}T.
\]
Hence, in order to prove the Minkowski dimension estimate in Theorem~\ref{thm:main}, it suffices to establish it for $K^0(\eps,\tilde K)$. First of all, observe that we may find $R>0$ large enough such that $\Omega \subset \cl{\Bbf}_R$, so $K^0(\eps,\tilde K)$ is a compact subset of $\cl{\Bbf}_R$. Clearly the result of the Work Raccoon Lemma is unchanged by replacing $\cl{\Bbf}_1$ with $\cl{\Bbf}_R$, 
so we may assume that $R = 1$. Our candidate for the family to which we wish to apply the Work Raccon Lemma is then
\[
    \Cscr(\eps,\tilde K) \coloneqq \set{\text{$K$ compact, $K \subset K^\infty$}}{\substack{\textup{\normalsize $K^0_{x_k,r_k}(\eps, \tilde K) \cap \cl{\Bbf}_1 \overset{d_H}\longrightarrow K^\infty$,} \\ \text{\normalsize $x_k \in \R^{m+n}, \ r_k \in (0,1]$ convergent}}}.
\]
Indeed, we can say the following:

\begin{proposition}\label{prop:compactfam}
    Given $\eps \in (0,1)$ and $\tilde K(\eps)$ as defined above, the collection $\Cscr(\eps,\tilde K)$ contains $K^0(\eps, \tilde K)$ and satisfies the assumptions of Lemma~\ref{lem:workraccoon}.
\end{proposition}
The proof of this Proposition is also contained in Section~\ref{sct:workraccoon}. It remains to show that $\alpha_0$ as in Lemma~\ref{lem:workraccoon} satisfies $\alpha_0 \leq m-2$ for this compact family $\Cscr(\eps,\tilde{K})$. 

First of all, by extracting a subsequence, we may assume that given any $K \in \Cscr(\eps,\tilde{K})$, we have
\[
    x_k \to x \qquad \text{and} \qquad r_k \to r,
\]
for the corresponding blow-up scales and centers.
 
Notice that if $r >0$, then
\[
    K^\infty \subset \Sing (T_{x,r})\cap \cl{\Bbf}_1,
\]
since $\cl{\Bbf}_{r_k}(x_k) \overset{d_H}{\longrightarrow} \cl{\Bbf}_r(x)$. Thus in this case the existing dimension estimate~\eqref{eq:Almgren} due to Almgren~\cite{Almgren_regularity} tells us that $\Hcal^{m-2+\alpha}(K^\infty) = 0$ for each $\alpha > 0$. So we may assume that $r_k \todown 0$.

The main result Theorem~\ref{thm:main} can now be concluded from the following:

\begin{theorem}\label{thm:Minkowskibd}
    There exists $\eps \in (0, 1)$ sufficiently small such that the following holds for any compact $\tilde K(\eps) \subset A(\eps)$.
    
    For any $K^\infty$ in the family $\Cscr(\eps,\tilde K)$ of Proposition~\ref{prop:compactfam}, we have the following dichotomy. Either
    \begin{enumerate}[(a)]
        \item\label{eq:a} there exists an $m$-dimensional area minimizing integral current $S$ in $\R^{m+n}$ such that $K^\infty \subset \Sing S$;
        \item\label{eq:b} there exists an $m$-plane $\pi_\infty \subset \R^{m+n}$ and a non-trivial $\Dir$-minimizer $u: \pi_\infty \cong \R^m\times \{0\} \supset B_{\frac{3}{2}} \to \Acal_Q$ with
    \[
        K^\infty \subset \Delta_Q(u) \coloneqq \set{z \in B_{\frac{3}{2}}}{u(z) = Q\llbracket 0 \rrbracket}.
    \]
    \end{enumerate}
    In particular, applying Almgren's dimension estimate~\eqref{eq:Almgren} in case~\eqref{eq:a} and~\cite[Prop.~3.22]{DLS_MAMS} in case~\eqref{eq:b}, we arrive at
    \[
        \Hcal^{m-2+\alpha}(K^\infty) = 0 \qquad \text{for every $\alpha > 0$}.
    \]
\end{theorem}

Unless otherwise mentioned, we will henceforth fix an arbitrary choice of $\eps > 0$ and a compact set $\tilde K (\eps) \subset A(\eps)$, and will omit dependencies on $\eps$ and $\tilde K$ for all the sets defined above. The following section is dedicated to reducing ourselves to the second alternative~\eqref{eq:b} in Theorem~\ref{thm:Minkowskibd}, by characterizing the blow-up limit along a sequence with varying centers.



\section{Reduction to the second altenative in Theorem~\ref{thm:Minkowskibd}}
We hope to use Almgren's center manifold construction to prove the above theorem. Unfortunately, in our reduction to merely establishing a uniform Hausdorff dimension bound of $m-2$ for our whole family of compact sets $\Cscr$, we have been forced to include blow-ups with varying centers. Therefore it is necessary to check that singularities still persist along such diagonal blow-up sequences.

Firstly, for any $K^\infty \in \Cscr$, we would like to ensure that the blow-up centers are singular $Q$-points of $T$, so that the sheets of our current are not collapsing around the centers. This will be important in order to ensure that our blow-up limit is non-trivial.
\begin{proposition}\label{prop:qptcenters}
    Let $T$ and $\Sigma$ be as in Assumption~\ref{asm:curr}. Then for any $K^\infty \in \Cscr$ as in Theorem~\ref{thm:Minkowskibd}, we may assume that the centers of the corresponding blow-ups $K^0_{x_k,r_k}$ satisfy
    \[
        x_k \in K^0.
    \]
\end{proposition}

\begin{proof}
    Consider the original sequences of centers $x_k \in \R^{m+n}$ and blow-up scales $r_k \todown 0$. We may without loss of generality assume that (for $k$ sufficiently large) $K^0\cap\cl{\Bbf}_{2r_k}(x_k) \neq \emptyset$, so we can find points $z_k \in K^0\cap\cl{\Bbf}_{2r_k}(x_k)$. 
    But now we can instead consider the blow-ups $K^0_{z_k,3r_k}\cap\cl{\Bbf}_1$, the support of which contain $\spt K^0_{x_k,r_k}\cap\cl{\Bbf}_1$. For $k$ sufficiently large, 
    \[
        K^0\cap \cl{\Bbf}_{3r_k}(z_k) \subset A,
    \]
    for $A = A(\eps)$ as in~\eqref{eq:Qpts}. This completes the proof, after replacing $\tilde K$ by a slightly larger compact set contained in $A$ if necessary.
\end{proof}

Notice that since $K^0$ is compact, Proposition~\ref{prop:qptcenters} allows us to assume that for any $K \in \Cscr$, the corresponding blow-up centers $x_k$ satisfy $x_k \to x \in K^0$ up to subsequence. In particular, we have
\begin{equation}\label{eq:densityx}
    Q \leq \Theta(T,x)\leq Q+\eps.
\end{equation}
\begin{remark}
    Observe that~\eqref{eq:densityx} holds because we are blowing up a \emph{compact} subset of $\Sing_{[Q,Q+\eps]} T$. Without this, up to the author's knowledge, it remains unknown whether it is possible to rule out the situation that $\Theta(T,x_k) \leq Q+\eps$ for each $k$, but $\Theta(T,x) > Q+\eps$, even under the assumption that $x_k \in \Sing T$. We will see the importance of~\eqref{eq:densityx} when we investigate the structure of these blow-ups in the limit as $k \to \infty$.
\end{remark}

We will henceforth assume the property~\eqref{eq:densityx} for any given $K \in \Cscr$. Another key assumption that we impose throughout will be the following:

\begin{assumption}\label{asm:curr2}
    Let $T$ and $\Sigma$ be as in Assumption~\ref{asm:curr}. Given $K \in \Cscr$ and $x = \lim_{k \to \infty} x_k$, we may assume that $\Sigma \cap \Bbf_{7\sqrt{m}}(x)$ is the graph of a $\Crm^{3,\eps_0}$ function $\Psi_p : \Trm_p\Sigma \cap \Bbf_{7\sqrt{m}}(x) \to \Trm_p\Sigma^\perp$ for every $p \in \Sigma\cap\Bbf_{7\sqrt{m}}(x)$. We may further assume that
    \[
        \boldsymbol{c}(\Sigma)\coloneqq\sup_{p \in \Sigma \cap \Bbf_{7\sqrt{m}}(x)}\|D\Psi_p\|_{\Crm^{2,\eps_0}} \leq \bar{\eps},
    \]
    where $\bar\eps$ will be determined later. This in particular gives us the following uniform control on the second fundamental form of $\Sigma$:
    \[
        \Abf \coloneqq \|A_\Sigma\|_{\Crm^0(\Sigma)} \leq C_0\boldsymbol{c}(\Sigma).
    \]
\end{assumption}

Unfortunately, due to the fact that the centers of the blow-ups are varying, we cannot necessarily deduce that the tangent current obtained in the blow-up limit is a cone; the monotonicity formula is only valid for a fixed center. Nevertheless, we are still able to deduce the following about the limit:
\begin{proposition}\label{prop:weakcpct}
    Suppose that $T$ and $\Sigma$ satisfy Assumption~\ref{asm:curr}. Then we may choose $\eps > 0$ sufficiently small such that the following holds for any $K^\infty \in \Cscr(\eps, \tilde{K})$ with corresponding blow-up centers $x_k \in K^0(\eps,\tilde{K})$ and blow-up scales $r_k \todown 0$, satisfying Assumption~\ref{asm:curr2} for $\bar\eps = \bar\eps(x) > 0$ small enough.
    
    There exists an $m$-dimensional integral area minimizing current $S$ in $\R^{m+n}$, such that up to subsequence,
    \begin{equation}\label{eq:tangent}
        T_{x_k, r_k} \toweakstar S, \qquad \| T_{x_k, r_k}\|(\Bbf) \longrightarrow \| S \|(\Bbf) \qquad \text{for any open ball $\Bbf$},
    \end{equation}
    and either
    \begin{enumerate}[(a)]
        \item $K^\infty \subset \Sing_{\geq Q} S$,\label{itm:singpersist} \\
        \item $S = Q\llbracket \pi_\infty \rrbracket$ for some plane $\pi_\infty$. \label{itm:flattangent}
    \end{enumerate}
\end{proposition}
 

\begin{remark}\label{rmk:Kinfty}
    An important consequence of this proposition is that the limiting set $K^\infty$ is determined by the structure of a tangent current obtained as a weak-$*$ limit of the blow-ups. Indeed, we demonstrate in the proof that
    \[
        K^\infty \subset D_{\geq Q} S.
    \]
    The first possibility is that the tangent current $\Reg_{\geq Q} S = \emptyset$, which leads to the conclusion of case~\eqref{itm:singpersist}. This case covers the first alternative in Theorem~\ref{thm:Minkowskibd}, where the singularities persist to the tangent $S$. The conclusion of the Minkowski dimension estimate in Theorem~\ref{thm:main} then follows immediately from Almgren's Big Regularity Theorem~\eqref{eq:Almgren}.
    
    In case~\eqref{itm:flattangent}, there is a flat tangent current. This can be considered as a first order approximation for $T$ near $x$, which captures the regular behaviour of $T$ there, meanwhile the singularities vanish in the limit. We thus need to show persistence of singularities at the level of a harmonic approximation to $T$ at an infinitesimal scale around $x$. This approximation is constructed in such a way that the regular behaviour is absorbed into the domain, and so singularities do indeed persist.
\end{remark}
 
\begin{proof}
    First of all, by the definition of $K^0$, $x$ is an interior point of $T$ and~\eqref{eq:densityx} tells us that we can find $\delta >0$ sufficiently small such that
    \begin{equation}\label{eq:massratio}
        \frac{\| T\|(\Bbf_\delta (x))}{\omega_m \delta^m} \leq Q + 2\eps.
    \end{equation}
    
    We claim that for any open ball $\Bbf_R \subset \R^{m+n}$ centered at the origin, we have
    \[
        \sup_k \| T_{x_k,r_k} \|(\Bbf_R) \leq Q + 3\eps.
    \] 
    Indeed, due to the pointwise convergence of $x_k$ to $x$,
    \[
        \| T \|(\Bbf_R(x_k)) \longrightarrow \| T \|(\Bbf_R(x)).
    \]
    for all but a countable collection of radii $R$. This can be easily verified by approximating the mass of $T$ on balls via a pairing with appropriate $m$-forms, and exploiting the continuity property of $T$.
    Combining this with~\eqref{eq:massratio} and the fact that $T$ is $m$-dimensional, and by using monotonicity formula around $x_k$ when $k$ is large enough such that $Rr_k < \delta$, we have
    \begin{equation*}\label{eq:curcpct}
        \frac{\| T_{x_k,r_k} \|(\Bbf_R)}{\omega_m R^m} = \frac{T(\Bbf_{R r_k}(x_k))}{\omega_m (R r_k)^m} \leq e^{C\Abf (\delta - Rr_k)}\frac{\|T\|(\Bbf_\delta(x_k))}{\omega_m \delta^m} \leq \frac{\| T \|(\Bbf_{\delta}(x))}{\omega_m \delta^m} + 3\eps,
    \end{equation*}
    provided that we choose $\bar\eps$ sufficiently small, only dependent on $x$.
    
    Since $x$ is an interior point for $T$, we may further choose $k$ large enough such that 
    \begin{equation}\label{eq:zerobdry}
        \partial T_{x_k,r_k} \mres \Bbf_R = 0, \qquad \text{for $k$ sufficiently large.}
    \end{equation}
    Thus, the Federer-Fleming Compactness Theorem (see, for example~\cite[Theorem~32.2]{Simon_GMT}) for normal currents tells us that we can extract a subsequence and an $m$-dimensional area minimizing integral current $S$ on $\R^{m+n}$ for which 
    \[
        T_{x_k,r_k} \toweakstar S \qquad \text{as Radon measures}.
    \]
    Moreover, $\partial S = 0$, due to~\eqref{eq:zerobdry} and the continuity of the boundary operator with respect to weak-$*$ convergence.
     
    To establish the convergence of masses, firstly observe that the lower-semicontinuity of the mass tells us that for any ball $\Bbf$,
    \begin{equation*}\label{eq:masslsc}
        \|S\|(\Bbf) \leq \liminf_k \|T_{x_k,r_k}\|(\Bbf).
    \end{equation*}
    Now suppose that for some ball $\Bbf$ we have $\|S\|(\Bbf) < \liminf_k \|T_{x_k,r_k}\|(\Bbf)$. For $k$ sufficiently large $\partial S \mres \Bbf = \partial T_{x_k,r_k}\mres \Bbf = 0$, so this contradicts the area minimizing property of $T$ and all its rescalings. We can therefore extract a subsequence for which we indeed have 
    \[
        \| T_{x_k,r_k}\|(\Bbf) \longrightarrow \|S\|(\Bbf).
    \]
    In particular, we deduce that for any $R> 0$, 
    \begin{equation}\label{eq:almostflat}
        \frac{\|S \|(\Bbf_{R})}{\omega_m R^m} \leq Q+4\eps.
    \end{equation}
    Furthermore, one can show that 
    \begin{equation}\label{eq:KinftyQ}
        K^\infty \subset D_{\geq Q} S.
    \end{equation}
    Indeed, given any $y \in K^\infty$ and any $\rho > 0$, the local convergence of the masses~\eqref{eq:tangent} tells us that for any $\eta > 0$, we can find $k$ sufficiently large such that
    \begin{align*}
    \frac{\| S \|(\Bbf_\rho(y))}{\omega_m \rho^m} &\geq \frac{\| T_{x_k,r_k} \|(\Bbf_\rho(y))}{\omega_m \rho^m} - \eta \\
    &= \frac{\| T \|(x_k + \Bbf_{\rho r_k}(r_k y))}{\omega_m (\rho r_k)^m} - \eta.
    \end{align*}
    Taking $\rho \todown 0$, we deduce that
    \[
    \Theta(S,y) \geq \Theta(T, x_k + r_k y) - \eta \quad \text{for $k$ sufficiently large}
    \]
    But now, since $x_k + r_k y \to x$, we can exploit the upper-semicontinuity of the density to deduce that $\Theta(S,y) \geq Q - \eta$. Since $\eta$ is arbitrary, we conclude that $\Theta(S,y) \geq Q$. However, we cannot deduce that $\Theta(S,y) \leq Q + c \eps$ for some constant $c > 0$; the density might increase in the limit.
    
    Thus, in order to try and determine the size of $K^\infty$, it makes sense to investigate the structure of the set of high multiplicity points in $\spt S$. There are two possibilities: either there exists a regular point
    \[
        z_0 \in \Reg_{\geq Q} S,
    \]
    or all points in $D_{\geq Q}(S)$ are singular. If the latter holds then~\eqref{eq:KinftyQ} tells us that we must necessarily have $K^\infty \subset \Sing_{\geq Q} S$ and so we are in case~\eqref{itm:singpersist}.
    
    Suppose on the other hand that there exists a point $z_0 \in \Reg_{\geq Q} S$. First of all, let us assume that $\Theta(S,z_0) = Q$. We want to show that this implies $S$ is flat. We plan to use Allard's Regularity Theorem~\cite[Section~8]{Allard_72} to see that $S = Q\llbracket \Gamma \rrbracket$ for some smooth surface $\Gamma$, and that $\Gamma$ is flat at both infinity and locally around $z_0$, so is in fact an $m$-dimensional plane. 
    In order to see this, we will need to apply Allard to $\frac{S}{Q}$. However, we do not even know that this object makes sense a priori. Thus, define
    \[
        E \coloneqq \set{w \in \Reg S}{\Theta (S, w) = Q},
    \]
    and let $\tilde{S} \coloneqq S \mres E$. By definition, $E$ is a $\Crm^{1,\alpha}$-submanifold of $\Sigma$ for some $\alpha \in (0,1)$. We proceed to show that $\tilde{S} = S$.
    
    Notice that $\cl{E}\setminus E \subset \Sing S$. Indeed, given any $z \in \cl{E} \setminus E$, we know that $\Theta(S,z) \geq Q$ by upper-semicontinuity of the density. If $z \in \Reg S$, then in fact we could further conclude that $\Theta(S,z) = Q$, because there is an open neighbourhood of $z$ on which the density agrees with $\Theta(S,z)$, but also $z_k \to z$ for some sequence $\{z_k\} \subset E$.
    
    Again exploiting the dimension estimate~\eqref{eq:Almgren} on the singular set of $S$, we deduce that
    \[
        \Hcal^{m-2+\alpha}(\cl{E}\setminus E) = 0 \qquad \forall \alpha > 0.
    \]
    But recall that $\partial S = 0$ and $E$ is a smooth manifold, so 
    \[
        \spt(\partial \tilde{S}) = \spt(S\mres \partial E) \subset \cl{E}\setminus E,
    \]
    and hence
    \[
        \Hcal^{m-2+\alpha}(\spt(\partial \tilde{S})) = 0 \qquad \forall \alpha > 0.
    \]
    However, $\partial\tilde{S}$ is an $(m-1)$-dimensional integral flat chain, 
    so by a well-known result of Federer, we deduce that
    \[
        \partial\tilde S = 0 = \partial S,
    \]
    and moreover $\tilde S$ is area minimizing, since $S$ is. Now since $\Theta(\tilde{S},\cdot) \equiv Q$, we can consider
    \[
        \bar{S} \coloneqq \frac{\tilde S}{Q}.
    \]
    
    By virtue of~\eqref{eq:almostflat}, for any $R > 0$ we have
    \begin{align}\label{eq:inftytang}
        \frac{\|\bar{S}\|(\Bbf_R(z_0))}{R^m} &\leq \frac{\|\bar{S}\|(\Bbf_{R+|z_0|})}{R^m} \\
        &\leq \frac{\omega_m (R+|z_0|)^m(Q+4\eps)}{Q R^m} \notag\\
        &= \omega_m\Big(1+ \frac{|z_0|}{R} \Big)^m\Big(1+\frac{4\eps}{Q}\Big).\notag
    \end{align}
    Hence, we have a uniform upper bound over $R >0$ on the masses $\|\bar{S}_{z_0,R}\|(\Bbf_1)$. When combining with the zero boundary condition, this allows us to again use the Federer-Fleming Compactness Theorem to extract a sequence $R_k \toup \infty$ and an $m$-dimensional area minimizing cone (with vertex at the origin) $\bar{S}_\infty$ such that $\bar{S}_{z_0,R_k} \toweakstar \bar{S}_\infty$ and $\| \bar S_{z_0,R_k} \|(\Bbf_1) \longrightarrow \| \bar S_\infty \|(\Bbf_1)$.

    The fact that this \emph{tangent at infinity} $\bar S_\infty$ is a cone follows from the monotonicity formula, in the same way as for a tangent at a fixed point.
    
    By considering~\eqref{eq:inftytang} at scales $R_k \toup \infty$, we have
    \[
        \|\bar{S}_\infty\|(\Bbf_1) \leq \omega_m \Big(1+\frac{4\eps}{Q}\Big).
    \]
    Since $\Theta(\bar S_\infty,\cdot) \equiv \Theta(\bar S, \cdot) \equiv 1$, we can choose $\eps > 0$ sufficiently small in order to apply Allard's Regularity Theorem and deduce that $\bar S_\infty$ is a $C^{1,\alpha}$-graph locally near the origin. Since $\bar{S}_\infty$ is a cone, we conclude that
    \[
        \bar S_\infty = \llbracket \pi \rrbracket \quad \text{for some $m$-plane $\pi$.}
    \]
    Note that we can conclude this for $\bar S_\infty$ regardless of the center that we take for the blow-down limit. Now if $\Theta(S, z_0) = Q$, we can combine this with the regularity of $\bar S$ near $z_0$ to deduce that in fact $\bar S~=~\llbracket \pi\cap E \rrbracket$ and hence $\tilde S = Q\llbracket \pi \cap E\rrbracket$. Indeed, since $\| \bar{S}_\infty \|(\Bbf_1) = \omega_m$, we have
    \begin{equation}\label{eq:flatcone}
        \lim_{r \todown 0} \frac{\| \bar S \|(\Bbf_r(z_0))}{\omega_m r^m} = 1 = \lim_{k \to \infty} \frac{\| \bar S \|(\Bbf_{R_k}(z_0))}{\omega_m R_k^m},
    \end{equation}
    so the monotonicity formula allows us to conclude.
    
    It remains to check that $\tilde S = S$. Due to~\eqref{eq:almostflat}, we have
    \[
        \| S - \tilde S \|(\Bbf_R) \leq 4\eps \omega_m R^m \qquad \text{for any $R > 0$.}
    \]
    Thus, for any $y \in \spt(S - \tilde S)$ and any $R>0$, we have
    \[
        \| S - \tilde S \|(\Bbf_R(y)) \leq \omega_m \eps (R+ |y|)^m,
    \]
    as well as
    \[
        \|S-\tilde S\|(\Bbf_R(y)) \geq \omega_m R^m,
    \]
    simply because the density at any point in the support of an integral current is at least 1. Taking $R> |y|$ and then $\eps < \frac{1}{2}$ (if this is not already the case), we indeed have $\tilde S = S$.
    
    Finally, we explain why we cannot have $z_0 \in \Reg_{\geq Q+1} S$. Note that the density at any regular point must be integer-valued, due to the dimension estimate~\cite[Theorem~35.3]{Simon_GMT} for the set of points with non-integer densities. If $\Theta(S, z_0) \geq Q+1$, then~\eqref{eq:flatcone} instead becomes
    \[
        \lim_{r \todown 0} \frac{\| \bar S \|(\Bbf_r(z_0))}{\omega_m r^m} = \frac{Q+1}{Q} > 1 = \lim_{k \to \infty} \frac{\| \bar S \|(\Bbf_{R_k}(z_0))}{\omega_m R_k^m}.
    \]
    This, however, contradicts the monotonicity formula, so cannot occur.
\end{proof}

We have therefore successfully reduced the problem to that of deducing that the second alternative of Theorem~\ref{thm:Minkowskibd} holds true whenever we are in case~\eqref{itm:flattangent} of Proposition~\ref{prop:weakcpct}. The remainder of the article will be dedicated to proving that this is indeed true.

Before continuing, recall that the \emph{excess} $\Ebf(T,\Bbf)$ on a ball is defined as follows. Given any $m$-dimensional plane $\pi$, denote by $\vec{\pi}$ the unit $m$-vector orienting this $m$-plane. Then
\begin{align*}
\Ebf(T,\Bbf, \pi) &\coloneqq \frac{1}{2|\Bbf|}\int_{\Bbf} |\vec{T} - \vec{\pi}|^2\dd\|T\| \qquad \text{for any $m$-plane $\pi \subset \R^{m+n}$}; \\
\Ebf(T,\Bbf) &\coloneqq \inf_{\text{$m$-planes $\pi$}} \Ebf(T,\Bbf,\pi),
\end{align*}
The definition of the cylindrical excess $\Ebf(T,\Cbf)$ for a cylinder $\Cbf(z,\pi) = B(z,\pi) \times \pi^\perp$ is analogously defined.

Applying Propositions~\ref{prop:qptcenters} and~\ref{prop:weakcpct} to the sequences $x_k$ and $r_k$, we can assume that the weak-$*$ limit $S$ of our blow-up sequence is an $m$-plane with multiplicity $Q$. In summary, we henceforth make the following assumptions:
\begin{assumption}\label{asm:cpct}
    Let $T$ and $\Sigma$ be as in Assumption~\ref{asm:curr} and let $\bar\eps, \eps > 0$ be given by Proposition~\ref{prop:weakcpct}. We assume that we have an arbitrary fixed set $K^\infty \in \Cscr(\eps,\tilde{K})$ satisfying Assumption~\ref{asm:curr2}. For this set $K^\infty$, we have associated blow-up centers $x_k$, and scales $r_k \todown 0$ such that for any arbitrarily fixed $y \in K^\infty$, there is a sequence of points $y_k$ with
    \begin{enumerate}[(i)]
        \item $x_k, x \in K^0, \qquad x_k \to x$, \\
        \item $y_k \in K^0_{x_k,r_k}\cap\Bbf_1, \qquad y_k \to y$, \\
        \item $\Ebf(T, \Bbf_{r_k}(x_k)) = \Ebf(T_{x_k,r_k}, \Bbf_1) \longrightarrow 0 \qquad$ as $k \to \infty$.\label{itm:excessdecay}
    \end{enumerate}
\end{assumption} 

The assumption~\eqref{itm:excessdecay} is a consequence of the weak-$*$ convergence and the convergence of masses. Indeed, for any open $(m+n)$-dimensional ball $B$ and any $m$-plane $\pi$, we have:
\begin{align*}
    \Ebf(T_{x_k,r_k},\Bbf, \pi) &= \frac{1}{2|\Bbf|}\int_{\Bbf} |\vec{T}_{x_k,r_k} - \vec{\pi}|^2\dd\|T_{x_k,r_k}\| \\
    &= \frac{\|T_{x_k,r_k}\|(\Bbf)}{|\Bbf|} - \frac{1}{|\Bbf|} \int_{\Bbf} \dpr{\vec{T}_{x_k,r_k}, \vec{\pi}} \dd\|T_{x_k,r_k}\|,
\end{align*}
and so, taking $k \to \infty$, we arrive at
\begin{equation}\label{eq:excesslim}
    \Ebf(T_{x_k,r_k},\Bbf, \pi) \longrightarrow \Ebf(S,\Bbf, \pi) = \frac{\|S\|(\Bbf)}{|\Bbf|} - \frac{1}{|\Bbf|} \int_{\Bbf} \dpr{\vec{S}, \vec{\pi}} \dd\|S\|.
\end{equation}    

Before we continue, we must first ensure that our setup will allow us to approximate our current well by a $Q$-valued map at the scales of the blow-up procedure (see~\cite{DLS14Lp}). In view of this, we will assume the following: 
\begin{assumption}\label{asm:Qfold}
    Let $T$ and $\Sigma$ be as in Assumption~\ref{asm:curr}, and suppose that Assumptions~\ref{asm:cpct} and~\ref{asm:curr2} hold. If we replace $\eps$ from Assumption~\ref{asm:cpct} by $\min\{\eps, \eps_3\}$, then for each $x_k$ we can find an $m$-plane $\pi_k$ such that
    \[
        \mathbf{p}_{\pi_k}^\sharp T \mres \Cbf_{\frac{11\sqrt{m}r_k}{2}}(x_k,\pi_k) = Q\llbracket B_{\frac{11\sqrt{m}r_k}{2}}(x_k, \pi_k)\rrbracket, \qquad \partial T \mres \Cbf_{\frac{11\sqrt{m}r_k}{2}}(x_k,\pi_k) = 0.
    \]
\end{assumption}


From now on, we will work under Assumptions~\ref{asm:curr}-\ref{asm:Qfold}.

\section{Overview of the remaining argument}
Let us now discuss the overall approach towards proving the second alternative in Theorem~\ref{thm:Minkowskibd}. Our goal is to show that $K^\infty \subset \Delta_Q(u)$, where $u$ is some Dir-minimizing $Q$-valued map defined over (an open ball in) an $m$-dimensional plane. We \emph{would like} to argue as follows:
\begin{enumerate}[Step 1.]
    \item Approximate our blow-up sequence $T_{x_k,r_k}$ by `almost' Dir-minimizing $\Wrm^{1,2}$-maps $u_k$ parameterized over $m$-dimensional planes; \label{itm:planeapprox}
    \item renormalize and recenter the $u_k$ appropriately so that, up to subsequence, we have strong convergence to some non-trivial Dir-minimizer $u$ in $\Lrm^2_\loc$;
    \item Check if $u$ is a viable candidate for the map in Theorem~\ref{thm:Minkowskibd}.
\end{enumerate}
This proposed scheme, however, \emph{does not work}. The main problem is that the limit $u$ could potentially be trivial. Thus, the singularities at $x_k$ and $y_k$ will not have persisted in the limit. Such a phenomenon would occur whenever $T$ has a prominent regular part of higher polynomial order coming into contact with its branched singular structure, locally around the blow-up centers. This regular behaviour would dominate in the limit, thus leading to disappearance of singularities.

We overcome this problem via the center manifold construction of Almgren, originally used to estimate the Hausdorff dimension of the singular set. Center manifolds provide a good replacement for $m$-planes as the objects on which we build our approximations in \eqref{itm:planeapprox} of the above desired scheme. The reason for this is that the center manifolds capture the regular behaviour of $T$ around the points $x_k$. We may then construct graphical approximations over the normal bundles of the center manifolds; these graphical approximations should only capture the singular behaviour of $T$ locally. 

We can then instead choose $u$ to be an appropriately normalized limit of these graphs over the center manifolds. The singularities should now persist to this limit; uniform bounds on the \emph{frequency function} along our sequence of approximating graphs will allow us to conclude this. This will allow us to use the information about the size of the singular set for Dir-minimizers to achieve the claimed dimension bound. 

Section~\ref{sct:Almgrencm} is dedicated to the set up of our sequence of center manifolds. In Section~\ref{sct:freq}, we will then discuss the frequency function and its key properties, including the uniform bounds that are then proved in Section~\ref{sct:frequb}. We conclude with the final persistence of singularities argument in Section~\ref{sct:Qpts}.

\section{Almgren's Center Manifold Construction}\label{sct:Almgrencm}

\subsection{The Refining Procedure}\label{sct:refine}

In order to build the sequence of center manifolds along the scales of our blow-up sequence, we require the refining procedure and Whitney decomposition from \cite[Proposition~1.11]{DLS16centermfld} for each $T_{x_k,r_k}$, constructed via stopping time criteria based on the size of the excess and the height of $T_{x_k,r_k}$. 

In view of~\cite{DLS16blowup}, we will be taking this sequence over small intervals containing the scales $r_k$. These intervals will detect the scales at which the current $T$ stops being sufficiently flat (and thus when a given center manifold no longer approximates $T$ sufficiently well), allowing us to rescale and build a new center manifold and normal approximation which will approximate $T$ better at the new scale. We restate these necessary preliminary results here for the convenience of the reader:

\bigskip

Fix $w \in \spt T \setminus \spt(\partial T)$ and $r < \frac{1}{6\sqrt{m}}\dist(w, \spt(\partial T))$. Define
\[
    T' \coloneqq T_{w,r}\mres\Bbf_{6\sqrt{m}}, \qquad \Sigma' \coloneqq \iota_{w,r}(\Sigma). 
\]
We will later choose $w$ to be our blow-up centers $x_k$ and the corresponding $r$ to be scales close to $r_k$.

Let $\Cscr = \bigcup_{j \in \N} \Cscr^j$ be the collection of all closed dyadic $m$-dimensional subcubes of $[-4,4]^m \subset \pi_0$. Namely, $\Cscr^j$ consists of all cubes inside $[-4,4]^m$ of side length $2^{1-j}$ with vertices in $2^{1-j}\Z$. Given $L \in \Cscr^k$, let $\ell(L) \coloneqq 2^{1-j}$ denote the side length of $L$. We call $J \in \Cscr^j$ an ancestor of $L$ if $L \subset J$ and we call $J$ a parent of $L$ (or $L$ a child of $J$) if furthermore $J \in \Cscr^{j-1}$. We denote the center of $L$ by $x_L$.
    
Given a cube $L \in \Cscr$, Assumption~\ref{asm:Qfold} allows us to find $y_L \in \Span \{x_L\}^\perp$ such that 
\[
    p_L \coloneqq (x_L,y_L) \equiv x_L + y_L \in \spt T'.
\]
Notice that the choice of $y_L$ is not unique (since we assume $Q > 1$), but for each $L$ we fix a choice.

Moreover, for any $M_0 > 0$, let $r_L \coloneqq M_0\sqrt{m}\ell(L)$ and let $\Bbf_L \coloneqq \Bbf_{64 r_L}(p_L)$. 
Let $\hat{\pi}_L$ be the optimal plane for the excess of $T'$ in $\Bbf_L$, namely
\[
    \Ebf(T', \Bbf_L) = \Ebf(T', \Bbf_L, \hat{\pi}_L).
\]
Since we wish to remain to also remain within $\Sigma$ when approximating $T$ graphically, we further find a plane $\pi_L \subset T_{p_L}\Sigma$ such that
\[
    |\pi_L - \hat{\pi}_L| = \inf_{\pi \in T_{p_L}\Sigma} |\pi - \hat{\pi}_L|.
\]
or equivalently, $\Ebf(T', \Bbf_L, \pi_L) = \inf_{\pi \in T_{p_L}\Sigma} \Ebf(T', \Bbf_L, \pi)$. The planes $\pi_L$ will orient the cylinders in which we construct the local Lipschitz approximations from~\cite[Theorem~2.4]{DLS14Lp}. This will be discussed in more detail in Section~\ref{sct:centermfld}.  
    
Fix $N_0 \in \Nbb$. We now construct our refining procedure using a Whitney decomposition of $[-4,4]^m\subset \pi_0$ as follows. We first choose constants $\beta_2 = 4\delta_2$ as in~\cite[Assumption~1.8]{DLS16centermfld}, and fix constants $C_e$ and $C_h$ to be determined later with the dependencies of~\cite[Assumption~1.9]{DLS16centermfld}. 

Given $j \geq N_0$, we set up three subfamilies of cubes $\Wscr^{j}_e$, $\Wscr^{j}_h$, $\Wscr^{j}_n \subset \Wscr$ and place cubes $L \in \Cscr$ into these families inductively, starting from $j = N_0$, as follows:  
\begin{align}
    &\text{Let $L \in \Wscr^{j}_e$} \ \text{if} \ \Ebf(T',\Bbf_L) > C_e\boldsymbol{m}_0\ell(L)^{2-2\delta_2}; \tag{EX}\label{itm:excess} \\
    &\text{Let $L \in \Wscr^{j}_h$} \ \text{if} \ L \notin \Wscr^j_e \ \text{and} \ \mathbf{h}(T',\Bbf_L) > C_h \boldsymbol{m}_0^{\frac{1}{2m}}\ell(L)^{1+\beta_2}; \tag{HT}\label{itm:height} \\
    &\text{Let $L \in \Wscr^{j}_n$} \ \text{if} \ L\notin \Wscr^{j}_e\cup\Wscr^{j}_h \ \text{but $L$ intersects an element of $\Wscr^{j-1}$.}\tag{NN}\label{itm:nn}
\end{align}
Here
\[
    \boldsymbol{m}_0 \coloneqq \max\{\boldsymbol{c}(\Sigma)^2, \Ebf(T',\Bbf_{6\sqrt{m}}) \leq \eps_2^2,
\]
and $\eps_2$ is a small positive constant to be determined later, with the dependencies of~\cite[Assumption~1.9]{DLS16centermfld}. If one of these conditions holds for $L$, then STOP refining. Define
\[
    \Wscr^{j} \coloneqq \Wscr^j_e\cup\Wscr^{j}_h\cup\Wscr^{j}_n \subset \Cscr^j \qquad \text{and} \qquad \Wscr \coloneqq \bigcup_{j \geq N_0} \Wscr^{j}.
\]
Otherwise, place $L$ in the family $\Sscr^{j} \subset \Cscr^j$ and apply the above stopping conditions to the children cubes $L' \in \Cscr^{j+1}$ of $L$. Furthermore, place all cubes $L \in \Cscr^j$ with $j < N_0$ in $\Sscr \coloneqq \bigcup_{j \geq N_0} \Sscr^{j}$. In other words, $\Sscr$ is characterized by
\[
    J \in \Sscr \qquad \iff \qquad \text{$J \in \Cscr$ and either $J$ has a child $L \in \Wscr$ or $\ell(J) > 2^{1-N_0}$.}
\]
Continue this refining procedure inductively. 
    
Observe that given any $x \in [-4,4]^m \subset \pi_0$, this refinement procedure either stops on some $L \in \Wscr$ containing $x$, or we can continue refining indefinitely around $x$. We denote by $\Gamma$ the set of points where the latter occurs:
\begin{equation*}
    \boldsymbol{\Gamma} \coloneqq [-4,4]^m\setminus \bigcup_{L \in \Wscr} L = \bigcap_{j \geq N_0} \bigcup_{L \in \Sscr^{j}} L.
\end{equation*}
Notice that for any $x \in \Gamma$, there exists a sequence of cubes $L_j \in \Sscr^j$ with $x \in L_j$. Moreover, we make the obvious but important observation that
\begin{equation}\label{eq:badcubegdparent}
    L \in \Wscr^{j}, \ j > N_0 \quad \implies \quad \text{the parent cube of $L$ lies in $\Sscr^{j-1}$.} 
\end{equation}
We recall the following result from~\cite{DLS16centermfld}, which tells us that the above construction indeed gives a Whitney decomposition.
    
\begin{proposition}[Whitney decomposition, \cite{DLS16centermfld},~Proposition~1.11]\label{prop:Whitney}
    Provided that Assumptions~1.8 and 1.9 from \cite{DLS16centermfld} hold and $\eps_2$ is sufficiently small, we can conclude that $(\boldsymbol{\Gamma}, \Wscr)$ is a Whitney decomposition of $[-4,4]^m \subset \pi_0$. More precisely,
    \begin{align}
        &\boldsymbol{\Gamma}\cup \bigcup_{L \in \Wscr} L = [-4,4]^m \ \text{and $\boldsymbol{\Gamma}$ does not intersect any cube $L \in \Wscr$};\tag{w1}\label{eq:w1} \\
        &\text{any pair of cubes $L, \ L' \in \Wscr$ have disjoint interiors};\tag{w2}\label{eq:w2} \\
        &\text{if $L, \ L' \in \Wscr$ and $L \cap L' \neq \emptyset$, then $\frac{1}{2}\ell(L) \leq \ell(L') \leq 2\ell(L)$}.\tag{w3}\label{eq:w3}
    \end{align}
    These conditions further imply that
    \begin{equation}\label{eq:Whitneysep}
        \mathrm{sep}(\boldsymbol{\Gamma},L) \coloneqq \inf\set{|x-y|}{x \in L, y \in \boldsymbol{\Gamma}} \geq 2 \ell(L) \quad \text{for every $L \in \Wscr$.}
    \end{equation}
    Moreover, for any choice of $M_0$ and $N_0$, there exists $C^*=C^*(M_0,N_0)$ such that whenever $C_e \geq C^*$ and $C_h \geq C^*C_e$, we have
    \begin{equation}\label{eq:smallbadcubes}
        \Wscr^{j} = \emptyset \qquad \text{for every $j \leq N_0 + 6$},
    \end{equation}
    and
    \begin{align}
        &\Ebf(T',\Bbf_J) \leq C_e \boldsymbol{m}_0\ell(J)^{2-2\delta_2}, \quad \mathbf{h}(T',\Bbf_J) \leq C_h \boldsymbol{m}_0^{\frac{1}{2m}} \ell(J)^{1+\beta_2} \quad \text{for $J \in \Sscr$};\label{eq:goodcubesbds} \\
        &\Ebf(T',\Bbf_L) \leq C \boldsymbol{m}_0\ell(L)^{2-2\delta_2}, \quad \mathbf{h}(T',\Bbf_L) \leq C \boldsymbol{m}_0^{\frac{1}{2m}} \ell(L)^{1+\beta_2} \quad \text{for $L \in \Wscr$}\label{eq:badcubesbds},
    \end{align}
    where $C = C(\beta_2,\delta_2,M_0,N_0,C_e,C_h)$.
\end{proposition}
Note that we do not require an upper bound on the separation in terms of the size of the cube, despite the fact that this additional condition is usually included in the criteria for a Whitney decomposition. The choices of $M_0$ and $N_0$ will be determined by Propositions~\ref{prop:excesssplitting} and~\ref{prop:comparecentermfld}. When combined with the stopping conditions, observe that~\eqref{eq:badcubesbds} tells us that for any $L \in \Wscr$, the excess of $T'$ over the ball $\Bbf_L$ is comparable to $\boldsymbol{m}_0 \ell(L)^{2-2\delta_2}$ and the height is comparable to $\boldsymbol{m}_0^{\frac{1}{2m}} \ell(L)^{1+\beta_2}$.

The first two conditions~\eqref{eq:w1} and~\eqref{eq:w2} are an automatic consequence of the construction. The validity of the remaining conditions relies on the observation~\eqref{eq:badcubegdparent} mentioned above and a quantitative control~\cite[Proposition~4.1]{DLS16centermfld} on the tilting of the optimal planes for the excess. We omit the details, and instead refer the reader to~\cite[Section~4]{DLS16centermfld}.


\subsection{The Center Manifold $\Mcal$}\label{sct:centermfld}
Given any cube $L \in \Sscr \cup \Wscr$, we wish to consider Almgren's strong Lipschitz approximation $f_L$ from~\cite[Theorem~2.4]{DLS14Lp} locally for $T'\mres \Cbf_{32 r_L}(p_L,\pi_L)$. Let us recall this result here (with the proof omitted), for the benefit of the reader.

\begin{theorem}[Almgren's strong approximation,~\cite{DLS14Lp}, Theorem~2.4]\label{thm:strongLip}
    There exist constants $C,\gamma_1, \eps_1 > 0$ (depending on $m,n,\bar{n},Q$) with the following property. Assume that $T'$ is area minimizing, satisfies~\cite[Assumption~2.1]{DLS14Lp} in the cylinder $\Cbf_{4r}(p)$ and $E = \Ebf(T',\Cbf_{4r}(p),\pi) < \eps_1$ for some $m$-plane $\pi \subset \R^{m+n}$. Then, there is a map $f: \pi \supset \B_r(p) \to \Acal_Q(\pi^\perp)$ with $\spt(f) \subset \Sigma$ and a closed set $K \subset B_r(p)$ such that
    \begin{align*}
        &\Lip(f) \leq CE^{\gamma_1}, \\
        &\Gbf_f\mres(K \times \pi^\perp) = T'\mres(K \times \pi^\perp) \quad \text{and} \quad |B_r(p)\setminus K| \leq CE^{\gamma_1}(E + r^2\Abf^2) r^m, \\
        &\left|\|T'\|(\Cbf_{\sigma r}(p)) - Q\omega_m(\sigma r)^m - \frac{1}{2}\int_{B_{\sigma r}(p)} |Df|^2\right| \leq CE^{\gamma_1} (E + r^2\Abf^2) r^m \quad \forall \sigma \in (0,1).
    \end{align*}
    If in addition $\boldsymbol{h}(T',\Cbf_{4r}(p),\pi) \coloneqq \sup\set{|\mathbf{p}^\perp(w) - \mathbf{p}^\perp(z)|}{w,z \in \spt T' \cap \Cbf_{4r}(p)} \leq r$, then
    \begin{equation}\label{eq:oscbd}
        \mathrm{osc}(f) \leq C\mathbf{h}(T',\Cbf_{4r}(p),\pi) + C(E^{\frac{1}{2}} + r\Abf)r.
    \end{equation}
\end{theorem}
In addition, let us recall that provided the excess is sufficiently small, there is the following persistence of $Q$-points from the current to the strong Lipschitz approximation $f$ in Theorem~\ref{thm:strongLip}.
\begin{theorem}[Persistence of $Q$-points I, \cite{DLS14Lp}, Theorem~2.7]\label{thm:Qptspersistf}
    For every $\hat{\delta},C^* > 0$, there is $\bar{s} \in \left(0,\frac{1}{2}\right)$ such that, for every $s < \bar{s}$, there exists $\hat{\eps}(s,C^*,\hat{\delta}) > 0$ with the following property. If $T'$ is as in Theorem~\ref{thm:strongLip}, $E \coloneqq \Ebf(T',\Cbf_{4r}(p),\pi) < \hat{\eps}$, $r^2\Abf^2 \leq C^*E$ and $\Theta(T,(w_0,z_0)) = Q$ for some $(w_0,z_0) \in (\pi \times \pi^\perp)\cap \Cbf_{\frac{r}{2}}(p,\pi)$, then the approximation $f$ of Theorem~\ref{thm:strongLip} satisfies
    \begin{equation}\label{eq:Qpts1}
        \int_{B_{sr}(w_0)}\Gcal(f,Q\llbracket \boldsymbol{\eta}\circ f\rrbracket)^2  \leq \hat{\delta}s^mr^{m+2}E.
    \end{equation}
\end{theorem}
We omit the proof here, but we make the following important observation, that will be useful at a later stage. Notice that~\eqref{eq:Qpts1} can in fact be improved to the height bound
\begin{equation}\label{eq:Qpts2}
    \int_{B_{sr}(w_0)}\Gcal(f,Q\llbracket \boldsymbol{\eta}\circ f\rrbracket)^2  \leq Cs^{m+\gamma}r^{m+2}E,
\end{equation}
where $\gamma > 0$ is a geometric constant given by~\cite[Lemma~5.2]{Spolaor_15} and $C$ is independent of $s$. This improved estimate is a consequence of the alternative proof in~\cite[Section~5.3]{Spolaor_15}, which crucially uses the improved height bound~\cite[Theorem~1.5]{Spolaor_15}.

From now on, we refer to such an approximation $f_L$ as a \emph{$\pi_L$-approximation}. In order to construct the $\pi_L$-approximations, the assumptions of Theorem~\ref{thm:strongLip} must be satisfied. Namely, we require that $\Ebf(T', \Cbf_{32 r_L}(p_L,\pi_L))$ remains below a given small threshold, independently of the choice of $L$.

Let $\rho_\lambda \coloneqq \lambda^{-m}\rho(\frac{\cdot}{\lambda})$, $\lambda > 0$, denote the $\Lrm^1$-invariant rescaling of a smooth bump function $\rho \in \Crm_c^\infty(B_1(0,\pi_L))$ with $\int \rho = 1$ and $\int |z|^2 \rho(z) \dd z = 0$. Define the \emph{average of the sheets} $\boldsymbol{\eta}\circ f$ of a $Q$-valued map $f = \sum_i \llbracket f_i \rrbracket$ to be $\boldsymbol{\eta}\circ f \coloneqq \frac{1}{Q}\sum_{i} f_i$.  Then we have the following:
\begin{lemma}[\cite{DLS16centermfld},~Lemma~1.15]\label{lem:reparam}
    Let the assumptions of Proposition~\ref{prop:Whitney} hold, and assume $C_e \geq C^*$ and $C_h \geq C^* C_e$. For any choice of the other parameters, if $\eps_2$ is sufficiently small, then the current $T'\mres \Cbf_{32 r_L}(p_L,\pi_L)$ satisfies the assumptions of Theorem~\ref{thm:strongLip} for any $L \in \Wscr \cup \Sscr$. 
    
    Moreover, if $f_L:B_{8r_L}(p_L,\pi_L) \to \Acal_Q(\pi_L^\perp)$ is a $\pi_L$-approximation, denote by $\hat{h}_L: B_{7r_L}(p_L,\pi_L)\to \pi_L^\perp$ its smoothed average $\hat{h}_L \coloneqq (\boldsymbol{\eta}\circ f_L) * \rho_{\ell(L)}$ and by $\bar{h}_L$ the map $\mathbf{p}_{T_{p_L}\Sigma}(\hat{h}_L)$, which takes values in the plane $\varpi_L \coloneqq T_{p_L}\Sigma \cap \pi_L^\perp$. If we let $h_L$ be the map $z \mapsto h_L(z) \coloneqq (\bar{h}_L(z),\boldsymbol{\Psi}_{p_L}(x,\bar{h}_L(z))) \in \varpi_L \times T_{p_L} \Sigma^\perp$, then there is a smooth map $g_L: B_{4r_L}(p_L,\pi_0) \to \pi_0^\perp$ such that $\Gbf_{g_L} = \Gbf_{h_L} \mres \Cbf_{4r_L}(p_L,\pi_0)$.
\end{lemma}
We refer to the maps $h_L$ and $g_L$ as \emph{tilted $L$-interpolating functions} and \emph{$L$-interpolating functions}, respectively. Thus, the $\pi_L$-approximations are indeed well-defined for all cubes $L \in \Wscr \cup \Sscr$, provided that we choose $\eps_2$ sufficiently small (dependent on $\eps_1$).

We omit the proof of Lemma~\ref{lem:reparam}; see~\cite{DLS16centermfld} for this. Roughly speaking, the latter part of this Lemma tells us that for each cube $L \in \Wscr \cup \Sscr$, we can take the $\pi_L$-approximation $f_L$ and average the sheets, smooth out and correct the image to lie within $\Sigma$, to give a smooth map that can be reparameterized over the original reference plane $\pi_0$. This is possible due to the control~\cite[Proposition~4.1]{DLS16centermfld} over the tilting of the optimal planes; we refer the reader to~\cite[Lemma~B.1, Proposition~4.2]{DLS16centermfld} for the details.

We are now ready to patch together these $L$-interpolating functions over $\Pscr^j \coloneqq \Sscr^{j}\cup\bigcup_{i=N_0}^j\Wscr^{i}$ for each fixed level $j$ in our partition of $[-4,4]^m \subset \pi_0$. More precisely, let $\vartheta \in \Crm_c^\infty\left(\left[-\frac{17}{16},\frac{17}{16}\right]\right)$ and given any $L \in \Pscr^j$, define $\vartheta_L \coloneqq \vartheta(\frac{\cdot - x_L}{\ell(L)})$. Then take the corresponding partition of unity function
\[
    \hat{\vartheta}_{L,j} \coloneqq \frac{\vartheta_L}{\sum_{L \in \Pscr^j} \vartheta_L}.
\]
This is well-defined and indeed forms a partition of unity over $L \in \Pscr^j$, because of the nature of the cubes $L$. Now consider
\[
    \hat{\vphi}_j \coloneqq \sum_{L \in \Pscr^j} \hat{\vartheta}_{L,j} g_L \qquad \text{on $(-4,4)^m \subset \pi_0$.}
\]
We have almost finished constructing the center manifold; it remains to once again correct these interpolations $\hat{\vphi}_j$ to ensure that the images remain within $\Sigma$. To be more precise, we let $\bar{\vphi}_j$ be the function constructed from the $\bar{n}$ components of $\hat{\vphi}_j$ that live within $T_{0}\Sigma \cap \pi_0^\perp$ and we define the \emph{glued interpolation}
\[
    \vphi_j(y) \coloneqq (\bar{\vphi}_j(y), \boldsymbol{\Psi}_0(y,\bar{\vphi}_j(y))), \qquad y \in (-4,4)^m \subset \pi_0.
\]
This iterative construction over $j \in \Nbb$ enables us to get compactness in the $\Crm^3$-topology; leading to the existence of the center manifold. This is summarized as follows:
\begin{proposition}[Existence of the center manifold, \cite{DLS16centermfld},~Theorem~1.17]\label{prop:centermfld}
    Assume that the hypotheses of Lemma~\ref{lem:reparam} hold and let $\kappa \coloneqq \min\{\frac{\eps_0}{2},\frac{\beta_2}{4}\}$. Provided that $\eps_2$ is sufficiently small, for any choice of the other parameters we have
    \begin{enumerate}[(i)]
        \item\label{itm:cmcurv} $\|D\vphi_j\|_{\Crm^{2,\kappa}} \leq C\boldsymbol{m}_0^{\frac{1}{2}}$ and $\|\vphi_j\|_{\Crm^0} \leq C\boldsymbol{m}_0^{\frac{1}{2m}}$ for $C = C(\beta_2,\delta_2, M_0,N_0,C_e,C_h))$;
        \item if $L \in \Wscr^i$ and $H$ is a cube concentric to $L$ with $\ell(H) = \frac{9}{8}\ell(L)$, then $\vphi_j \equiv \vphi_k$ on $H$ for any $j, k \geq i + 2$;
        \item $\vphi_j \overset{\Crm^3}{\longrightarrow} \boldsymbol{\vphi}$, and $\Mcal \coloneqq \Gr (\boldsymbol{\vphi}|_{(-4,4)^m})$ is a $\Crm^{3,\kappa}$-submanifold of $\Sigma$.
    \end{enumerate}
\end{proposition}
See~\cite[Section~4.4]{DLS16centermfld} for a proof of this. The general idea is that the operations of smoothing out and averaging the sheets preserve the `almost harmonicity' of the graphical approximations for $T'$, resulting in uniform elliptic estimates over $j \in \Nbb$ for \emph{single-valued} maps over $\pi_0$. Note that the optimal regularity of the center manifold is governed by the regularity of $\Sigma$. We only require $\Crm^{3,\kappa}$-regularity in order to get the necessary variational estimates for the $\Mcal$-normal approximations appearing in Section~\ref{sct:Mnormal}; this was noticed by De Lellis \&~Spadaro in their simplified proof of Almgren's dimension estimate~\eqref{eq:Almgren}. In conclusion, we may define:
\begin{definition}[center manifold]\label{def:cm}
    We call $\Mcal$ the center manifold for $T'$ relative to $\pi_0$. Letting $\boldsymbol{\Phi}(y)\coloneqq(y,\boldsymbol{\vphi}(y))$, we define the \emph{contact set} to be $\boldsymbol{\Phi}(\boldsymbol{\Gamma})$, where the pair $(\boldsymbol{\Gamma},\Wscr)$ is the Whitney decomposition associated to $\Mcal$.
\end{definition}

Intuitively, the contact set is the part of $\Mcal$ constructed above the places in $[-4,4]^m$ where we are able to refine indefinitely as a result of both excess and the height remaining sufficiently small at every scale. On the contact set, we expect that
    \[
        T'\mres \boldsymbol{\Phi}(\Gamma) \equiv Q\llbracket \Mcal \rrbracket \mres \boldsymbol{\Phi}(\Gamma).
    \]
Given $L \in \Wscr$, let $H$ be the cube concentric to $L$ with $\ell(H) = \frac{17}{16}\ell(L)$, and define
\[
    \Lcal = \Lcal(L) \coloneqq \boldsymbol{\Phi}\left(H\cap \left[-7/2,7/2\right]^m\right)
\]
to be the \emph{Whitney region} of $L$ on $\Mcal$.

\subsection{The $\Mcal$-Normal Approximations}\label{sct:Mnormal}

Now that we have constructed our center manifold $\Mcal$, we need to better approximate $T'$ in those regions where $\Mcal$ is not a sufficiently good approximation for $T'$; namely, away from the contact set. We do this by building an approximating graph for $T'$ over the normal bundle of the center manifold. This is what we will refer to as an \emph{$\Mcal$-normal approximation}. Before we state the precise definition, we need to introduce the notion of an orthogonal projection map to $\Mcal$:
\begin{assumption}\label{asm:projcm}
    Let $\Ubf$ be the 1-tubular neighbourhood of $\Mcal$, namely
    \[
        \Ubf \coloneqq \set{z \in \R^{m+n}}{\exists ! \ w \eqqcolon \mathbf{p}(z) \in \Mcal \ \text{with $|z-w| <1$ and $z-w \in (T_w \Mcal)^\perp$}}.
    \]
    Then, for any choice of the other parameters, we choose $\eps_2$ to be small enough such that $\mathbf{p}:\Ubf \to \Mcal$ extends to a $\Crm^{2,\kappa}$ map on $\cl{\Ubf}$ and $\mathbf{p}^{-1}(w) = w + \cl{(B_1(0,(T_w \Mcal)^\perp))}$ for every $w \in \Mcal$. For each $w \in \Mcal$, $\mathbf{p}^\perp(w,\cdot):\R^{m+n} \to \R^{m+n}$ will denote the orthogonal projection onto $(T_x \Mcal)^\perp$.
\end{assumption}

\begin{definition}
    An $\Mcal$-normal approximation of $T'$ is a pair $(\Kcal,F)$ such that $F:\Mcal \to \Acal_Q(\Ubf)$ is Lipschitz and has the form
    \[
        F(z) = \sum_{i=1}^Q \llbracket z + N_i(z)\rrbracket, \quad \text{$N_i(z) \in (T_z\Mcal)^\perp$ and $z+N_i(z) \in \Sigma$ for every $z$ and $i$},
    \]
    and $\Kcal \subset \Mcal$ is closed, contains $\boldsymbol{\Phi}\Big(\boldsymbol{\Gamma}\cap\Big[-\frac{7}{2},\frac{7}{2}\Big]^m\Big)$ and $\Tbf_F\mres \mathbf{p}^{-1}(\Kcal) = T\mres \mathbf{p}^{-1}(\Kcal)$.
\end{definition}

We will often abuse notation by referring to the \emph{normal part} of $F$
\[
    N = \sum_i \llbracket N_i\rrbracket : \Mcal \to \Acal_Q(\R^{m+n})
\]
as the $\Mcal$-normal approximation.

We want to see that it is possible to construct an $\Mcal$-normal approximation for $T'$ in a way that the approximation still behaves well outside of the contact set. We want to have \emph{local estimates} on our $\Mcal$-normal approximations on each Whitney region $\Lcal$ in terms of the size of the corresponding cube $L$. Indeed this can be done by again using the properties of the $\pi_L$-approximations $f_L$ from Theorem~\ref{thm:strongLip}. We recall the estimates here, for the benefit of the reader:
\begin{theorem}[$\Mcal$-normal approximation]\label{thm:Nlocalest}
    Let $\gamma_2 \coloneqq \frac{\gamma_1}{4}$, where $\gamma_1$ is as in Theorem~\ref{thm:strongLip}. Then, under the hypotheses of \cite[Theorem~1.17]{DLS16centermfld}, for $\eps_2$ sufficiently small there exists an $\Mcal$-normal approximation $(\Kcal,F)$ such that the following estimates hold on every Whitney region $\Lcal$ associated to a cube $L \in \Wscr$, with constants $C=C(\beta_2,\delta_2,M_0,N_0,C_e,C_h)$: 
    \begin{align}
        &\Lip(N|_\Lcal) \leq C\boldsymbol{m}_0^{\gamma_2}\ell(L)^{\gamma_2}, \qquad \|N|_\Lcal \|_{\Crm^0} \leq C\boldsymbol{m}_0^{\frac{1}{2m}}\ell(L)^{1+\beta_2}, \label{eq:normalLip}\\
        &|\Lcal\setminus \Kcal| + \|\Tbf_{F} - T'\|(\mathbf{p}^{-1}(\Lcal)) \leq C\boldsymbol{m}_0^{1+\gamma_2}\ell(L)^{m+2+\gamma_2},\label{eq:normalerror} \\
        & \int_\Lcal |DN|^2 \leq C\boldsymbol{m}_0 \ell(L)^{m+2-2\delta_2}\label{eq:normalDir}.
    \end{align}
    Moreover, given any $a > 0$ and any Borel subset $\Vcal \subset \Lcal$, we have
    \begin{equation}\label{eq:normalavgest}
        \int_\Vcal |\boldsymbol{\eta}\circ N| \leq C \boldsymbol{m}_0\big(\ell(L)^{m+3+\frac{\beta}{3}} + a \ell(L)^{2+\frac{\gamma_2}{2}}|\Vcal|\big) +\frac{C}{a}\int_\Vcal \Gcal\big(N,Q\llbracket \boldsymbol{\eta}\circ N\rrbracket \big)^{2+\gamma_2}.
    \end{equation}   
\end{theorem}
For this $\Mcal$-normal approximation, the set $\Kcal$ consists of those points in $\Mcal$ above which (with respect to $\mathbf{p}$) $T$ is graphical.

The main idea behind the proof of this theorem is to show that $T$ can be approximated by a Lipschitz map that comes from the graphical approximation in Theorem~\ref{thm:strongLip}, but is reparameterized over $\Mcal$. This is again first done locally, over the Whitney regions $\Lcal$ of the cubes $L$; see~\cite[Theorem~5.1]{DLS_multiple_valued}. One can easily check the compatibility of these local graphical approximations over the region $\Kcal$ where they agree with $T'$. To construct $F$ away from $\Kcal$, we follow these steps:
\begin{enumerate}[Step 1.]
    \item First extend $F$ to a Lipschitz map (with little increase in the Lipschitz constant) $\bar{F}$ taking values in $\Acal_Q(\Ubf)$, using~\cite[Theorem~1.7]{DLS_MAMS};
    \item then modify $\bar{F}$ to take the form $\hat{F}(z) = \sum_i \llbracket z + \hat{N}_i(z) \rrbracket$ with $\hat{N}_i(z) \in (T_z\Mcal)^\perp$ for each $z$, by replacing $\bar{N}(z)$ with the projection 
    \[
        \hat{N}(z) = \sum_i\llbracket \mathbf{p}^\perp(z,\bar{N}_i(z))\rrbracket
    \]
    to the normal bundle;
    \item finally, let $\varpi_z \coloneqq T_z\Sigma \cap (T_z \Mcal)^\perp$ and replace each component $\hat{N}_i(z)$ with $\mathbf{p}_{\varpi_z}(\hat{N}_i(z))$ to reach the desired extension $F(z) = \sum_i \llbracket z + N_i(z)\rrbracket$ with $N_i(z) \in (T_z\Mcal)^\perp$ and $z+N_i(z) \in \Sigma$ for each $z$.
\end{enumerate}
When combined with the verification of numerous technical details, this scheme indeed gives the result of Theorem~\ref{thm:Nlocalest}; see~\cite[Section~6.2]{DLS16centermfld} for a full proof.

\subsection{Properties of Center Manifolds}
Let us now recall some important properties of center manifolds from~\cite{DLS16centermfld} that will be useful in the succeeding sections. We omit the proofs here.

We begin with the following \emph{splitting before tilting} phenomenon, which tells us that whenever we stop our refining procedure on a cube $L \in \Wscr_e$, the lower excess bound passes through to the $\Mcal$-normal approximation $N$ and is comparable to the Dirichlet energy of $N$. 

\begin{proposition}[Splitting before tilting, \cite{DLS16centermfld}, Proposition~3.4]\label{prop:excesssplitting}
    There are functions $C_1 = C_1(\delta_2)$, $C_2 = C_2(M_0,\delta_2)$ such that if $M_0 \geq C_1$, $C_e \geq C_2$, if the hypotheses of Theorem~\ref{thm:Nlocalest} hold and if $\eps_2$ is chosen sufficiently small, then the following holds.
    
    Suppose that for some center manifold $\Mcal$ with corresponding rescaled current $T'$, we find a cube $L \in \Wscr_e$, $q \in \pi_0$ with $\dist(L,q) \leq 4\sqrt{m}\ell(L)$ and $\Omega \coloneqq \boldsymbol{\Phi}(B_{\ell(L)/4}(q,\pi_0))$. Then
    \begin{align*}
    &C_e\boldsymbol{m}_0\ell(L)^{m+2-2\delta_2} \leq \ell(L)^m\Ebf(T', \Bbf_L) \leq C \int_\Omega |DN|^2, \\
    &\int_\Lcal |DN|^2 \leq C\ell(L)^m \Ebf(T',\Bbf_L) \leq C_3 \ell(L)^{-2}\int_\Omega |N|^2.
    \end{align*}
    Here, the constants $C$ and $C_3$ are dependent on $\beta_2,\delta_2,M_0,N_0,C_e,C_h$.
\end{proposition}

We will further require the following persistence of $Q$-points result from a current $T'$ to the $\Mcal$-normal approximation when we stop refining due to large excess.

\begin{proposition}[Persistence of $Q$-points II]\label{prop:persistence}
    Assume the hypotheses of Proposition~\ref{prop:excesssplitting} hold. For every $\eta_2 > 0$, there are constants $\bar{s},\bar{\ell} > 0$, depending upon $\eta_2$, $\beta_2$, $\delta_2$, $M_0$, $N_0$, $C_e$ and $C_h$, such that, if $\eps_2$ is sufficiently small, then the following holds. If $L \in \Wscr_e$, $\ell(L) \leq \bar{\ell}$, $\Theta(T',p) = Q$ and $\dist(\mathbf{p}_{\pi_0}(p),L) \leq 4\sqrt{m}\ell(L)$, then
    \[
        (\bar{s}\ell(L))^{-m}\int_{\Bcal_{\bar{s}\ell(L)}(\mathbf{p}(p))} \Gcal\left(N,Q\llbracket \boldsymbol{\eta}\circ N \rrbracket\right)^2 \leq \frac{\eta_2}{\ell(L)^{m-2}} \int_{\Bcal_{\ell(L)}(\mathbf{p}(p))}|DN|^2.
    \]
\end{proposition}
Another vital consequence of the center manifold construction is the following result. It tells us that if we are able to immediately restart a new center manifold after stopping the previous one (because the excess has remained sufficiently small), then we can compare the two center manifolds. 
\begin{proposition}[Comparison of center manifolds, \cite{DLS16centermfld},~Proposition~3.7]\label{prop:comparecentermfld}
    There is a geometric constant $C_0$ and a constant $\bar{c}_s = \bar{c}_s(\beta_2,\delta_2,M_0,N_0,C_e,C_h) > 0$ with the following property. 
    
    Assume the hypotheses of Proposition~\ref{prop:excesssplitting}, $N_0 \geq C_0$ and $\eps_2$ is sufficiently small. If for some $\bar{r} \in (0,1)$,
    \begin{enumerate}[(a)]
        \item $\ell(L) \leq c_s\rho$ for every $\rho > \bar{r}$ and every $L \in \Wscr$ with $L \cap B_\rho(0, \pi_0) \neq \emptyset$; \\
        \item $\Ebf(T',\Bbf_{6\sqrt{m}\rho}) < \eps_2$ for every $\rho > \bar{r}$; \\
        \item there is a cube $L \in \Wscr$ with $\ell(L) \geq c_s\bar{r}$ and $L \cap \cl{B}_{\bar{r}}(0,\pi_0) \neq \emptyset$,
    \end{enumerate}
    then
    \begin{enumerate}[(i)]
        \item the current $T'' \coloneqq (T')_{0,\bar{r}}\mres\Bbf_{6\sqrt{m}}$ and the submanifold $\Sigma'' \coloneqq \iota_{0,\bar{r}}(\Sigma')\cap\Bbf_{7\sqrt{m}}$ satisfy all of the necessary assumptions for Proposition~\ref{prop:Whitney} for some new plane $\tilde{\pi}_0$ in place of $\pi_0$;
        \item for the center manifold $\widetilde{\Mcal}$ of $T''$ relative to $\tilde{\pi}_0$ and the $\widetilde{\Mcal}$-normal approximation $\widetilde{N}$, we have
        \begin{equation}
        \int_{\widetilde{\Mcal}\cap\Bbf_2} |\widetilde{N}|^2 \geq \bar{c}_s \max\{\Ebf(T'',\Bbf_{6\sqrt{m}}),\boldsymbol{c}(\Sigma'')^2\}.\label{eq:splitting}
        \end{equation}
    \end{enumerate}
\end{proposition}

Let us also recall the height bound~\cite[Theorem~A.1]{DLS16centermfld}, which tells us that for $T'$ satisfying all of the previous assumptions, as long as $E \coloneqq \Ebf(T',\Cbf_r(w,\pi)) < \eps_2$ for $\eps_2$ sufficiently small and some $m$-plane $\pi$, one can decompose $T'$ into $k \geq 1$ pairwise disjoint horizontal `strips' parallel to $\pi$ over $B_r(w,\pi)$, each one of height $rC_0 E^{\frac{1}{2m}}$. 
    
Furthermore, if $\Theta(T',w) \geq Q$ then we necessarily have $k=1$, and so this realizes as the following height bound for $T'$:
\[
    \mathbf{h}(T', \Cbf_r(w,\pi)) \leq C E^{\frac{1}{2m}}r^{1+\beta_2} \qquad \text{for any $r>0$ sufficiently small.}
\]
In fact, the improved height bound~\cite[Theorem~1.5]{Spolaor_15} tells us that around any point $w \in \spt T'$ with $\Theta(T',w) \geq Q$, there is the sharp height bound
\begin{equation}\label{eq:LucahtbdT}
\mathbf{h}(T', \Cbf_r(w,\pi)) \leq C E^{\frac{1}{2}}r + C \Abf r^2 \qquad \text{for any $r>0$ sufficiently small.}
\end{equation}
Since this result uses the Hardt-Simon inequality, note that the improved height bound only holds around points of density $Q$ (or higher). This allows us to improve the height bound in Theorem~\ref{thm:Nlocalest} to the following:
    \begin{lemma}\label{lem:Lucaheightbd}
        Suppose that $T'$ satisfies all of the prior assumptions with $\eps_2$ sufficiently small. Then for any choice of center manifold $\Mcal$ with corresponding rescaled current $T'$, $\Mcal$-normal approximation $N$ and any interval of flattening $(s,t]$ around center $\mathbf{p}(w)$ with $\Theta(T',w) \geq Q$, we have 
        \begin{equation}\label{eq:Lucaht}
        \int_{\Lcal} |N|^2 \leq C \ell(L)^{m+4-2\delta_2} \boldsymbol{m}_0,
        \end{equation}
        for any $L \in \Wscr$.
    \end{lemma}

    This Lemma will be used in the final contradiction argument within the proof of the lower frequency bound of Theorem~\ref{thm:freqlowerbd}.
    
    \begin{proof}
        The proof relies on passing the improved height bound~\eqref{eq:LucahtbdT} for the current $T'$ through to the graphical approximations. 
        
        By virtue of a scaled version of the estimate~\cite[Lemma~1.8]{Spolaor_15} that gives the improved height bound~\eqref{eq:LucahtbdT}, for $L \in \Wscr \cup \Sscr$ we have
        \[
        \mathbf{h}(T',\Cbf_{32r_L}(p_L,\pi_L)) \leq Cr_L\Ebf(T',\Cbf_{32r_L}(p_L,\pi_L))^{\frac{1}{2}} + Cr_L^2\Abf \leq C\boldsymbol{m}_0^{\frac{1}{2}}r_L^{2-\delta_2},
        \]
        as long as we choose $\eps_3$ sufficiently small. Via the oscillation bound~\eqref{eq:oscbd} in Theorem~\ref{thm:strongLip} for the graphical approximations $f_L$, this in turn gives
        \[
        \textrm{osc}(f_L) \leq C\boldsymbol{m}_0^{\frac{1}{2}}r_L^{2-\delta_2} \qquad \text{on $B_{8r_L}(x_L,\pi_L)$.}
        \]
        We may now substitute this improved bound into the computations for the estimates on the $\Mcal$-normal approximations in~\cite[Section~6]{DLS16centermfld} to conclude the desired improvement on the existing $\Crm^0$ and $\Lrm^2$ estimates on the $\Mcal$-normal approximations in Theorem~\ref{thm:Nlocalest}. We omit the details.
\end{proof}

\subsection{Intervals of Flattening}\label{sct:flattening}
We are now ready to take a sequence of center manifolds and corresponding $\Mcal$-normal approximations along our blow-up sequence of currents $T_{x_k,r_k}$. We begin by introducing \emph{intervals of flattening} around the blow-up centers $x_k$. These will contain the scales $r_k$, but will be chosen to detect the scales at which we need to replace an existing graphical approximation with a new, improved one. 

Given any fixed $k \in \N$, and $\eps_3 \in (0,\eps_2)$, set
\[
    \Rcal_k \coloneqq \set{r \in (0,1]}{\Ebf(T,\Bbf_{6\sqrt{m}r}(x_k))\leq \eps_3^2}.
\]
Notice that $\Rcal_k$ is closed under left-hand limits. Now construct a family $\Fcal_k = \{I^{(k)}_j\}_j$ of intervals $I^{(k)}_j = \big(s^{(k)}_j,t^{(k)}_j\big]$ inductively as follows. Let $t^{(k)}_0$ be the largest element of $\Rcal_k$. Given $t^{(k)}_j$, define
\[
T_{k,j} \coloneqq T_{x_k,t_j^{(k)}}\mres \Bbf_{6\sqrt{m}}, \quad \Sigma_{k,j} \coloneqq \iota_{x_k,t_j^{(k)}}(\Sigma)\cap\Bbf_{7\sqrt{m}}.
\]
Let $\pi_{k,j}$ be the $m$-plane satisfying
\[
    \Ebf(T_{k,j}, \Bbf_{6\sqrt{m}}, \pi_{k,j}) = \Ebf(T_{k,j},\Bbf_{6\sqrt{m}})
\]
for each $k, j$. Namely, $\pi_{k,j}$ is the plane $\pi_0$ for the rescaled current $T' = T_{k,j}$.

Let $\Mcal_{k,j}$ be the center manifold from Section~\ref{sct:centermfld} for $T' = T_{k,j}$ and $\Sigma = \Sigma_{k,j}$ with respect to the $m$-plane $\pi_{k,j}$. Let $N^{(k,j)}$ be the corresponding $\Mcal_{k,j}$-normal approximation. 

We denote by $\boldsymbol{\vphi}_{k,j}: \pi_{k,j} \supset [-4,4]^m  \to \pi_{k,j}^\perp$ the map from Proposition~\ref{prop:centermfld}, whose graph is $\Mcal_{k,j}$, meanwhile $\boldsymbol{\Phi}_{k,j}(z) \coloneqq (x,\boldsymbol{\vphi}_{k,j}(z))$ is the corresponding parameterization of $\Mcal_{k,j}$ over $\pi_{k,j}$. We use the usual notation $\mathbf{p}_{\pi_{k,j}}$ for the orthogonal projection to the plane $\pi_{k,j}$, meanwhile $\mathbf{p}_{k,j}$ (or simply $\mathbf{p}$ when there is no ambiguity) will be the projection map from Assumption~\ref{asm:projcm} to $\Mcal_{k,j}$. We will use the notation $\Bcal$ for the geodesic balls on our center manifolds.

Consider the Whitney decomposition $\Wscr^{(k,j)}$ of bad cubes from  Proposition~\ref{prop:Whitney} for $T_{k,j}$ and define
\[
    s^{(k)}_j \coloneqq t^{(k)}_j\max\Big(\set{c_s^{-1}\ell(L)}{L \in \Wscr^{(k,j)} \ \text{and $c_s^{-1}\ell(L) \geq \dist(0,L)$}}\cup\{0\}\Big),
\]
where $c_s \coloneqq \frac{1}{64\sqrt{m}}$.

Thus, for the center manifold $\Mcal_{k,j}$, we call $(s_j^{(k)}, t_j^{(k)}]$ the (centered) interval of flattening around $x_k$.

If $s^{(k)}_j = 0$, then stop refining. Otherwise, let $t^{(k)}_{j+1}$ be largest element of $\Rcal_k\cap\big(0,s^{(k)}_j\big]$, and continue as above. By construction, our stopping conditions are characterized as follows:
\begin{align}
    &\text{If $s^{(k)}_j > 0$ and $\bar{r}^{(k)}_j\coloneqq\frac{s^{(k)}_j}{t^{(k)}_j}$, then there is a cube $L \in \Wscr^{(k,j)}$ with} \tag{Stop}\label{eq:stop}\\
    &\notag
        \ell(L) = c_s\bar{r}^{(k)}_j \quad \text{and} \quad L \cap \cl{B}_{\bar{r}^{(k)}_j}(0,\pi_{k,j}) \neq \emptyset. \\
    &\text{If $\rho > \bar{r}^{(k)}_j = \frac{s^{(k)}_j}{t^{(k)}_j}$, then} \tag{Go}\label{eq:go}\\
    &\ell(L) < c_s\rho \quad \text{for every $L \in \Wscr^{(k,j)}$ with $L \cap B_\rho(0,\pi_{k,j}) \neq \emptyset$.}\notag
\end{align}

Moreover, letting
\[
    \boldsymbol{m}_0^{(k,j)} \coloneqq \max \{\boldsymbol{c}(\Sigma_{k,j})^2, \Ebf(T_{k,j}, \Bbf_{6\sqrt{m}})\},
\]
we recall the following additional vital properties of the intervals of flattening.
\begin{proposition}[\cite{DLS16blowup},~Proposition~2.2]\label{prop:flattening}
    Assuming $\eps_3 > 0$ is sufficiently small, the following holds for any centered interval of flattening on any center manifold $\Mcal_{k,j}$:
    \begin{enumerate}[(i)]
        \item\label{itm:flattening1} $s^{(k)}_j < \frac{t^{(k)}_j}{2^5}$ and each family $\Fcal_k$ is either countable with $t^{(k)}_j \todown 0$ as $j \to \infty$, or finite with $s^{(k)}_j = 0$ for the largest $j$;
        \item $\bigcup_j I^{(k)}_j = \Rcal_k$, and for each $k$ sufficiently large, we can find an interval $\big(s^{(k)}_{j(k)}, t^{(k)}_{j(k)}\big]$ around $x_k$ for which
        \[
            r_k \in \Big(\frac{s^{(k)}_j}{t^{(k)}_j}, 1\Big];
        \]
        \item if $r \in \big(\frac{s^{(k)}_{j}}{t^{(k)}_j}, 3\big)$ and $J \in \Wscr^{(k,j)}_n$ intersects $B \coloneqq \mathbf{p}_{\pi_{k,j}}\big(\Bcal_r(\boldsymbol{\Phi}_{k,j}(0))\big)$, then $J$ is in the domain of influence $\Wscr^{(k,j)}_n(H)$ (see~\cite[Definition~3.3]{DLS16centermfld}) of a cube $H \in \Wscr^{(k,j)}_e$, with
        \[
            \ell(H) \leq 3c_s r, \qquad \max\{\dist (H,B), \dist(H,J)\} \leq 3\sqrt{m}\ell(H) \leq \frac{3r}{16};
        \]
        \item $\Ebf(T_{k,j};\Bbf_r) \leq C_0 \eps_3^2 r^{2-2\delta_2}$ for every $r \in \big(\frac{s^{(k)}_j}{t^{(k)}_j}, 3\big)$;
        \item $\sup\set{\dist(x,\Mcal_{k,j})}{x \in\spt T_{k,j}\cap\mathbf{p}_{k,j}^{-1}\big(\Bcal_r(\boldsymbol{\Phi}_{k,j}(0))\big)} \leq C_0 [\boldsymbol{m}_0^{(k,j)}]^{\frac{1}{2m}}r^{1+\beta_2}$ for every $r \in \big(\frac{s^{(k)}_j}{t^{(k)}_j}, 3\big)$.
    \end{enumerate}
\end{proposition}

In particular,~\eqref{itm:flattening1} tells us that $I^{(k)}_j \neq \emptyset$ for each $k, j$. 
We will need analogous stopping criteria for scales around other points $w \neq x_k$. In particular, we will be interested in considering the points $w = y_k$, when we show persistence of $Q$-points. We hence introduce the following definition:
\begin{definition}[Non-centered intervals of flattening]\label{def:nonctrrefine}
    Let $\Mcal$ be a center manifold with a given centered interval of flattening $(s,t]$ and corresponding rescaled current $T'$. Let $w \in \spt T' \cap \big(\Bbf_{6\sqrt{m}}\setminus \{0\}\big)$. Then, let $\tilde{t} = \tilde{t}(w)$ be the largest element of 
    \[
        \set{r \in\big(0,\dist(w, \partial\Bbf_{t})\big]}{\Ebf(T', \Bbf_{6\sqrt{m}r}(w)) \leq \eps_3^2},
    \]
    and let 
    \[
        \tilde{s} = \tilde{s}(w) \coloneqq \tilde{t}\max\Big(\set{c_s^{-1}\ell(L)}{L \in \Wscr \ \text{and} \ c_s^{-1}\ell(L) \geq \dist(\mathbf{p}(w), L)}\cup\{0\}\Big),
    \] 
    we call $(\tilde{s},\tilde{t}]$ the (non-centered) interval of flattening around $w$ (corresponding to $(s,t]$).
\end{definition}
Observe that conditions~\eqref{eq:stop} and~\eqref{eq:go} still hold for any such non-centered interval of flattening, only for balls centered at $\mathbf{p}_{\pi_{k,j}}(w)$. Moreover, the appropriate analogues of the conclusions of Proposition~\ref{prop:flattening} hold, but we will not use this so we do not discuss this in detail here.

In many of the following arguments, we will be taking diagonal sequences of center manifolds and intervals of flattening around varying centers. Thus, it will be beneficial to simplify notation as follows.

Whenever we have a diagonal sequence of blow-ups $T_{k,j(k)}$ for $T$,  $\Mcal_{k, j(k)}$ is a diagonal sequence of center manifolds with corresponding $\Mcal_{k,j(k)}$-normal approximations $N^{(k,j(k))}$, centered flattening intervals $(s^{(k)}_{j(k)}, t^{(k)}_{j(k)}]$ and rescaled currents $T_{k,j(k)}$, we will use the notation
\[
	T_k \coloneqq T_{k,j(k)}, \qquad \Mcal_{k} \coloneqq \Mcal_{k,j(k)}, \qquad N^{(k)} \coloneqq N^{(k,j(k))}, \qquad \Sigma_k \coloneqq \Sigma_{k,j(k)},
\]
for the blow-ups of $T$ and the associated center manifolds, normal approximations and rescaled ambient manifolds. We will use the notation
\[
	s_k \coloneqq s^{(k)}_{j(k)}, \qquad  t_k \coloneqq t^{(k)}_{j(k)},
\]
for the endpoints of the interval of flattening corresponding to $\Mcal_k$, and we let
\[
	\boldsymbol{m}_0^{(k)} \coloneqq \boldsymbol{m}_0^{(k,j(k))}, \qquad \boldsymbol{\Phi}_k \coloneqq \boldsymbol{\Phi}_{k,j(k)}, \qquad \boldsymbol{\vphi}_k \coloneqq \boldsymbol{\vphi}_{k,j(k)}, \qquad \mathbf{p}_k \coloneqq \mathbf{p}_{k,j(k)}.
\]
We will also use the notation
\[
	\pi_k \coloneqq \pi_{k,j(k)}, \qquad \Wscr^{(k)} \coloneqq \Wscr^{(k,j(k))}, \qquad \Sscr^{(k)} \coloneqq \Sscr^{(k,j(k))},
\]
when discussing the Whitney decompositions for $\Mcal_k$. 
%
We are now in a position to introduce the frequency function and investigate its properties around $Q$-points of $T$.

\section{The Frequency Function}\label{sct:freq}
As previously mentioned, this section is dedicated to uniformly bounding the frequency function for our $\Mcal_{k,j}$-normal approximations $N^{(k,j)}$. Morally, we expect the frequency function to capture the `dominant frequency of oscillation' for each map $N^{(k,j)}$ at a prescribed scale around $\mathbf{p}_{k,j}(x_k)$.

Before continuing, let us recall the ultimate goal. We would like to take a sequence of center manifolds $\Mcal_{k}$ approximating our current $T$ at scales $r_k$ around $x_k$. We wish to show that the corresponding $\Mcal_k$-normal approximations $N^{(k)}$ (after normalizing appropriately) converge to the graph of a non-trivial Dir-minimizer $u$ in a sufficiently strong sense to ensure persistence of $Q$-points. This will in turn enable us to conclude the alternative~\eqref{eq:b} in Theorem~\ref{thm:Minkowskibd}. In order to conclude that the limit $u$ is non-trivial and that singularities persist, we will need to establish both upper and lower uniform bounds on the frequency function. 

We begin by introducing the frequency for Dir-minimizers on an open subset of $\R^m$; this, along with its key properties, will play an important role in the proof of Theorem~\ref{thm:freqlowerbd} and the concluding persistence of $Q$-points argument in Section~\ref{sct:Qpts}. 

Let $\Omega \subset \R^m$ be an open subset, and let $u : \Omega \to \Acal_Q$. For $r \in \left(0, \dist(y, \partial \Omega)\right)$, we define
\begin{equation}\label{eq:Dirminfreq}
D_{y, u}(r) \coloneqq \int_{B_r(y)} |Du|^2, \quad H_{y,u}(r) \coloneqq \int_{\partial B_r(y)} |u|^2, \quad I_{y,u}(r) \coloneqq \frac{r D_{y,u}(r)}{H_{y,u}(r)}.
\end{equation}
Note that these quantities are well-defined for Dir-minimizers; see~\cite[Remark~3.14]{DLS_MAMS}. Moreover, the frequency $I_{y,u}$ is monotone non-decreasing, so the limit 
\[
    I_0 \coloneqq \lim_{r \todown 0}I_{y,u}(r)
\]
is well-defined. We refer the reader to~\cite[Section~3.4]{DLS_MAMS} for further details on the frequency function for Dir-minimizers and its properties.

A particularly important property of Dir-minimizers is the way their $\Lrm^2$-height scales on balls centered around $Q$-points:
\begin{proposition}\label{lem:freqdecay}
    Suppose that $\Omega \subset \R^m$ open and that $u \in \Wrm^{1,2}(\Omega;\Acal_Q)$ is $\Dir$-minimizing. If for some $y \in \Omega$, $u(y) = Q\llbracket 0\rrbracket$ and $I_0 \coloneqq \lim_{r \todown 0}I_{y,u}(r)$ denotes the frequency of $u$ at $y$, then $I_0 > 0$ and for every $\rho < r < \dist(y,\partial\Omega)$, we have
    \begin{equation}\label{eq:massdecayfreq}
    \int_{B_\rho(y)} |u|^2 \leq \left(\frac{\rho}{r}\right)^{m+2 I_0}\int_{B_r(y)} |u|^2.
    \end{equation}
    
    Conversely, if for some $y \in \Omega$ and some $c > 0$ one of the following alternatives holds:
    \begin{enumerate}[(a)]
        \item\label{eq:massdecay} $$\int_{B_\rho(y)} |u|^2 \leq C\rho^{m+2c} \qquad \text{for every $\rho < \dist(y,\partial\Omega)$ sufficiently small};$$ \\
        \item\label{eq:freqlb} $I_0 \geq c$, 
    \end{enumerate}
    then $u(y) = Q\llbracket 0\rrbracket$.
\end{proposition}

\begin{remark}
    In fact, the alternatives~\eqref{eq:massdecay} and~\eqref{eq:freqlb} are equivalent; this can easily be checked via a contradiction argument.
\end{remark}
The proof of this can be inferred from~\cite[Corollary~3.18]{DLS_MAMS}, but we repeat it at the end of this section for convenience.

\begin{definition}[Almgren's frequency function]\label{def:freq}
    Let $\Mcal$ be a center manifold with corresponding $\Mcal$-normal approximation $N$ and rescaled current $T'$, let $w \in \spt T'$, and let $(s,t]$ be an interval of flattening around $w$. 
    
    Let $r \in \Big(\frac{s}{t},R\Big]$ be an arbitrary scale. Then for a Lipschitz cutoff $\phi: [0, \infty) \to [0,1]$ that is identically $1$ on $[0,\frac{1}{2}]$, and vanishes on $[1,\infty)$.we define
    \begin{align*}
        \Dbf_{\mathbf{p}(w),N}(r) &\coloneqq \int_{\Mcal} \phi\left(\frac{d(p-\mathbf{p}(w))}{r}\right) |DN|^2(p)\dd p, \\
        \Hbf_{\mathbf{p}(w),N}(r) &\coloneqq -\int_{\Mcal} \phi ' \left(\frac{d(p-\mathbf{p}(w))}{r}\right) \frac{|N|^2(p)}{d(p-\mathbf{p}(w))} \dd p, \\
        \Ibf_{\mathbf{p}(w),N}(r) &\coloneqq \frac{r\Dbf_{\mathbf{p}(w),N}(r)}{\Hbf_{\mathbf{p}(w),N}(r)},
    \end{align*}
    where $d(q)$ denotes the geodesic distance between $q$ and $\boldsymbol{\Phi}(0)$ on $\Mcal$. Let 
    \[  
        \boldsymbol{\Omega}_{\mathbf{p}(w),N} = \max\{\log\Ibf_{\mathbf{p}(w),N},\log c_0\},
    \]
    for $c_0$ to be determined later. Observe that the above functions are regularized versions of the analogous functions defined for $Q$-valued harmonic maps (see~\eqref{eq:Dirminfreq}), and so it is easy to check that they are absolutely continuous.
    
    When it is clear from context, we will omit $N$ and/or $w$ from the notation for the above quantities. Moreover, whenever we take center manifolds $\Mcal_{k,j}$ at varying scales $[s^{(k)}_j, t^{(k)}_j]$ with varying centers $x_k$, we will let
    \[
        \Dbf_{k,j} \coloneqq \Dbf_{\mathbf{p}_{k,j}(x_k), N^{(k,j)}}, \quad \Hbf_{k,j} \coloneqq \Hbf_{\mathbf{p}_{k,j}(x_k), N^{(k,j)}}, \quad \Ibf_{k,j}(r) \coloneqq \frac{r\Dbf_{k,j}(r)}{\Hbf_{k,j}(r)}.
    \]
    For diagonal blow-up sequences as discussed in Section~\ref{sct:flattening} with varying centers $\mathbf{p}_k(w_k) \coloneqq \mathbf{p}_{k,j(k)}(w_k)$ (where $w_k$ are not necessarily $x_k$) we will let
    \begin{equation}\label{eq:diagonalfreq}
        \Dbf_k \coloneqq \Dbf_{\mathbf{p}_k(w_k), N^{(k,j(k))}}, \quad \Hbf_k \coloneqq \Hbf_{\mathbf{p}_k(w_k), N^{(k,j(k))}}, \quad \Ibf_k(r) \coloneqq \frac{r\Dbf_k(r)}{\Hbf_k(r)}.
    \end{equation}
    We use the analogous notation for $\boldsymbol{\Omega}$ also, when necessary.
\end{definition}

Before we continue, we henceforth fix the following useful notation. Given an arbitrary center manifold $\Mcal = \Mcal_{k,j}$ with a corresponding rescaled current $T' = T_{x_k, t_j^{(k)}}$ and a point $w \in \spt T'$, we define
\begin{equation}\label{eq:radius}
    R = R(w) \coloneqq \dist(\mathbf{p}(w), \partial B_3(x_k)).
\end{equation}
This regularized frequency function exhibits a convenient ~\emph{almost monotonicity} property, which can be compared with the monotonicity \cite[Theorem~3.15]{DLS_MAMS} of the frequency for Dir-minimizers:
\begin{theorem}\label{thm:monfreq}
    There exists a bounded function $f : [0,\infty) \to \R$ with $\lim_{r\todown 0} f(r) = 0$ such that the following holds.
    
    Suppose that we have a center manifold $\Mcal = \Mcal_{k,j}$ with a corresponding rescaled current $T' = T_{x_k,t^{(k)}_j}$, $\Mcal$-normal approximation $N = N^{(k,j)}$ and an interval of flattening $(s,t]$ around any center $\mathbf{p}(w)$. Then
    \[
        \boldsymbol{\Omega}_{\mathbf{p}(w),N}(a) \leq \boldsymbol{\Omega}_{\mathbf{p}(w),N}(b) + f(b) \qquad \text{for every $[a,b] \subset \Big[\frac{s}{t},R\Big]$}.
    \]
\end{theorem}

\begin{remark}
    Note that the function $f$ is \emph{independent} of the choice of center manifold, choice of point around which we take the interval, and the interval itself. Moreover, observe that this almost-monotonicity estimate tells us that
    \begin{equation}\label{eq:monfreq}
    \Ibf_{\mathbf{p}(w),N}(a) \leq e^{f(b)}\max\{c_0,\Ibf_{\mathbf{p}(w),N}(b)\}.
    \end{equation}
\end{remark}
The proof of this theorem is the same as that of \cite[Theorem~3.2]{DLS16blowup}, with merely a sharper final estimate. Nevertheless, we reproduce the argument at the end of this section. The proof is based on calculating variation estimates; the main idea is to compare the frequency with that of homogeneous maps which share boundary values as $N^{(k,j)}$, since the frequency of homogeneous maps is constant and equal to their degree of homogeneity. This is analogous to the proof of the monotonicity of the mass ratios for a current $T$ (around a fixed center $p$ with $\Theta(T,p) = Q$), which comes from comparing $T$ with $m$-dimensional cones $T\mres \partial B_r(p)\RangeX Q\llbracket p \rrbracket$, which have constant mass ratios.
 
The almost-monotonicity of the frequency will be crucial for establishing uniform bounds on the frequency function.

As explained above, we would firstly like to show uniform upper bounds on the frequencies $\Ibf_{k,j}$, independently of the center:
\begin{theorem}\label{thm:frequpperbd}
    There exists $J \in \Nbb$ such that the following holds. 
    For the center manifolds $\Mcal_{k,j}$ with corresponding normal approximations $N^{(k,j)}$, rescaled current $T_{k,j} = T_{x_k,t_j^{(k)}}$ and centered flattening intervals $\big(s^{(k)}_j, t^{(k)}_j\big]$ around $x_k$, we have
    \begin{equation}\label{eq:smallfrequb}
        \sup_{k,j}\sup_{r \in \big(\frac{s^{(k)}_j}{t^{(k)}_j},\frac{3}{2^{J-1}}\big]}\Ibf_{k,j}(r) < +\infty,
    \end{equation}
    and
    \begin{equation}\label{eq:largefrequb}
        \sup_{k,j}\sup_{r \in (\frac{3}{2^{J-1}}, 3]} \min\left\{\Ibf_{k,j}(r), \frac{r^2 \int_{\Bcal_r} |DN^{(k,j)}|^2}{\int_{\Bcal_r} |N^{(k,j)}|^2}\right\} < +\infty.
    \end{equation}
    In particular, given any point $w_{k} \in \spt T_{k,j} \cap \cl{\Bbf}_1$ for each center manifold $\Mcal_{k,j}$ and corresponding non-centered interval of flattening $\big[\tilde{s}^{(k)}_j,\tilde{t}^{(k)}_j\big)$ around $w_{k}$, we also have
    \[
        \sup_{k,j}\sup_{r \in \big(\frac{\bar{s}^{(k)}_j}{\bar{t}^{(k)}_j},\delta_{k,j}\big]} \Ibf_{k,j}(r) < \infty
    \]
    and
    \[
        \sup_{k,j}\sup_{r \in (\delta_{k,j}, R_{k,j}]} \min\left\{\Ibf_{k,j}(r), \frac{r^2 \int_{\Bcal_r} |DN^{(k,j)}|^2}{\int_{\Bcal_r} |N^{(k,j)}|^2}\right\} < +\infty,
    \]
    where 
    \[
        \delta_{k,j} \coloneqq \dist\Big(\mathbf{p}_{k,j}(w_{k}), \partial \Bcal_{\frac{3}{2^{J-1}}}(\mathbf{p}_{k,j}(x_k))\Big),
    \]
    and $R_{k,j}$ is as in~\eqref{eq:radius} for $\mathbf{p}_{k,j}(w_k)$.
\end{theorem} 
For the version of this result for a sequence of center manifolds with a fixed center, see~\cite[Theorem~5.1]{DLS16blowup}. 

Theorem~\ref{thm:frequpperbd} will allow us to disregard the possibility that as we blow up, the mass of $T$ accumulates around $\partial \Bbf_{R_{k,j}t_j^{(k)}}(w_{k})$, creating a locally trivial blow-up limit. 
Indeed, this kind of phenomenon could only occur if the frequency were unbounded at the scales $t_j^{(k)}$, due to `infinite order of collapsing' at these scales.

Note that for the larger scales $r \in \big(\frac{3}{2^{J-1}}, 3\big]$ around $x_k$ on our center manifolds, we are satisfied with a cruder estimate of the form~\eqref{eq:largefrequb}. This is because it suffices to have a uniform frequency bound up to some uniform scale independent of $k$ and $j$, which in itself gives us the desired conclusion.

We will be using this upper frequency bound in the form of the following \emph{reverse Sobolev inequality} (see~\cite[Corollary~5.3]{DLS16blowup}), which will be a key tool for our concluding argument.
\begin{corollary}\label{cor:reverseSob}
    There exists a constant $C = C(T)$ such that the following holds. Suppose that $\Mcal = \Mcal_{k,j}$ is any center manifold with corresponding rescaled current $T' = T_{k,j}$ and $\Mcal$-normal approximation $N = N^{(k,j)}$. Then for any point $w \in \spt T'$ with $\Theta(T', w) \geq Q$ and any interval of flattening $[s,t)$ around $\mathbf{p}(w)$, for every $r \in \big(\frac{s}{t},1\big]$, there is $\bar{s} \in \big(\frac{R}{2}r, Rr\big]$ such that
    \[
        \int_{\Bcal_{\bar{s}}(\mathbf{p}(w))} |DN|^2 \leq \frac{C}{r^2} \int_{\Bcal_{\bar{s}}(\mathbf{p}(w))} |N|^2,
    \]
\end{corollary}
In the case where $w=x_k$ for the center manifold $\Mcal_{k,j}$, the proof is in fact \emph{exactly} the same as that for a fixed center in~\cite[Corollary~5.3]{DLS16blowup}. This is because we recenter our planes $\pi_{k,j}$ in such a way that $\mathbf{p}_{k,j}(x_k) = 0$. For non-centered intervals of flattening, the argument is still analogous, since we are once again making use of the upper frequency bound of Theorem~\ref{thm:frequpperbd}, which is independent of the center. We omit the proof, since it is merely a simple application of the coarea formula, Fubini's Theorem and some basic calculus.

We also have the following uniform lower frequency bound for all scales sufficiently close to where we stop the center manifolds. This is to be expected, since at the stopping scales, we inherit the behaviour of $T$ around the $Q$-point. Namely, there is a height bound on $N$ coming from concentration of mass of $T$ around the $Q$-point, which should be inconsistent with the splitting that arises if we did not have a uniform lower frequency bound.
\begin{theorem}\label{thm:freqlowerbd}
    It is possible to choose $c_0$ appropriately to ensure that the following holds.
    
    One has the following lower bound on the frequency for any center manifold $\Mcal$ with corresponding $\Mcal$-normal approximation $N$, rescaled current $T'$, and any point $w \in \spt T'$ with $\Theta(T', w) \geq Q$ that is the center of the interval of flattening $(s,t]$:
    \[
        \boldsymbol{\Omega}_{\mathbf{p}(w),N}\Big(\frac{r}{t}\Big) > \log (c_0),
    \]
    for every $\frac{r}{t} > 0$ sufficiently small in $\left(\frac{s}{t},R\right]$.
\end{theorem}
\begin{remark}
    Note that $w$ needn't be the center of the center manifold construction. This will be important, since we will later choose $w$ to be $\mathbf{p}_{k,j(k)}(y_k)$ in order to show that the $Q$-points persist in the blow-up limit.
\end{remark}


\begin{remark}
        Observe that the claim of Theorem~\ref{thm:Minkowskibd} is stronger than the blow-up argument used in~\cite{DLS16blowup} and~\cite{Almgren_regularity} to show the inheritence of singularities of the limiting map. There, for a fixed blow-up center $x$, the contradiction relies on a delicate capacitary argument. More precisely, after obtaining a non-trivial limiting Dir-minimizer $u$ as above, the authors there conclude that Hausdorff limit $K^\infty$ of the sets $\Sing_Q(T_k)$ for $T_k \coloneqq T_{x,r_k}\mres \Bbf_{6\sqrt{m}}$ satisfies
        \[
        \Hcal^{m-2+\alpha}\left(K^\infty\setminus\{\text{points sufficiently close to $\Sing_Q u$}\}\right) \geq \eta.
        \]
        This may then be passed through to $T_k$ via upper semi-continuity:
        \begin{equation}\label{eq:usc}
        \Hcal_\infty^{m-2+\alpha}\left(\Sing_Q(T_k)\setminus\{\text{points close to $\Sing_Q u$}\}\right) \geq \eta.
        \end{equation}
        However, for any point $p \in \Sing_Q(T_k)$, there exists $\rho = \rho(p)$ arbitrarily small and $x_p \in \Mcal_k$ near $p$ such that:
        \[
        c(p,\alpha)\|N^{(k)}\|_{\Lrm^2(\Bcal_1(x_p))}\rho^{m-2+\alpha} \leq \int_{\Bcal_{\rho}(x_p)}|DN^{(k)}|^2.
        \]
        In other words, there is a quantitative control on the frequency from below, coming from the persistence of singularities from $T_k$ to $N^{(k)}$. One can now use this estimate to apply a covering argument to $\Sing_Q(T_k)$, contradicting~\eqref{eq:usc}.

        Unfortunately, such an argument is insufficient in our case, since we require a dimension estimate for all of $K^\infty$, rather than just those points that are far from $\Sing_Q u$. The easiest way to demonstrate the desired dimension estimate for $K^\infty$ is hence by showing the set containment claimed in Theorem~\ref{thm:Minkowskibd}, for which the lower frequency bound of Theorem~\ref{thm:freqlowerbd} is crucial. It will allow us to get an improvement on previous persistence of singularities arguments.
\end{remark}

Before we proceed to prove the above results, let us show the following local estimate for the Dirichlet energy of the $\Mcal$-normal approxiamtions at scales within the flattening intervals, which will come in useful later.
\begin{proposition}\label{prop:Dirbd}
    Suppose that $\eps_3$ is sufficiently small. Then for any choice of center manifold $\Mcal$ with corresponding rescaled current $T'$, $\Mcal$-normal approximation $N$ and any interval of flattening $(s,t]$ around $w \in \spt T'$, the following estimate holds for every $r \in \Big(\frac{s}{t}, R\Big]$:
    \[
        \Dbf_{\mathbf{p}(w)}(r) \leq \int_{\Bcal_r(\mathbf{p}(w))} |DN|^2 \leq C\boldsymbol{m}_0 r^{m+2-2\delta_2},
    \]
    where $C=C(\beta_2,\delta_2,M_0,N_0,C_e,C_h)$ is independent of the choice of center manifold, flattening interval, and center $\mathbf{p}(w)$.
\end{proposition}
\begin{proof}
    We use the same notation as in Section~\ref{sct:Almgrencm}. Choose $\eps_3$ sufficiently small such that for each $r \in \big(\frac{s}{t}, R\big]$, the geodesic ball $\Bcal_r(\mathbf{p}(w)))$ is contained in $\boldsymbol{\Phi}(\cl{B}_{\frac{R}{2}r}(\mathbf{p}_{\pi_0}(w)))$. Fix a radius $r$ in this range. The first inequality is trivial due to the nature of the map $\phi$. We can cover $\cl{B}_{\frac{R}{2}r}(\mathbf{p}_{\pi_0}(w))$ by $\boldsymbol{\Gamma}$ and the concentric cubes $\frac{17}{16}L$ of some subcollection $\Fscr$ of cubes $L \in \Wscr$. Note that since $\cl{B}_{\frac{R}{2}r}(\mathbf{p}_{\pi_0}(w))\setminus \boldsymbol{\Gamma}$ is relatively compact in $[-4,4]^m$, this subcollection must be finite (and its cardinality can be controlled independently of $r$).

    Now by construction, $N \equiv 0$ on $\boldsymbol{\Gamma}\cap\cl{B}_{\frac{R}{2}r}(\mathbf{p}_{\pi_0}(w))$. Moreover, for each $L \in \Wscr$ with $L \cap \cl{B}_{\frac{R}{2}r}(\mathbf{p}_{\pi_0}(w)) \neq \emptyset$, the condition~\eqref{eq:go} tells us that $\ell(L) < \frac{R}{2}c_s r$. Hence, by the estimates on the $\Mcal$-normal approximation established in Theorem~\ref{thm:Nlocalest}, we deduce that
    \[
        \int_{\boldsymbol{\Phi}(\cl{B}_{Rr/2}(\mathbf{p}_{\pi_0}(w)))} |DN|^2 \leq \sum_{L \in \Fscr} \int_{\Lcal} |DN|^2 \leq c\sum_{L \in \Fscr}\boldsymbol{m}_0\ell(L)^{m+2-2\delta_2} \leq c\boldsymbol{m}_0 r^{m+2-2\delta_2}.
    \]
    The result follows.
\end{proof}
We now prove the characterization of $Q$-points in Proposition~\ref{lem:freqdecay}.

\begin{proof}[Proof of Proposition~\ref{lem:freqdecay}]
    We will henceforth omit the dependency of the quantities in~\eqref{eq:Dirminfreq} on $u$ and $y$, for simplicity. Moreover, we will assume without loss of generality that $y = 0$. We begin by showing that as long as $H(r) > 0$,
    \begin{equation}\label{eq:massder}
    \left.\frac{\dd}{\dd \tau} \right|_{\tau = r} \log\left(\frac{H(\tau)}{\tau^{m-1}}\right) = \frac{2I(r)}{r}.
    \end{equation}
    Firstly, note that $H \in \Wrm^{1,1}$ since $u \in \Wrm^{1,2}$. Moreover, the distributional derivative of $|u|^2$ coincides with its approximate differential, so the chain rule~\cite[Proposition~2.8]{DLS_MAMS} applies and thus one can compute $H'$. These facts justify the elementary computation (at almost-every $\tau > 0$) that
    \[
    \frac{\dd}{\dd \tau} \log\left(\frac{H(\tau)}{\tau^{m-1}}\right) =  \frac{H'(\tau)}{H(\tau)} - \frac{m-1}{\tau},
    \]
    with
    \begin{align*}
    H'(\tau) &= \frac{\dd}{\dd \tau} \int_{B_\tau\setminus B_{\tau/2}} \frac{|u(z)|^2}{|z|}\dd z \\
    &= \frac{\dd}{\dd \tau} \int_{B_1 \setminus B_{1/2}} \tau^{m-1}\frac{|u(\tau z)|^2}{|z|} \dd z \\
    &= 2\tau^{m-1}\int_{\frac{1}{2}}^1 s^{-1}\int_{\partial B_s} \sum_{i=1}^Q \dpr{u_i(\tau z), \nabla_z u_i(\tau z)} \dd\Hcal^{m-1}(z) \dd s +\frac{(m-1)}{\tau} H(\tau) \\
    &= 2\tau^{m} \int_{\frac{1}{2}}^1 \int_{\partial B_s} \sum_{i=1}^Q \dpr{u_i(\tau z), \partial_\nu u_i(\tau z)} \dd\Hcal^{m-1}(z) \dd s +\frac{(m-1)}{\tau} H(\tau) \\
    &= 2\int_{B_\tau \setminus B_{\tau/2}} \sum_{i=1}^Q \dpr{u_i, \partial_\nu u_i} +\frac{(m-1)}{\tau} H(\tau) \\
    &= 2D(\tau) +\frac{m-1}{\tau} H(\tau).
    \end{align*}
    From this differential identity, we infer that
    \[
    \frac{\dd}{\dd \tau} \log\left(\frac{H(\tau)}{\tau^{m-1}}\right) = 2 \frac{I(\tau)}{\tau}.
    \]
    By continuity, \eqref{eq:massder} follows when we evaluate $\tau$ at some $r > 0$.
    
    Integrating \eqref{eq:massder} from $\rho$ to $r$ and using the fact that $I$ is monotone increasing, we have
    \begin{equation}\label{eq:freqode}
    \log \left(\frac{H(r)}{r^{m-1}}\right) - \log \left(\frac{H(\rho)}{\rho^{m-1}}\right) = 2\int_\rho^r \frac{I(\tau)}{\tau} \dd \tau \geq 2I_0\log\left(\frac{r}{\rho}\right),
    \end{equation}
    or equivalently,
    \begin{align*}
    \frac{1}{\rho^{m-1+2I_0}} \int_{B_\rho \setminus B_{\rho/2}} \frac{|u(z)|^2}{|z|} \dd z &= \frac{H(\rho)}{\rho^{m-1 + 2I_0}} \\
    &\leq \frac{H(r)}{r^{m-1 + 2I_0}} \\
    &= \frac{1}{r^{m-1+2I_0}} \int_{B_r \setminus B_{r/2}} \frac{|u(z)|^2}{|z|} \dd z.
    \end{align*}
    More generally, for any $k \in \Nbb$, this rescales as
    \[
    \frac{2^{k(m-1+2I_0)}}{\rho^{m-1+2I_0}} \int_{B_{\rho/2^k} \setminus B_{\rho/2^{k+1}}} \frac{|u(z)|^2}{|z|} \dd z \leq \frac{2^{k(m-1+2I_0)}}{r^{m-1+2I_0}} \int_{B_{r/2^k} \setminus B_{r/2^{k+1}}} \frac{|u(z)|^2}{|z|} \dd z,
    \]
    and thus
    \[
    \frac{1}{\rho^{m+2I_0}} \int_{B_{\rho/2^k} \setminus B_{\rho/2^{k+1}}} |u(z)|^2 \dd z \leq \frac{1}{r^{m+2I_0}} \int_{B_{r/2^k} \setminus B_{r/2^{k+1}}} |u(z)|^2 \dd z.
    \]
    Summing over $k$, this yields
    \[
    \frac{1}{\rho^{m+2I_0}} \int_{B_\rho} |u(z)|^2 \dd z \leq \frac{1}{r^{m+2I_0}} \int_{B_r} |u(z)|^2 \dd z,
    \]
    as required. It remains to show that $I_0 > 0$, but this follows easily from the inequality
    \[
    H(r) \leq Cr D(r) \qquad \text{for every $r \in \Big(0,\frac{1}{2}\dist(0,\partial\Omega)\Big)$},
    \]
    which comes from the H\"{o}lder-regularity~\cite[Theorem~3.9]{DLS_MAMS} of $Q$-valued $\Wrm^{1,2}$-maps and the fact that $u(0) = Q\llbracket 0 \rrbracket$.
    
    Now we show the converse. Suppose that $I_0 \geq c$ for some $c > 0$, but that $u(0) \neq Q\llbracket 0 \rrbracket$. Then, $u(0) = T \in \Acal_Q$ for some $T$ with $d(T) > 0$, where $d(T)$ is defined as in~\cite[Definition~3.4]{DLS_MAMS}. However, the lower frequency bound allows us to conclude that the assumptions of~\cite[Proposition~3.6]{DLS_MAMS} hold for $\Omega = B_r$, for $r > 0$ sufficiently small. This would enable us to decompose $u$ locally around $0$ into two simpler functions. One of these must necessarily take a value other than zero at the origin, thus contradicting the height decay which comes as a consequence of the lower frequency bound. Indeed,~\eqref{eq:freqode} and the succeeding computations tell us that we must have
    \[
        \frac{1}{\rho^{m+2c}}\int_{B_\rho} |u|^2 \leq \frac{1}{r^{m+2c}}\int_{B_r} |u|^2 \qquad \text{for any $\rho < r < \frac{1}{2}\dist(0,\partial\Omega)$.}
    \]
    
    The argument under assumption~\eqref{eq:massdecay} follows analogously.
\end{proof}

We shall now prove the almost-monotonicity in Theorem~\ref{thm:monfreq}. We must first introduce following quantities, centered at any point $w \in \spt T'$, with corresponding $\Mcal$-normal approximation $N$:
\begin{align*}
\Ebf_{\mathbf{p}(w), N}(r) &\coloneqq -\int_{\Mcal} \phi ' \left(\frac{d(p-\mathbf{p}(w))}{r}\right) \sum_{i=1}^Q \dpr{N_i(p),\partial_r N_i(p)} \dd p, \\
\Gbf_{\mathbf{p}(w), N}(r) &\coloneqq -\int_{\Mcal} \phi ' \left(\frac{d(p-\mathbf{p}(w))}{r}\right) |\partial_r N|^2(p) \ d(p-\mathbf{p}(w)) \dd p, \\
\boldsymbol{\Sigma}_{\mathbf{p}(w), N}(r) &\coloneqq \int_{\Mcal} \phi\left(\frac{d(p-\mathbf{p}(w))}{r}\right)|N|^2(p) \dd p,
\end{align*}
Note that for a given point $w$, the quantities $\Ebf_{\mathbf{p}(w), N}$, $\Gbf_{\mathbf{p}(w), N}$ and $\boldsymbol{\Sigma}_{\mathbf{p}(w), N}$ are all classically differentiable. By considering the inner and outer variations of the current $\Tbf_F$ along convenient choices of vector fields, the following estimates arise:
\begin{proposition}\label{prop:firstvarest}
    For every $\gamma_3$ sufficiently small, there exists $C = C(\gamma_3) > 0$ such that if $\eps_3 \in (0, \eps_2)$ sufficiently small, the following holds.
    
    Let $\Mcal$ be a center manifold with corresponding rescaled current $T'$, $\Mcal$-normal approximation $N$, interval of flattening $[s,t]$ centered at $x = \mathbf{p}(w)$ and any $[a,b] \subset \big[\frac{s}{t},R\big]$ with $\Ibf \geq c_0$ on $[a,b]$. Omitting the dependency on $N$ and $x$ for $\Dbf_{x,N}, \Hbf_{x,N}$ and all related quantities for simplicity, we have the following for any $r \in [a,b]$:
    \begin{align}
        &\Big|\Hbf'(r) - \frac{m-1}{r}\Hbf_{x}(r) -\frac{2}{r}\Ebf (r)\Big| \leq C\Hbf (r),\label{eq:firstvar1} \\
        &\Big|\Dbf (r) - \frac{1}{r}\Ebf (r)\Big| \leq C\Dbf (r)^{1+\gamma_3} + C\eps_3^2\boldsymbol{\Sigma} (r),\label{eq:firstvar2} \\
        &\Big|\Dbf' (r) - \frac{m-2}{r}\Dbf (r) - \frac{2}{r^2}\Gbf (r)\Big| \label{eq:firstvar3} \\
        &\qquad\leq C\big(\Dbf (r) +\Dbf (r)^{\gamma_3}\Dbf' (r) +\frac{1}{r}\Dbf (r)^{1+\gamma_3}\big), \notag\\
        &\boldsymbol{\Sigma} (r) + r\boldsymbol{\Sigma}' (r) \leq Cr^2\Dbf (r) \leq C r^{m+2}\eps_3^2.\label{eq:firstvar4}
    \end{align}
\end{proposition}   
We refer the reader to~\cite{DLS16blowup} for the proof. We are now ready to prove the almost-monotonicity.

\begin{proof}[Proof of Theorem~\ref{thm:monfreq}]   
    First of all, notice that since $\boldsymbol{\Sigma} ' \geq 0$ (due to the nature of $\phi$), \eqref{eq:firstvar4} tells us that also $\boldsymbol{\Sigma} (r), \ r\boldsymbol{\Sigma}' (r) \leq Cr^2\Dbf (r)$.
    
    Note that we may assume that $\Ibf  > c_0$ on $[a,b]$, since $\Ibf (a) > c_0$ (otherwise there is nothing to prove) and if $\Ibf (r) \leq c_0$ at some point $b' \in (a,b)$, we can replace $b$ by $b'$. Fix $\gamma_3$ and $\eps_3$ small enough so that the above estimates hold. By \eqref{eq:firstvar2}, \eqref{eq:firstvar4} and the fact that $r \leq 3$, we have
    \begin{align*}
    \Dbf (r) &\leq \frac{\Ebf (r)}{r} + C\Dbf (r)^{1+\gamma_3} + \eps_3^2\boldsymbol{\Sigma} (r) \\
    &\leq \frac{\Ebf (r)}{r} + C\eps_3^2 \Dbf (r) \left[r^{\gamma_3 m} + r^2\right] \\
    &\leq \frac{\Ebf (r)}{r} + C\eps_3^2 \Dbf (r).
    \end{align*}
    Rearranging and decreasing $\eps_3$ further if necessary, we conclude that
    \[
    	\frac{\Dbf (r)}{2} \leq \frac{\Ebf (r)}{r}.
    \]
    An analogous argument yields that
    \[
    	\frac{\Dbf (r)}{r} \leq 2\Dbf (r),
    \]
    so $\Ebf > 0$ on $(a,b)$ and $\frac{1}{\Ebf (r)}$ is thus a well-defined quantity. Now let
    \[
    	\Fbf (r) \coloneqq \frac{1}{\Dbf (r)} - \frac{r}{\Ebf (r)}.
    \]
    Let us now estimate the derivative of $\boldsymbol{\Omega} $, in the hope that we can then integrate to obtain the almost-monotonicity. Observe that
    \begin{align*}
    -\boldsymbol{\Omega} '(r) &= -\frac{\Ibf '(r)}{\Ibf (r)} \\
    &= -\frac{\Hbf (r)}{r\Dbf (r)}\left[\frac{\Dbf (r)}{\Hbf (r)} + \frac{r\Dbf' (r)}{\Hbf (r)} - \frac{r\Dbf (r)\Hbf' (r)}{\Hbf (r)^2}\right] \\
    &= -\frac{1}{r} -\frac{\Dbf' (r)}{\Dbf (r)} + \frac{\Hbf' (r)}{\Hbf (r)} \\
    &= -\frac{1}{r} + \frac{\Hbf' (r)}{\Hbf (r)} - \Dbf' (r)\Fbf (r) - \frac{r\Dbf' (r)}{\Ebf (r)}.
    \end{align*}
    Now let us further estimate each of the terms on the right-hand side. By \eqref{eq:firstvar1}, we have
    \[
        \frac{\Hbf' (r)}{\Hbf (r)} \leq \frac{m-1}{r} +\frac{2\Ebf (r)}{r\Hbf (r)} + C.
    \]
    Furthermore, by~\eqref{eq:firstvar1} and the comparability of $\Dbf (r)$ and $\frac{\Ebf (r)}{r}$, we can achieve the bound
    \begin{align*}
        |\Fbf (r)| &= \left|\frac{r}{\Dbf (r)\Ebf (r)}\left(\Dbf (r) - \frac{1}{r}\Ebf (r)\right)\right| \\
        &\leq C\frac{r\left(\Dbf (r)^{1+\gamma_3} + \boldsymbol{\Sigma} (r)\right)}{\Dbf (r)\Ebf (r)} \\
        &\leq C\left[\Dbf (r)^{\gamma_3 - 1} + \frac{\boldsymbol{\Sigma} (r)}{\Dbf (r)^2}\right].
    \end{align*}
    Combining this estimate with~\eqref{eq:firstvar3}, \eqref{eq:firstvar4}, and again by comparing $\frac{\Ebf (r)}{r}$ to $\Dbf (r)$, we conclude that
    \begin{align*}
    -\frac{r\Dbf' (r)}{\Ebf (r)} &\leq \frac{r\Dbf (r)}{\Ebf (r)}\left(C - \frac{m-2}{r}\right) - \frac{2\Gbf (r)}{r\Ebf (r)} + C\frac{r\Dbf (r)^{\gamma_3}\Dbf' (r) +\Dbf (r)^{1+\gamma_3}}{\Ebf (r)} \\
    &\leq C - \frac{m-2}{r} + C\frac{\Dbf (r)|\Fbf(r)|}{r} - \frac{2\Gbf (r)}{r\Ebf (r)} \\
    &\qquad+ C\left[ \Dbf (r)^{\gamma_3 - 1}\Dbf' (r) + \frac{\Dbf (r)^{\gamma_3}}{r}\right] \\
    &\leq - \frac{m-2}{r} - \frac{2\Gbf (r)}{r\Ebf (r)} \\
    &\qquad + C\left[1 + \frac{\Dbf (r)^{\gamma_3}}{r} + \frac{\boldsymbol{\Sigma} (r)}{r\Dbf (r)} + \Dbf (r)^{\gamma_3 - 1}\Dbf' (r) + r^{\gamma_3 m - 1}\right] \\
    &\leq - \frac{m-2}{r} - \frac{2\Gbf (r)}{r\Ebf (r)} + C\left[1 + \Dbf (r)^{\gamma_3 - 1}\Dbf' (r) + r^{\gamma_3 m - 1}\right].
    \end{align*}
    Now observe that Cauchy-Schwartz (applied first to the inner product, then to the integral) yields
    \[
    \Ebf (r)^2 \leq \Gbf (r)\Hbf (r),
    \]
    and thus
    \[
    \frac{2\Ebf (r)}{r\Hbf (r)} \leq \frac{2\Gbf (r)}{r\Ebf (r)}.
    \]
    Bringing everything together, we obtain
    \begin{align*}
    -\boldsymbol{\Omega} '(r) &\leq -\frac{1}{r} + \frac{m-1}{r} + \frac{2\Ebf (r)}{r\Hbf (r)} - \frac{m-2}{r} - \frac{2\Gbf (r)}{r\Ebf (r)} \\
    &\qquad + C\left[1 + \Dbf (r)^{\gamma_3 - 1}\Dbf' (r) + \frac{\boldsymbol{\Sigma} (r)\Dbf' (r)}{\Dbf (r)^2} + r^{\gamma_3 m -1}\right] \\
    &\leq C\left[1 + \Dbf (r)^{\gamma_3 - 1}\Dbf' (r) + \frac{\boldsymbol{\Sigma} (r)\Dbf' (r)}{\Dbf (r)^2} + r^{\gamma_3 m -1}\right] \\
    &\leq C\left[1 + \partial_r\left[\Dbf (r)^{\gamma_3}\right] + \frac{\boldsymbol{\Sigma} (r)\Dbf' (r)}{\Dbf (r)^2} + r^{\gamma_3 m -1}\right].
    \end{align*}
    Integrating over the interval $[a,b]$, we deduce that
    \begin{align*}
    \boldsymbol{\Omega} (a) - \boldsymbol{\Omega} (b) &\leq C\left[b + \Dbf (b)^{\gamma_3} - \Dbf (a)^{\gamma_3} - \int_a^b \boldsymbol{\Sigma} (r)\partial_r \left(\frac{1}{\Dbf (r)}\right)\dd r \right] \\
    &\leq C\left[b + \Dbf (b)^{\gamma_3} + \frac{\boldsymbol{\Sigma} (a)}{\Dbf (a)} - \frac{\boldsymbol{\Sigma} (b)}{\Dbf (b)} + \int_a^b \frac{\boldsymbol{\Sigma}' (r)}{\Dbf (r)}\dd r \right] \\
    &\leq C\left[b + \Dbf (b)^{\gamma_3} + \frac{\boldsymbol{\Sigma} (a)}{\Dbf(a)} + \int_a^b r \dd r \right] \\
    &\leq C\left[b + b^{\gamma_3 m} + b^2 \right].
    \end{align*}
    Note that the constant $C$ depends only on the dimension. Thus, setting
    \[
        f(r) = C\left[r + r^{\gamma_3 m} + r^2 \right],
    \]
    for $C$ as above, the proof is complete.
\end{proof}


\section{Uniform bounds on frequency function}\label{sct:frequb}

We now prove Theorem~\ref{thm:frequpperbd}. Before we proceed, we recall the following result from~\cite{DLS16blowup}, which infers height decay from excess decay for the $\Mcal$-normal approximations:

\begin{lemma}[\cite{DLS16blowup},~Lemma~5.2]\label{lem:currentcollapse}
    Using the notation of Section~\ref{sct:flattening}, suppose that $\{\Mcal_k\}_k$ is a sequence of center manifolds with corresponding $\Mcal_k$-normal approximations $N^{(k)}$, rescaled currents $T_k = T_{x_k,t_k}$ and centered intervals of flattening $(s_k,t_k]$ around $x_k$. Suppose that for some sequence $\eta_k \in \big(\frac{s_k}{t_k}, 3\big]$, we have
    \[
        \int_{\Bcal_{\eta_k}(\mathbf{p}_k(x_k))\setminus \Bcal_{\frac{\eta_k}{2}}(\mathbf{p}_k(x_k))} |N^{(k)}|^2 \longrightarrow 0 \qquad \text{as $k \to \infty$}..
    \] 
    Then we must necessarily have $\Ebf(T_k, \Bbf_{\eta_k}(x_k)) \longrightarrow 0$.
\end{lemma}
Although we have varying centers, the proof is analogous to that in~\cite{DLS16blowup}. Nevertheless, we repeat it here for the purpose of clarity.
\begin{remark}
    Recall that we already have decay of the excess at scales $r_k$ around $x_k$. This lemma will allow us to deduce that the excess also decays at any other scales along which the height decays.
\end{remark}
\begin{proof}
 
    Firstly, note that if $\eta_k \todown 0$, then by Proposition~\ref{prop:flattening}, we automatically have $\Ebf(T_k, \Bbf_{\eta_k}(x_k)) \longrightarrow 0$. 
    
    Thus, we can assume that $\limsup_k \eta_k > 0$. We prove this result by contradiction. Suppose that the statement of the lemma is not true. Then up to subsequence, there exists $\delta > \frac{s_k}{t_k}$ such that for $\Acal_\delta \coloneqq \Bcal_\delta(\mathbf{p}_k(x_k))\setminus \Bcal_{\frac{\delta}{2}}(\mathbf{p}_k(x_k)) \subset \Mcal_k$ we have
    \begin{equation}\label{eq:excesscontrad}
        \int_{\Acal_\delta} |N^{(k)}|^2 \longrightarrow 0 \qquad \text{but} \qquad \Ebf(T_k, \Bbf_{3}(x_k)) \geq c(\delta) >0,
    \end{equation}
    Now project this annulus $\Acal_\delta$ from the center manifold to the reference plane $\pi_k$. For $\eps_3$ sufficiently small, 
    this projection contains the annulus $A_\delta \coloneqq B_{\frac{15}{16}\delta}\setminus B_{\frac{9}{16}\delta}$. Consider the family of Whitney cubes $\Wscr^{(k)}$ for the center manifold $\Mcal_{k}$. If no cube from this family intersects $A_\delta$, then by construction, $N^{(k)} \equiv 0$ there. Otherwise, for each $k \in \N$ let $L_k \in \Wscr^{(k)}$ be the \emph{largest} cube that intersects the annulus $A_\delta$ and let $d_k \coloneqq \ell(L_k) < c_s \delta$. We will now use the height decay in~\eqref{eq:excesscontrad} to show that the sizes of these bad cubes shrink to zero.
    
    By Proposition~\ref{prop:flattening}, we may replace any cube in $\Wscr^{(k)}_n$ with a cube in either $\Wscr^{(k)}_e$ or $\Wscr^{(k)}_h$ of comparable size, in its domain of influence. Thus, we may assume that $L_k \in \Wscr^{(k)}_e \cup \Wscr^{(k)}_h$.
        
    Now since $L_k \cap A_\delta \neq \emptyset$ and $\delta$ is sufficiently large in comparison to $\ell(L_k)$, 
    we can find a ball $B^k \subset A_\delta$ of radius $\frac{\sqrt{m}d_k}{2}$ with $\dist(B^k,L_k) \leq \frac{\sqrt{m}d_k}{2}$. Then $B^k \subset B_{2\sqrt{m}d_k}(x_{L_k})$, 
    so we can apply~\cite[Proposition~3.1(S3)]{DLS16centermfld} to deduce that if $L_k \in \Wscr_k^{(k)}$, the height bound forces splitting for $N^{(k)}$:
    \[
        \int_{\Acal_\delta} |N^{(k)}|^2 \geq \int_{\boldsymbol{\Phi}_k(B^k)}|N^{(k)}|^2 \geq C [\boldsymbol{m}_0^{(k)}]^\frac{1}{m} d_k^{m+2+2\beta_2}.
    \]
    Observe that our choice of $B^k$ further ensures that the center $q_k$ of $B^k$ satisfies $\dist(L_k,q_k) \leq 4\sqrt{m}d_k$, so whenever $L_k \in \Wscr_e^{(k)}$, Proposition~\ref{prop:excesssplitting} can be applied to show that the splitting of $N^{(k)}$ is instead forced by the large excess:
    \[
        \int_{\Acal_\delta} |N^{(k)}|^2 \geq \int_{\boldsymbol{\Phi}_k\big(B_{\frac{d_k}{4}}(q_k)\big)}|N^{(k)}|^2 \geq C \boldsymbol{m}_0^{(k)} d_k^{m+4-2\delta_2}.
    \]
    In both of these estimates, $C= C(\beta_2,\delta_2,M_0,N_0,C_e,C_h)$. We thus conclude that $d_k \todown 0$.
    
    Another important consequence of the height bound~\cite[Lemma~A.1]{DLS16centermfld} is that locally above every cube $L \in \Wscr^{(k)}$, on each fiber (with respect to $\mathbf{p}_k$) one has the following control on the separation between $T_k$ and the center manifold $\Mcal_k$ (see~\cite[Cor.~2.2(ii)]{DLS16centermfld}):
    \[
        \spt \langle T_k, \mathbf{p}_k, \boldsymbol{\Phi}_k(q) \rangle \subset \set{y}{|\boldsymbol{\Phi}_k(q) - y|\leq C [\boldsymbol{m}_0^{(k)}]^{\frac{1}{2m}} \ell(L)^{1+\beta_2}} \quad \text{for each $q \in L,$}
    \]
    where $C=C(\beta_2,\delta_2,M_0,N_0,C_e,C_h)$. 
    
    Now since $L_k$ is the largest Whitney cube with $L_k \cap A_\delta \neq \emptyset$, this allows us to contain $\spt T_k \cap N_{\boldsymbol{\Phi}_k(A_\delta)} \Mcal_k$ inside a  $C[\boldsymbol{m}_0^{(k)}]^{\frac{1}{2m}}d_k^{1+\beta_2}$-tubular neighbourhood of $\Mcal_k$. Again, for $\eps_3$ sufficiently small, we can find $s<t$ independently of $k$ such that 
    \[
        \Bbf_t\setminus \Bbf_s \cap \Mcal_k \subset \boldsymbol{\Phi}_k(A_\delta).
    \]

    On the other hand, we can argue analogously to that in the proof of Proposition~\ref{prop:weakcpct} to find an $m$-dimensional area minimizing current $S$ in $\R^{m+n}$ with $T_k \toweakstar S$ and $\| T_k\|(\Bbf) \longrightarrow \|S\|(\Bbf)$ for any open ball $\Bbf$ (up to subsequence), such that either $S = Q\llbracket \pi \rrbracket$ for some $m$-plane $\pi$, or all points in $D_{\geq Q} S$ are singular points. But our assumption~\eqref{eq:excesscontrad} tells us that $\Ebf(S,\Bbf_3) \geq c(\delta)$ due to convergence of the excesses as in~\eqref{eq:excesslim}, ruling out the possibility that $S$ is flat.
    
    However, due to the uniform $\Crm^{3,\kappa}$-estimates on the center manifolds, we can use the Arzela-Ascoli Compactness Theorem to deduce that, up to subsequence, $\Mcal_k \longrightarrow \Mcal$ in $\Crm^3$. Combining this with the decay $d_k \todown 0$ for the sizes of the tubular neighbourhoods, we arrive at the conclusion that $\spt S\mres(\Bbf_t\setminus\bar{\Bbf}_s) \subset\Mcal\cap (\Bbf_t\setminus\bar{\Bbf}_s)$. 
    
    Now one can apply Allard's Constancy Theorem 
    to see that 
    \[
        S\mres (\Bbf_t\setminus\bar{\Bbf}_s) = Q_0\llbracket \Mcal\cap(\Bbf_t\setminus\bar{\Bbf}_s) \rrbracket \qquad \text{for some $Q_0$.}
    \]
    Recalling that each $T_k$ is a $Q$-fold cover of $\Mcal_k \cap (\Bbf_t\setminus \bar{\Bbf}_s)$ (see~\cite[Corollary~2.2(i)]{DLS16centermfld}), we may further deduce that $Q_0 = Q$, contradicting the fact that any $Q$-point of $S$ should be singular.
    
\end{proof}

\subsection{Proof of the upper frequency bound}
We are now in a position to show the uniform upper frequency bound of Theorem~\ref{thm:frequpperbd}. Note that by \cite[Theorem~5.1]{DLS16blowup}, we know that this result holds when all of the blow-ups have the same fixed center. The argument for varying centers and non-centered flattening intervals is essentially the same, since the estimates used are independent of the center. Nevertheless, we repeat it here for clarity, with some minor modifications that simplify the argument slightly.

\begin{proof}[Proof of Theorem~\ref{thm:frequpperbd}]
    Observe that the second part of this Theorem follows immediately from the first part, so we just need to prove the uniform boundedness around the points $x_k$ that we center our center manifolds around. 
    
    Suppose, for a contradiction, that the statement of the theorem is false. First of all, suppose that~\eqref{eq:smallfrequb} fails. Then, for an arbitrary choice of $J \in \Nbb$, up to subsequence (using the notation in Section~\ref{sct:flattening}) we can find scales $\frac{\rho_k}{t_k} \in \big(\frac{s_k}{t_k}, \frac{3}{2^{J-1}}\big]$ around $x_k$ such that
    \begin{equation}\label{eq:freqexplode}
        \lim_{k \to \infty}\Ibf_{k}\Big(\frac{\rho_k}{t_k}\Big) = \infty.
    \end{equation}
    By the almost-monotonicity~\eqref{eq:monfreq} of the frequency, we may assume that $\rho_k = \frac{3t_k}{2^{J-1}}$ for each $k \in \N$. 
    There are three possibilities to consider, based on the size of the scales along which the frequency blows up. Due to Proposition~\ref{prop:Dirbd}, unboundedness of the frequency along given scales corresponds to the collapsing of the sheets of $N^{(k)}$ along these scales. Either
    \begin{enumerate}[(a)]
        \item $\limsup_k t_k > 0$; \label{itm:limsuppos}
        \item $\limsup_k t_k = 0$, but there exists a subsequence (not relabelled) along which each starting scale $t_k$ comes immediately after a flattening interval $(\bar{s}_k,\bar{t}_k] \coloneqq (s^{(k)}_{j(k) - 1}, t^{(k)}_{j(k) - 1}]$, centered at the same point $x_k$; namely $t_k = \bar{s}_k$; \label{itm:smallradii1}
        \item $\limsup_k t_k = 0$ and for every $k$ (sufficiently large) we have $t_k < \bar{s}_k$. \label{itm:smallradii2}
    \end{enumerate}
    
    Let us begin with case \eqref{itm:limsuppos}. Extract a subsequence for which $\frac{3t_k}{2^{J-1}} \geq \eta$ for some $\eta > 0$. By Proposition~\ref{prop:Dirbd}, for every $r \in \big(\frac{s_k}{t_k}, 3\big]$, the Dirichlet energy estimate
    \begin{equation}\label{eq:Direps}
        \Dbf_{k}(r) \leq C\boldsymbol{m}_0^{(k)} r^{m+2-2\delta_2} \leq C \eps_3^2.
    \end{equation}
    holds. Hence, the unboundedness of $\Ibf_k\big(\frac{3}{2^{J-1}}\big)$ tells us that for 
    \[
        \Acal_J \coloneqq \Bcal_{\frac{3}{2^{J-1}}}(\mathbf{p}_k(x_k))\setminus\Bcal_{\frac{3}{2^{J-2}}}(\mathbf{p}_k(x_k)),
    \]
    we have
    \[
        \int_{\Acal_J} |N^{(k)}|^2 \longrightarrow 0 \quad \text{as $k \to \infty$.}
    \]
    We can now use Lemma~\ref{lem:currentcollapse} to deduce that
    \[
        \Ebf\big(T_k, \Bbf_{\frac{3}{2^{J-1}}}\big) = \Ebf\big(T, \Bbf_{\frac{3t_k}{2^{J-1}}}(x_k)\big)\longrightarrow 0.
    \]
    Since $\frac{3t_k}{2^{J-1}} \geq \eta$ and $x_k \to x$, we can thus conclude that
    \[
        \Ebf(T_{x,\eta}, \Bbf_1) = \Ebf(T, \Bbf_\eta(z)) \leq c\liminf_{k \to \infty} \Ebf\big(T, \Bbf_{\frac{3t_k}{2^{J-1}}}(x_k)\big) = 0.
    \]
    However, this contradicts the fact that $x \in \Sing_{\geq Q} T$.
    
    We now move on to case \eqref{itm:smallradii1}. In this case, the excess is still below the threshold $\eps_3^2$ at scale $\bar{s}_k = t_k$, but we must restart the center manifold because it no longer serves as a good approximation for the current at that scale. 
    
As a result of the condition \eqref{eq:go} in the flattening procedure, we encounter a cube $L \in \Wscr^{(k,j(k)-1)}$ of size $\ell(L) \geq c_s \frac{\bar{s}_k}{\bar{t}_k} \eqqcolon c_s \bar{r}_k$ with $L \cap \cl{B}_{\bar{r}_k} \neq \emptyset$ at which the refinement for $\Mcal_k$ stops. We then restart the center manifold $\Mcal_k$ and the $\Mcal_k$-normal approximation $N^{(k)}$ immediately at this scale, since the total excess remains below the threshold $\eps_3^2$, but locally near the center $x_k$ the excess becomes too large. 

The splitting observed at the stopping scale of the previous center manifold must then propagate to the new starting scale, so by the comparison of center manifolds in Proposition~\ref{prop:comparecentermfld} and the control on the Dirichlet energy in Proposition~\ref{prop:Dirbd} (at scale $R$), we deduce that
    \[
        \int_{\Bcal_R} |N^{(k)}|^2 \geq \bar{c}\int_{\Bcal_R} |DN^{(k)}|^2,
    \]
    where $\bar{c}$ is a geometric constant, independent of $k$ and of the center $x_k$.
    
    It remains to deal with the technicality that in our definition of the frequency, we are considering the $\Lrm^2$-mass of $N^{(k)}$ on \emph{annuli} rather than balls. This is where the choice of $J$ becomes important. We now revert back to our full collection of center manifolds. For any $J \in \Nbb$ and any $k,j \in \N$, we may apply H\"{o}lder's inequality and a Sobolev embedding for $q \in [2, 2^*]$ if $m > 2$ and any $q \in [2,\infty)$ otherwise, to deduce that
    \begin{align*}
        \int_{\Bcal_{\frac{3}{2^{J}}}} |N^{(k,j)}|^2 &\leq \left[\Hcal^m(\Bcal_{\frac{3}{2^{J}}})\right]^{1-\frac{2}{q}}\left[\int_{\Bcal_{\frac{3}{2^{J}}}} |N^{(k,j)}|^q \right]^{\frac{2}{q}} \\
        &\leq C 2^{-Jm\left(1-\frac{2}{q}\right)} \left[\int_{\Bcal_{\frac{3}{2^{J}}}} |N^{(k,j)}|^2 + \int_{\Bcal_{\frac{3}{2^{J}}}} |DN^{(k,j)}|^2 \right].
    \end{align*}
    Here, we have again used the uniform curvature bounds in Proposition~\ref{prop:centermfld}\eqref{itm:cmcurv} for $\Mcal_{k,j}$ to compare the volume of geodesic balls in $\Mcal_{k,j}$ with that of Euclidean ones. We can thus choose $J$ sufficiently large, \emph{independently of both $k$ and $j$}, such that
    \[
        \int_{\Bcal_3\setminus\Bcal_{\frac{3}{2^{J}}}} |N^{(k,j)}|^2 \geq c \int_{\Bcal_3} |DN^{(k,j)}|^2.
    \]
    By decomposing this large annulus dyadically, we can find a small annulus $\Acal_\ell \coloneqq \Bcal_{\frac{3}{2^{\ell}}}\setminus \Bcal_{\frac{3}{2^{\ell+1}}}$ for some $\ell = \ell(k,j) \leq J - 1$ with
    \begin{equation}\label{eq:annuli}
        \Hbf_{k,j}\Big(\frac{3}{2^{\ell}}\Big) \geq \int_{\Acal_\ell} |N^{(k,j)}|^2 \geq \frac{c}{2^J}\int_{\Bcal_3} |DN^{(k,j)}|^2 \geq \frac{c}{2^J} \Dbf_{k,j}\Big(\frac{3}{2^{\ell}}\Big),
    \end{equation}
    where $c$ is a geometric constant. Combining this with the almost-monotonicity of the frequency, we thus have $\Ibf_k\big(\frac{3}{2^\ell}\big) \leq C \Ibf_k\big(\frac{3}{2^{J-1}}\big) \leq C(J)$, contradicting~\eqref{eq:freqexplode}. 

    This brings us to the final case~\eqref{itm:smallradii2}, where for every $k$ sufficiently large, we consistently restart the center manifolds around $x_k$ because the excess exceeds the threshold $\eps_3^2$:
    \[
        \Ebf(T_k, \Bbf_{6\sqrt{m}\eta_k}(x_k)) > \eps_3^2 \qquad \text{for some $\eta_k \in \big(1,\min\{3,\frac{\bar{s}_k}{t_k}\}\big]$.}
    \]
    But then Lemma~\ref{lem:currentcollapse} tells us that
    \[
        \liminf_{k \to \infty} \Hbf_{k}(\eta_k) > 0.
    \]
    Combining this with the uniform boundedness of $\Dbf_{k}(\eta_k)$ from Proposition~\ref{prop:Dirbd} and the almost-monotonicity of the frequency, we achieve uniform bounds for $\Ibf_k$ on $\big(\frac{s_k}{t_k},\eta_k\big]$. Since $\eta_k \geq 1$ for every $k$, we can further choose $J$ such that $\frac{3}{2^{J-1}} \leq 1$. This once again gives the desired contradiction and so~\eqref{eq:smallfrequb} indeed holds for $J$ sufficiently large.
    
    Finally, we check that~\eqref{eq:largefrequb} for the choice of $J$ given by the validity of~\eqref{eq:smallfrequb}. Here, we simply observe that by~\eqref{eq:annuli}, for any $r \in \big(\frac{3}{2^{J-1}}, 3\big]$ we have
    \[
        \int_{\Bcal_r} |N^{(k,j)}|^2 \geq \Hbf_{k,j}\Big(\frac{3}{2^{J-1}}\Big) \geq c(J) \int_{\Bcal_3} |DN^{(k,j)}|^2 \geq c(J) r^2 \Dbf_{k,j}(r),
    \]
    where the presence of the factor $r^2$ can be verified by scaling.
\end{proof}

\subsection{Proof of the lower frequency bound}\label{sct:lb}
We now demonstrate the uniform lower frequency bound of Theorem~\ref{thm:freqlowerbd}, which will be vital for the persistence of singularities in the blow-up limit along the scales $r_k$.

Firstly, we require the following result, which tells us that frequency decay forces excess:
\begin{lemma}\label{lem:freqexcessdecay}
    There exists $\eps_3$ sufficiently small such that the following holds. Using the notation of Section~\ref{sct:flattening}, suppose $\Mcal_k$ is a sequence of center manifolds corresponding to rescaled currents $T_k = T_{x_k,t_k}$ satisfying Assumptions~\ref{asm:curr}-\ref{asm:Qfold}, with given intervals of flattening $(s_k,t_k]$ around $w_k \in D_{[Q,Q+\eps]}T_k\cap \cl{\Bbf}_1$. Suppose that there exists a sequence of scales $\frac{\rho_k}{t_k} \in \big(\frac{s_k}{t_k},R_k\big]$ for which
    \[
        \lim_{k \to \infty} \Ibf_k\Big(\frac{\rho_k}{t_k}\Big) = 0
    \]
    Then
    \[
        \lim_{k \to \infty}\Ebf\big(T_k, \Bbf_{\frac{\rho_k}{t_k}}(w_k)\big) = 0.
    \]
    Here $\Ibf_k$ and all related quantities are centered around $\mathbf{p}_k(w_k)$.
\end{lemma}
\begin{proof}
We argue in an analogous way to that in the proof of Lemma~\ref{lem:currentcollapse}. Once again, we may assume that the scales $\eta_k \coloneqq \frac{\rho_k}{t_k}$ satisfy $\limsup_k \eta_k \geq \delta > 0$. Firstly, we distinguish between two possibilities. Either
\begin{enumerate}[(a)]
    \item $\liminf_k\Hbf_k(\delta) = 0$;
    \item\label{eq:lrght} $\liminf_k \Hbf_k(\delta) > 0$.
\end{enumerate}
In the first scenario, we immediately apply Lemma~\ref{lem:currentcollapse} to conclude. So let us assume that~\eqref{eq:lrght} holds (up to subsequence). We again argue by contradiction. If the statement of the lemma is false, then the almost-monotonicity of the frequency tells us that
\begin{equation}\label{eq:contr}
\Ibf_{k} (\delta) \longrightarrow 0 \qquad \text{but} \qquad \Ebf\big(T_k, \Bbf_{\frac{\rho_k}{t_k}}(w_k)\big) \geq c(\delta) > 0.
\end{equation}
Combining this with~\eqref{eq:lrght}, we have
\begin{equation}\label{eq:Dirdecay}
    \Dbf_k(\delta) = \int_{\Bcal_\delta(\mathbf{p}_k(w_k))} |DN^{(k)}|^2 \longrightarrow 0.
\end{equation}
Projecting $\Bcal_\delta(\mathbf{p}_k(w_k)) \subset \Mcal_k$ to the reference plane $\pi_k$, we may assume that $\eps_3>0$ is small enough to ensure that $B_{\frac{15}{16}\delta}(\mathbf{p}_{\pi_k}(w_k)) \subset \mathbf{p}_{\pi_k}(\Bcal_\delta(\mathbf{p}_k(w_k)))$. Consider the family of Whitney cubes $\Wscr^{(k)}$ on $\pi_k$. If no cube in this family intersects $B_{\frac{15}{16}\delta}(\mathbf{p}_{\pi_k}(w_k))$, then by construction, $N^{(k)} \equiv 0$ there, which we have assumed is no the case.

Thus, for each $k \in \N$ we may select $L_k \in \Wscr^{(k)}$ to be the largest cube that intersects $B_{\frac{15}{16}\delta}(\mathbf{p}_{\pi_k}(w_k))$. Letting $d_k \coloneqq \ell(L_k) < c_s\delta$, we proceed to show that $d_k \todown 0$. 

We may assume that each cube $L_k$ is in $\Wscr^{(k)}_e$. Indeed, the fact that $\Theta(T_k,w_k) \geq Q$ ensures that the only stopping condition is~\eqref{itm:excess}, up to replacing any cube in $\Wscr_n^{(k)}$ by a nearby one in $\Wscr^{(k)}_e$ of a comparable size. The fact that we do not stop refining due to~\eqref{itm:height} can be ensured by choosing the constants $C_e$ and $C_h$ appropriately, due to the height bound~\cite[Theorem~A.1]{DLS16centermfld}.

Now since $L_k \cap B_{\frac{15}{16}\delta}(\mathbf{p}_{\pi_k}(w_k)) \neq \emptyset$ and $\delta> c_s^{-1}\ell(L_k)$, we may choose a ball $B^k \subset B_{\frac{15}{16}\delta}(\mathbf{p}_{\pi_k}(w_k))$ of radius $\sqrt{m}d_k$ with $\dist(B^k,L_k) \leq \sqrt{m}d_k$. 

Since $L_k \in \Wscr_e^{(k)}$, then the center $q_k$ of $B^k$ satisfies $\dist(L_k,q_k) \leq 4\sqrt{m}d_k$, and so we may use the splitting before tilting of Proposition~\ref{prop:excesssplitting} to deduce that
\begin{equation}\label{eq:splittinglb}
    \int_{\Bcal_\delta(\mathbf{p}_k(w_k))} |DN^{(k)}|^2 \geq \int_{\boldsymbol{\Phi}_k(B_{d_k/4}(q_k))} |DN^{(k)}|^2 \geq c \boldsymbol{m}_0^{(k)}d_k^{m+2-2\delta_2},
\end{equation}
where $c= c(\beta_2,\delta_2,M_0,N_0,C_e,C_h)$. Combining this with~\eqref{eq:Dirdecay}, we deduce that $d_k \longrightarrow 0$. We then proceed exactly as in the proof of Lemma~\ref{lem:currentcollapse} to extract an $m$-dimensional area minimizing integral current $S$ that is a weak-$*$ limit of the currents $T_k$, with $S = Q\llbracket \pi \rrbracket$ for some $m$-plane $\pi$, contradiction our large excess assumption.
\end{proof}

Let us now prove Theorem~\ref{thm:freqlowerbd}.
\begin{proof}[Proof of Theorem~\ref{thm:freqlowerbd}]
We will once again prove this by contradiction. We again adopt the notation in Section~\ref{sct:flattening}, and we let $\Ibf_k \coloneqq \Ibf_{\mathbf{p}_k(w_k),N^{(k)}}$ and we use analogous notation for all related quantities. If the statement of the theorem is false, we can find rescaled currents $T_k = T_{x_k,t_k}\mres \Bbf_{6\sqrt{m}}$, center manifolds $\Mcal_k$, $\Mcal_k$-normal approximations $N^{(k)}$, and flattening intervals $(s_k,t_k]$ around $\mathbf{p}_k(w_k)$ for $w_k \in D_{[Q,Q+\eps]} T_k$ such that the frequency $\Ibf_k$ satisfies
    \begin{equation}\label{eq:scales}
        \Ibf_k\left(\frac{\rho_k}{t_k}\right) \longrightarrow 0 \qquad \text{for some $\frac{\rho_k}{t_k} \in \left(\frac{s_k}{t_k}, R\right]$}.
    \end{equation}
    with
    \[
        R \coloneqq \liminf_k\dist(\mathbf{p}_k(w_k), \partial \Bcal_3(\mathbf{p}_k(x_k)).
    \]
    Observe that $R \geq 2$, since $w_k \in \Bbf_1(x_k)$. Up to selecting a subsequence, we may further assume that $w_k \to w$. We now have three possible cases to consider; either
\begin{enumerate}[(a)]
	\item\label{itm:comparableblowup} $s_k > 0$ for every $k$ and $\limsup_{k \to \infty} \frac{s_k}{\rho_k} > 0$;
    \item\label{itm:incomparableblowup} $s_k > 0$ for every $k$ and $\limsup_{k \to \infty} \frac{s_k}{\rho_k} = 0$;
    \item\label{itm:contactblowup} $s_k = 0$ eventually.
\end{enumerate}
In all three cases, we rescale the blow-ups of the current so that we are blowing up at the scales $\rho_k$ instead of $t_k$. Let $\frac{\bar{s}_k}{t_k} \in \big(\frac{R \rho_k}{2 t_k}, \frac{R\rho_k}{t_k}\big]$ be the scale at which the reverse Sobolev inequality of Corollary~\ref{cor:reverseSob} holds for $r = \frac{\rho_k}{t_k}$. Then let $\bar{r}_k \coloneqq \frac{2\bar{s}_k}{Rt_k} \in \big(\frac{\rho_k}{t_k}, \frac{2\rho_k}{t_k}\big]$. Let
    \[
        \bar{T}_k \coloneqq (\iota_{w_k,\bar{r}_k})_\sharp T_k = \big( (\iota_{w_k,\bar{r}_k t_k})_\sharp T \big)\mres \Bbf_{\frac{6\sqrt{m}}{\bar{r}_k}}, \quad \bar{\Sigma}_k \coloneqq \iota_{w_k,\bar{r}_k}(\Sigma_k), \quad \widebar{\Mcal}_k \coloneqq \iota_{w_k,\bar{r}_k}(\Mcal_k),
    \] 
    and let
    \[
        \bar{\boldsymbol{m}}_0^{(k)} \coloneqq \max\{\boldsymbol{c}(\bar{\Sigma}_k)^2, \Ebf(\bar{T}_k, \Bbf_{6\sqrt{m}})\}.
    \]
Since the ambient manifolds $\bar{\Sigma}_k$ converge in $\Crm^{3,\eps_0}_\loc$ to $T_w\Sigma \cong \R^{m+\bar{n}} \times \{0\}$, we have $\boldsymbol{c}(\bar{\Sigma}_k)^2 \xrightarrow{k \to \infty} 0$. Moreover, we we have $6\sqrt{m}\bar{r}_k \in \big(\frac{12\sqrt{m}\rho_k}{t_k}, \frac{24\sqrt{m}\rho_k}{t_k}\big]$, so Lemma~\ref{lem:freqexcessdecay} tells us that
    \begin{equation}\label{eq:excessdecay}
    \Ebf(\bar{T}_k, \Bbf_{6\sqrt{m}}) \leq C \Ebf\big(T_k, \Bbf_{\frac{24\sqrt{m}\rho_k}{t_k}}(w_k)\big) \longrightarrow 0.
    \end{equation}
    Thus, $\bar{\boldsymbol{m}}_0^{(k)} \to 0$. 

    By an analogous argument to that in the proof of Proposition~\ref{prop:weakcpct}, we may thus extract a subsequence for which $T_k \toweakstar \pi_\infty$ for some $m$-dimensional plane $\pi_\infty$. We will let $\mathbf{p}_\infty$ denote the orthogonal projection to this plane. We will use the notation $\bar{\mathbf{p}}_k$ for the projection map in Assumption~\ref{asm:projcm} for the rescaled center manifold $\widebar{\Mcal}_k$.
    
    Define
    \[
        \bar{N}^{(k)}: \widebar{\Mcal}_k \to \R^{m+n}, \qquad \bar{N}^{(k)}(p) \coloneqq \frac{1}{\bar{r}_k} N^{(k)}(\mathbf{p}_k(w_k) + \bar{r}_k p),
    \]
    and let
    \[
        v_k \coloneqq \frac{\bar{N}^{(k)} \circ \textbf{e}_k}{\mathbf{h}_k}, \qquad v_k:\pi_k \supset B_R \to \Acal_Q(\R^{m+n}),
    \]
    where $\textbf{e}_k$ is the exponential map at $p_k \coloneqq \frac{\boldsymbol{\Phi}_k(0)}{\bar{r}_k} \in \widebar{\Mcal}_k$ defined on $B_R \subset \pi_\infty \simeq T_{p_k} \widebar{\Mcal}_k$ and $\mathbf{h}_k \coloneqq \|\bar{N}_k\|_{\Lrm^2(\Bcal_{\frac{R}{2}})}$.
    
    Notice that 
     \[
        \mathbf{p}_k(w_k) \longrightarrow \mathbf{p}_\infty (w) \qquad \text{as $k \to \infty$.}
     \]
    
    This decay of the excesses at our chosen scales $\frac{\rho_k}{t_k}$ further allows us to conclude that 
    \[
        \widebar{\Mcal}_k \longrightarrow \pi_\infty \quad \text{in $\Crm^{3,\frac{\kappa}{2}}(\Bbf_{\frac{4}{3}R})$}, \qquad \mathbf{e}_k \longrightarrow \id : \pi_\infty \supset B_R \to B_R \quad \text{in $\Crm^{2,\frac{\kappa}{2}}$}.
    \]
    The former is an easy consequence of the estimates in Proposition~\ref{prop:centermfld}, whereas the latter is a technical result coming from the regularity of $\widebar{\Mcal}_k$; see~\cite[Proposition~A.4]{DLS16blowup} and the proof of~\cite[Lemma~6.1]{DLS16centermfld} for more details.
    
    Now we may pass the reverse Sobolev inequality that holds at scale $\frac{\bar{s}_k}{t_k}$ around $w_k$ to the maps $v_k$:
    \begin{align*}
        \int_{B_{\frac{R}{2}}} |v_k|^2 &\geq C\mathbf{h}_k^{-2} \int_{\Bcal_{\frac{R}{2}}} |\bar{N}^{(k)}|^2 \dd\Hcal^m \\
        &=\mathbf{h}_k^{-2}\bar{r}_k^{-2}\int_{\Bcal_{\frac{R}{2}}} |N^{(k)}(\mathbf{p}_k(w_k)+\bar{r}_k p)|^2 \dd\Hcal^m(p) \\
        &=\mathbf{h}_k^{-2}\bar{r}_k^{-m-2} \int_{\Bcal_{\frac{R\bar{r}_k}{2}}(\mathbf{p}_k(w_k))} |N^{(k)}|^2 \dd \Hcal^m \\
        &= \mathbf{h}_k^{-2}\bar{r}_k^{-m-2} \int_{\Bcal_{\frac{\bar{s}_k}{t_k}}(\mathbf{p}_k(w_k))} |N^{(k)}|^2 \dd \Hcal^m \\
        &\geq C\mathbf{h}_k^{-2} \bar{r}_k^{-m}\int_{\Bcal_{\frac{\bar{s}_k}{t_k}}(\mathbf{p}_k(w_k))} |DN^{(k)}|^2 \\
        &\geq C \int_{B_{\frac{R}{2}}} |Dv_k|^2.
    \end{align*}
    Here and in all that follows, we use the notation $\Bcal$ for both the balls on the original center manifolds $\Mcal_k$ and the rescaled and recentered ones $\widebar{\Mcal}_k$. It should always be clear from context which of these the ball lies in. Thus,
    \[
        \limsup_k \| v_k \|_{\Wrm^{1,2}(B_{\frac{R}{2}})} < \infty,
    \] 
    and so by the Banach-Alaoglu weak-$*$ compactness theorem and the reflexivity of the space $\Wrm^{1,2}$, we can extract a subsequence for which
    \[
        v_k \toweak v \quad \text{in $\Wrm^{1,2}(B_{\frac{R}{2}};\Acal_Q(\R^{m+n}))$.}
    \]
    Furthermore, Rellich-Kondrachov tells us that
    \[
        v_k \longrightarrow v \quad \text{strongly in $\Lrm^2(B_{\frac{R}{2}};\Acal_Q)$}.
    \]

    Now let us analyze the properties of this limit $v$. First of all, clearly $v$ is non-trivial, since our choice of normalization tells us that $\|v\|_{\Lrm^2(B_{\frac{R}{2}})} = 1$. 
    
    Moreover, the average of the sheets of $v$ vanishes. Indeed, arguing as in~\cite[(7.4)]{DLS16blowup}, the estimate~\eqref{eq:normalavgest} in Theorem~\ref{thm:Nlocalest} for the average of the sheets of $N^{(k)}$ allows us to deduce that
    \begin{align*}
        \int_{\Bcal_{\frac{R\bar{r}_k}{2}}(\mathbf{p}_k(w_k))} |\boldsymbol{\eta}\circ N^{(k)}| &\leq C\bar{\boldsymbol{m}}_0^{(k)}\bar{r}_k\sum_i \ell(L_i)^{m+2+\frac{\gamma_2}{2}}\\
        &\qquad+ \frac{C}{\bar{r}_k}\int_{\Bcal_{\frac{R\bar{r}_k}{2}}(\mathbf{p}_k(w_k))} \Gcal(N^{(k)}, Q\llbracket \boldsymbol{\eta}\circ N^{(k)}\rrbracket)^2 \\
        &\leq \frac{C}{\bar{r}_k}\int_{\Bcal_{R\bar{r}_k}(\mathbf{p}_k(w_k))} |N^{(k)}|^2,\notag
    \end{align*}
where $\{L_i\}_{i \in \N}$ is the family of cubes corresponding to the collection of cube-ball pairs $\Lscr = \{(L, B(L))\}$ from~\cite[Section~4.1]{DLS16blowup} that intersect $\mathbf{p}_{\pi_k}\big(\Bcal_{\frac{R\bar{r}_k}{2}}(\mathbf{p}_k(w_k))\big)$. Here, the final constant $C$ on the right-hand side depends on $R$ and the quantities stated in Theorem~\ref{thm:Nlocalest}. The final inequality follows from the arguments of~\cite[Section~4.1]{DLS16blowup} and the reverse Sobolev inequality of Corollary~\ref{cor:reverseSob}. Rescaling these estimates and using the $\Crm^0$-control from Theorem~\ref{thm:Nlocalest} for $\bar{N}^{(k)}$ combined with the excess decay~\eqref{eq:excessdecay}, we thus have
\begin{align*}
    \int_{B_{\frac{R}{2}}} |\boldsymbol{\eta}\circ v_k| &\leq C\mathbf{h}_k^{-1}\int_{\Bcal_{\frac{R}{2}}} |\boldsymbol{\eta}\circ \bar{N}^{(k)}| \\
    &= \mathbf{h}_k^{-1}\bar{r}_k^{-m-1} \int_{\Bcal_{\frac{R\bar{r}_k}{2}}(\mathbf{p}_k(w_k))} |\boldsymbol{\eta}\circ N^{(k)}| \\
    &\leq C\mathbf{h}_k^{-1}\bar{r}_k^{-m-2}\int_{\Bcal_{R\bar{r}_k}(\mathbf{p}_k(w_k))} |N^{(k)}|^2 \\
    &\leq C\mathbf{h}_k^{-1}\int_{\Bcal_{R}} |\bar{N}^{(k)}|^2 \\
    &\leq C\mathbf{h}_k \\
    &\leq [\bar{\boldsymbol{m}}_0^{(k)}]^{\frac{1}{2m}}\big[\Hcal^m(\Bcal_R(\bar{\Mcal}_k))\big]^{\frac{1}{2}} \\
    &\leq C[\bar{\boldsymbol{m}}_0^{(k)}]^{\frac{1}{2m}+\frac{1}{4}} \longrightarrow 0.
\end{align*}

    Combining this with the strong $\Lrm^2$-convergence of $v_k$ to $v$, we indeed have $\boldsymbol{\eta}\circ v \equiv 0$ almost-everywhere on $B_{\frac{R}{2}}$.
    
    Finally, one can show that $v$ is Dir-minimizing, and that in addition $v_k$ converges to $v$ strongly in $W^{1,2}_\loc$. This is proved by contradiction, using the existence of a better competitor function for the Dirichlet energy to build a suitable competitor current for $\bar{T}_k$, thus contradicting its area minimizing property. The argument may be found in~\cite[Section~7.3]{DLS16blowup}, and is omitted here.
    
    In summary, we have shown that the maps $v_k$ converge strongly in $\Wrm^{1,2}_\loc\cap \Lrm^2$ on $B_{\frac{R}{2}} \subset \pi_\infty$ to a map $v$ such that
    \begin{itemize}
        \item $v$ is non-trivial;
        \item $\boldsymbol{\eta}\circ v  \equiv 0$ almost-everywhere on $B_{\frac{R}{2}}$;
        \item $v$ is Dir-minimizing on $B_{\frac{R}{2}}$.
    \end{itemize}

    Moreover, the assumption~\eqref{eq:scales}, rescaled appropriately, tells us that
    \begin{equation}\label{eq:scaledfreqdecay}
        \bar{\Ibf}_k\left(\frac{1}{2}\right) \leq C\bar{\Ibf}_k\left(\frac{\rho_k}{\bar{r}_k t_k}\right) \leq C \Ibf_k\left(\frac{\rho_k}{t_k}\right) \xrightarrow{k \to \infty} 0,
    \end{equation}
    where $\bar{\Ibf}_k \coloneqq \Ibf_{0,\bar{N}^{(k)}}$.
    
    We will now reach a contradiction by showing that the $Q$-points $w_k$ for $T_k$ persist to the limit; namely, that $v(0) = Q\llbracket 0 \rrbracket$. We argue slightly differently for each of the cases outlined previously. We will henceforth let $\bar{\rho}_k \coloneqq \frac{\rho_k}{t_k}$. 
    
    \proofstep{Case~\eqref{itm:comparableblowup}.} In this case, the stopping condition~\eqref{eq:stop} for the intervals of flattening and Lemma~\ref{lem:Lucaheightbd} tell us that
    \[
    	\int_{\Bcal_{\frac{\sigma s_k}{t_k}}} |N^{(k)}|^2 \leq C \boldsymbol{m}_0^{(k)}\left(\frac{s_k}{t_k}\right)^{m+4-2\delta_2},
    \]
    for a (small) geometric constant $\sigma$. Moreover, the (localized) Sobolev embedding~\cite[Proposition~2.11]{DLS_MAMS} and the fact that (up to subsequence) $\rho_k \in (s_k,Cs_k]$ tells us that
    \begin{equation}\label{eq:SE}
    	\int_{\Bcal_{\bar\rho_k}} |N^{(k)}|^2 \leq C \int_{\Bcal_{\frac{\sigma s_k}{t_k}}} |N^{(k)}|^2 + C\bar\rho_k^{2}\int_{\Bcal_{\frac{\bar\rho_k}{2}}} |DN^{(k)}|^2.
    \end{equation}
    Now the excess splitting of Proposition~\ref{prop:excesssplitting} at scale $\frac{s_k}{t_k}$ and the comparability of $s_k$ and $\rho_k$ gives
    \[
    	\int_{\Bcal_{\bar\rho_k}} |DN^{(k)}|^2 \geq c \boldsymbol{m}_0^{(k)}\bar\rho_k^{m+2-2\delta_2}.
    \]
    Together with~\eqref{eq:SE} this yields
    \begin{equation}\label{eq:freqlower}
    	\int_{\Bcal_{\bar\rho_k}} |N^{(k)}|^2 \leq C \bar\rho_k^{2} \int_{\Bcal_{\frac{\bar\rho_k}{2}}} |DN^{(k)}|^2 \leq C \bar\rho_k^{2} \Dbf_k(\bar\rho_k).
    \end{equation}
    Finally, one may replace the left-hand side of~\eqref{eq:freqlower} with $\Hbf_k(\bar\rho_k)$ via the following estimate:
    \[
    	-\int_\Mcal \frac{1}{d(x)} \phi'\left(\frac{d(x)}{\bar\rho_k}\right) |N^{(k)}|^2 \dd x = 2\int_{\Bcal_{\bar\rho_k}\setminus\Bcal_{\frac{\bar\rho_k}{2}}} \frac{1}{d(x)} |N^{(k)}|^2 \dd x \leq C \bar\rho_k^{-1} \int_{\Bcal_{\bar\rho_k}} |N^{(k)}|^2.
    \]
    This, however, is in contradiction with~\eqref{eq:scaledfreqdecay}.

	\proofstep{Case~\eqref{itm:incomparableblowup}.}
	This case follows similarly to case~\eqref{itm:comparableblowup}, only now we have to propagate the lower frequency bound observed at the stopping scale $s_k$ up to the blow-up scale $\rho_k$.
	
	Proceeding as in case~\eqref{itm:comparableblowup} only at scale $\frac{10s_k}{t_k}$, we conclude that
	\[
		\Ibf_k\left(\frac{10s_k}{t_k}\right) \geq \eta,
	\]
	where $\eta$ is a small geometric constant. We may now use the almost-monotonicity of the frequency from Theorem~\ref{thm:monfreq} (cf.~\eqref{eq:monfreq}), with a choice of $c_0$ small enough so that
	\[
		\eta - e^{f(\bar\rho_k)}c_0 > 0,
	\]
	to conclude that we have a uniform lower bound for $\Ibf(\bar\rho_k)$, once again contradicting~\eqref{eq:scaledfreqdecay}.

\proofstep{Case~\eqref{itm:contactblowup}.} Finally, we need to deal with the case where $s_k = 0$ eventually along our sequence of flattening intervals. Let us fix an arbitrary center manifold $\Mcal$ and corresponding $\Mcal$-normal approximation $N$ in our sequence. We will denote the interval of flattening for this center manifold by $(0,t]$, and we will omit dependencies on $N$ for $\Ibf$ and related quantities.  In light of the estimates in Proposition~\ref{prop:firstvarest} and a more careful examination of the proof in~\cite{DLS16blowup}, we may conclude that
\[
	\partial_r \log\left(r^{-(m-1)}\Hbf(r)\right) \geq -C + \frac{2}{r}\Ibf(r) - C\Hbf(r)^{\gamma_3} - Cr\Ibf(r),
\]
for every $r \in (0,1)$, irrespectively of the existence of an a priori lower frequency bound. Thus, for any $0<\rho < 1$, letting
\[
	\Lambda(\rho) \coloneqq \int_\rho^1 \frac{\Ibf(r)}{r},
\]
we have
\[
	e^{2\Lambda(\rho)}\rho^{-(m-1)}\Hbf(\rho) \leq C \Hbf(1) \leq C.
\]
Combining this with the reverse Sobolev inequality in Corollary~\ref{cor:reverseSob} (or the uniform upper frequency bound in Theorem~\ref{thm:frequpperbd}), we deduce that
\[
	\Dbf(\rho) \leq C\rho^{m-2} e^{-2\Lambda(\rho)}.
\]
On the other hand, Proposition~\ref{prop:Dirbd} gives the decay rate
\[
	\Dbf(\rho) \leq C\rho^{m+2 -2\delta_2}.
\]
This forces the inequality
\[
	e^{-2\Lambda(\rho)} \geq c\rho^{2(2-\delta_2)},
\]
or equivalently
\[
	\Ibf(\rho) \geq c(2-\delta_2) \qquad \forall \ \rho \in (0,1),
\]
contradicting~\eqref{eq:scaledfreqdecay} for $k$ sufficiently large.

\end{proof}

\section{Persistence of $Q$-points}\label{sct:Qpts}
We now construct our limiting Dir-minimizer using our blow-up sequence with varying centers $x_k$, which we assume converge to some point $x$. Recall that we wish to blow up at the scales $r_k$ around $x_k$, along which $T$ has a flat tangent cone as in Proposition~\ref{prop:weakcpct}. Thus, we need to take a diagonal sequence of center manifolds with the intervals of flattening $(s_k, t_k] \ni r_k$ centered at $x_k$, and rescale them appropriately so that the reverse Sobolev inequality of Corollary~\ref{cor:reverseSob} holds at scale $1$, analogously to that in the preceding section. As usual, we let $\Mcal_k$ denote the center manifold at scale $t_k$ around $x_k$, with corresponding current $T_k = T_{x_k,t_k}\mres \Bbf_{6\sqrt{m}}$. 

Let $\frac{\bar{s}_k}{t_k} \in \big(\frac{3r_k}{2 t_k}, \frac{3r_k}{t_k}\big]$ be the scale at which the reverse Sobolev inequality holds for $r = \frac{r_k}{t_k}$. Then let $\bar{r}_k \coloneqq \frac{2\bar{s}_k}{3t_k} \in \big(\frac{r_k}{t_k}, \frac{2r_k}{t_k}\big]$. Let
\[
    \bar{T}_k \coloneqq (\iota_{0,\bar{r}_k})_\sharp T_k =  \big((\iota_{x_k,\bar{r}_k t_k})_\sharp T\big)\mres \Bbf_{\frac{6\sqrt{m}}{\bar{r}_k}}, \qquad \bar{\Sigma}_k \coloneqq \iota_{0,\bar{r}_k}\Sigma_k, \qquad \widebar{\Mcal}_k \coloneqq \iota_{0,\bar{r}_k}\Mcal_k,
\]
and let
\[
    \bar{\boldsymbol{m}}_0^{(k)} \leq \max\{\boldsymbol{c}(\bar{\Sigma})^2, \Ebf(\bar{T}_k, \Bbf_{6\sqrt{m}})\}.
\]
Since the ambient manifolds $\bar{\Sigma}_k$ converge in $\Crm^{3,\eps_0}_\loc$ to $T_x\Sigma \simeq \R^{m+\bar{n}} \times \{0\}$, for $k$ sufficiently large we have
\[
    \boldsymbol{c}(\bar{\Sigma}_k)^2 \longrightarrow 0.
\]
We will use the notation $\bar{\mathbf{p}}_k$ for the projection map in Assumption~\ref{asm:projcm} for the rescaled center manifold $\widebar{\Mcal}_k$. Moreover, $6\sqrt{m}\bar{r}_k \in \big(\frac{12\sqrt{m}\rho_k}{t_k}, \frac{24\sqrt{m}\rho_k}{t_k}\big]$, so Lemma~\ref{lem:freqexcessdecay} tells us that
\begin{equation}
    \Ebf(\bar{T}_k, \Bbf_{6\sqrt{m}}) \leq C \Ebf\big(T_k, \Bbf_{\frac{24\sqrt{m}r_k}{t_k}}\big) \longrightarrow 0.
\end{equation}
Thus, $\bar{\boldsymbol{m}}_0^{(k)} \to 0$.

We now argue in the same way as in the proof of Theorem~\ref{thm:freqlowerbd}. Let $\pi_\infty$ be the plane defining the flat tangent cone of $T_k$ as in Proposition~\ref{prop:weakcpct}. We will let $\mathbf{p}_\infty$ denote the orthogonal projection to this plane.

Define
\[
    \bar{N}^{(k)}: \widebar{\Mcal}_k \to \R^{m+n}, \qquad \bar{N}^{(k)}(p) \coloneqq \frac{1}{\bar{r}_k} N^{(k)}(\bar{r}_k p),
\]
and let
\[
    u_k \coloneqq \frac{\bar{N}^{(k)} \circ \textbf{e}_k}{\mathbf{h}_k}, \qquad u_k:\pi_k \supset B_3 \to \Acal_Q(\R^{m+n}),
\]
where $\textbf{e}_k$ is the exponential map at $p_k \coloneqq \frac{\boldsymbol{\Phi}_k(0)}{\bar{r}_k} \in \widebar{\Mcal}_k$ defined on $B_3 \subset \pi_k \simeq T_{p_k} \widebar{\Mcal}_k$
and $\mathbf{h}_k \coloneqq \|\bar{N}^{(k)}\|_{\Lrm^2(\Bcal_{\frac{3}{2}})}$. 
As before, this decay of the excesses at our chosen scales $\frac{r_k}{t_k}$ further allows us to conclude that
\[
    \widebar{\Mcal}_k \longrightarrow \pi_\infty \quad \text{in $\Crm^{3,\frac{\kappa}{2}}(\Bbf_{\frac{4}{3}})$}, \qquad \mathbf{e}_k \longrightarrow \id: \pi_\infty \supset B_3 \to B_3 \quad \text{in $\Crm^{2,\frac{\kappa}{2}}$}.
\]
Now we pass the reverse Sobolev inequality that holds at scale $\frac{\bar{s}_k}{t_k}$ around $x_k$ to the maps $u_k$:
\[
    \int_{B_{\frac{3}{2}}} |u_k|^2 \geq C \int_{B_{\frac{3}{2}}} |Du_k|^2.
\]  
Once again, this allows us to use both weak-$*$ compactness and the Rellich-Kondrachov Compact Embedding Theorem to deduce that there exists a map $u \in \Wrm^{1,2}(B_{\frac{3}{2}};\Acal_Q)$ such that up to subsequence,
\[
    u_k \longrightarrow u \quad \text{strongly in $\Lrm^2(B_{\frac{3}{2}};\Acal_Q)$}.
\]
Combining this with the reverse Sobolev inequality, we can in fact further improve this strong convergence to that in $\Wrm^{1,2}_\loc(B_{\frac{3}{2}};\Acal_Q)$. 

Moreover, by arguing exactly as in the proof of the lower frequency bound, $u$ satisfies the following properties:
\begin{itemize}
    \item $u$ is non-trivial;
    \item $\boldsymbol{\eta}\circ u \equiv 0$ almost-everywhere on $B_{\frac{3}{2}}$;
    \item $u$ is Dir-minimizing on $B_{\frac{3}{2}}$.
\end{itemize}
It remains to show that whenever $y_k \to y$, with $y_k \in \cl{\Bbf}_1 \cap \Sing_{[Q,Q+\eps]}T_k$, we have $u (y) = Q \llbracket 0 \rrbracket$. We already know that $y_k$ is a $Q$-point for $T_k$, but this does not ensure that the excess should decay quickly enough close to $y_k$; we might need stop our center manifolds at scales close to these points. Even in the case where we can refine indefinitely around each $y_k$, it is still necessary to check that this behaviour persists in the limit as we take $k \to \infty$. We consider cases depending on the non-centered intervals of flattening for $\Mcal_k$ around $\mathbf{p}_k(y_k)$. For $\tilde{s}_k \coloneqq s(y_k)$ as in Definition~\ref{def:nonctrrefine}, there are 3 possibilities: 
\begin{enumerate}[(a)]
    \item There is a subsequence of indices $k$ for which $\tilde{s}_k = 0$; namely, $\mathbf{p}_k(y_k)$ which lie in the contact sets $\boldsymbol{\Phi}_k(\boldsymbol{\Gamma}_k)$ for the center manifolds $\Mcal_k$, \label{itm:contactcase}
    \item There is a subsequence of indices $k$ with $\frac{\tilde{s}_k}{r_k} \todown 0$; namely, we stop refining around $\mathbf{p}_{\pi_k}(y_k)$ at scales that decay relative to the blow-up scales $r_k$, \label{itm:quasicontact}
    \item We have $\liminf_k \frac{\tilde{s}_k}{r_k} > 0$; we stop refining around $\mathbf{p}_{\pi_k}(y_k)$ at scales comparable to the blow-up scales. \label{itm:Lucacase}
\end{enumerate}
Note that these points $y_k$ are \emph{not} the centers for our center manifold constructions; they are simply points lying in $\cl{\Bbf}_1$ that are `captured' by $T_k$.

The idea is as follows. We wish to establish persistence of $Q$-points at the blow-up scales $r_k$, by exploiting the uniform frequency bound of Theorem~\ref{thm:freqlowerbd}. 

Unfortunately, the scales $r_k$ may be very far from the scales at which we stop refining around $y_k$ and inherit the behaviour from $T_k$. These stopping scales are those at which we may use the lower frequency bound to show that the singularities persist. We then propagate this persistence up to the larger scales $r_k$ that we are interested in, via the almost-monotonicity of Theorem~\ref{thm:monfreq}.

\subsection{Case~\eqref{itm:contactcase}}
This case can be thought of as case~\eqref{itm:quasicontact} with `$\tilde{s}_k = 0$'. More precisely, we can continue our refinement procedure indefinitely around $\mathbf{p}_{\pi_k}(y_k)$,  
so the limiting frequency of $N^{(k)}$ at each $y_k$ exists.
\begin{lemma}\label{prop:contactfreq}
    Suppose that we have a center manifold $\Mcal$ with corresponding $\Mcal$-normal approximation $N$, rescaled current $T'$ and a point $w \in \spt T'$ for which $\mathbf{p}(w) \in \boldsymbol{\Phi}(\boldsymbol{\Gamma})$. Then the limit
    \begin{equation}\label{eq:freqlimit}
    \Ibf_{\mathbf{p}(w),N}(0) \coloneqq \lim_{r \todown 0} \Ibf_{\mathbf{p}(w),N}(r)\quad \text{exists}. 
    \end{equation}
\end{lemma}

\begin{proof}
    The proof of this is a simple consequence of the almost-monotonicity of the frequency in Theorem~\ref{thm:monfreq}, where we crucially exploit the behaviour of the error function $f$. Consider
    \[
        \boldsymbol{\Omega}_0 \coloneqq \liminf_{r \todown 0} \boldsymbol{\Omega}_{\mathbf{p}(w),N}(r) \in [c_0, +\infty).
    \]
    Fix $\eps > 0$. By the nature of $f$ from Theorem~\ref{thm:monfreq}, we can choose $r_0 > 0$ (independently of the center manifold $\Mcal$ and the interval of flattening) such that
    \[
        f(r) < \frac{\eps}{2} \quad \text{for every $r < r_0$}, \qquad \text{and} \qquad \boldsymbol{\Omega}_{\mathbf{p}(w),N}(r_0) < \boldsymbol{\Omega}_0 + \frac{\eps}{2}.
    \]
    But then by the almost-monotonicity, for every $r < r_0$ it holds that
    \[
        \boldsymbol{\Omega}_{\mathbf{p}(w),N}(r) \leq \boldsymbol{\Omega}_{\mathbf{p}(w),N}(r_0) + \frac{\eps}{2} \leq \boldsymbol{\Omega}_0 + \eps.
    \]
    Thus, the limit in \eqref{eq:freqlimit} indeed exists. 
\end{proof}

Let us now proceed with the proof that $u(y) = Q\llbracket 0 \rrbracket$ in case~\eqref{itm:contactcase}. We can combine the result of the above Proposition with Theorem~\ref{thm:freqlowerbd} to deduce that we have the uniform lower bound
    \begin{equation}\label{eq:uniffreqlb}
        \Ibf_{\bar{\mathbf{p}}_k(y_k),\bar{N}^{(k)}}(0) \geq c_0 > 0.
    \end{equation}
    Furthermore, the $\Wrm^{1,2}_\loc$-convergence of $u_k$ to $u$ allows us to establish convergence of the frequencies
    \begin{equation}\label{eq:regfreqconv}
        \Ibf_{\mathbf{p}_k(y_k), N^{(k)}}(r)\longrightarrow \Ibf_{y,u}(r) \qquad \text{for each $r \in \left[0,\dist\big(y,\partial B_{\frac{3}{2}}\big)\right)$,}
    \end{equation}
    where $\Ibf_{y,u}$ is the regularized frequency of $u$ centered at $y$, defined via the same Lipschitz cutoff as in Definition~\ref{def:freq}. Notice that we may include $r=0$ in this convergence, due to Lemma~\ref{prop:contactfreq}.
    
    Then, a simple consequence of~\cite[Theorem~9.6]{DLsurvey} is that for non-trivial Dir-minimizers $u$, the limit $\Ibf_{y,u}(0)$ exists and that its value is the homogeneity of any tangent function $g$ that is the subsequential limit as $r \todown 0$ of the blow-ups
    \[
        z \mapsto \frac{r^{\frac{m-2}{2}} u(y+ rz)}{\left[\Dir(u, B_r(y))\right]^{\frac{1}{2}}}.
    \]
    Thus, we must necessarily have
    \[
        \Ibf_{y,u}(0) = I_{y,u}(0) \geq c_0 > 0.
    \]
    This allows us to apply Lemma~\ref{lem:freqdecay} to conclude.

Finally, we deal with the `quasicontact' case. This is proven in essentially the same way as case~\eqref{itm:contactcase}, but we cannot take the limit of the frequency \emph{all the way to zero} around every fixed $y_k$. However, we can use the information of the frequency at zero from the limiting Dir-minimizer as $k \to \infty$.

\subsection{Case~\eqref{itm:quasicontact}}
In this case, we use the same reasoning as for the case~\eqref{itm:contactcase}, only now we must stop refining at scales close to $y_k$, that shrink to zero with $k$. We can still show persistence of singularities at these stopping scales, asymptotically. 

In this case, we proceed as follows. Firstly, observe that there is almost monotonicity~\eqref{eq:monfreq} for the frequency $\Ibf_{\mathbf{p}_k(y_k),N^{(k)}}$ for all scales $r \in \left(\frac{\tilde{s}_k}{t_k}, 1\right]$, where the function $f$ in Theorem~\ref{thm:monfreq} is \emph{independent of $k$}.
    
Due to the nature of $f$, an analogous argument to that in the proof of Lemma~\ref{prop:contactfreq} yields the existence of the limit
\[
    \Ibf_0 \coloneqq \lim_{k\to \infty} \Ibf_{\mathbf{p}_k(y_k),N^{(k)}}\left(\frac{2\tilde{s}_k}{t_k}\right).
\]

By Theorems~\ref{thm:frequpperbd} and~\ref{thm:freqlowerbd}, we can deduce that
\[
    0 < c_0 \leq \Ibf_0 \leq C_0 < \infty.
\]
We may now use the almost-monotonicity of the frequency to choose $r_0 > \frac{2\tilde{s}_k}{t_k}$ independently of $k$ such that
\[
    \frac{c_0}{2} \leq \Ibf_{\mathbf{p}_k(y_k),N^{(k)}}(r) = \Ibf_{\bar{\mathbf{p}}_k(y_k),\bar{N}^{(k)}}\left(\frac{r}{\bar{r}_k}\right) \leq 2c_0 \qquad \text{for every $r \in \left(\frac{2\tilde{s}_k}{t_k}, r_0\right],$}
\]
provided that we take $k$ sufficiently large. But now we can once again use the convergence~\eqref{eq:regfreqconv} and~\cite[Theorem~9.6]{DLsurvey}, combined with the fact that (up to subsequence) $\frac{\tilde{s}_k}{r_k} \todown 0$, to deduce that
\[
    \frac{c_0}{4} \leq \Ibf_{y,u}(r) \leq 4c_0 \qquad \text{for every $r \in \left(0, \frac{r_0}{\bar{r}_k}\right].$}
\]
We may now once again use Proposition~\ref{lem:freqdecay} to reach the desired conclusion.

\subsection{Case~\eqref{itm:Lucacase}}
In this case, (up to subsequence) the scales $\tilde{s}_k$ at which we stop refining around $\mathbf{p}_k(y_k)$ are comparable to the blow-up scales $r_k$. Take the cubes $L_k \in \Wscr^{(k)}$ with $\ell(L_k) \geq c_s\frac{\tilde{s}_k}{t_k}$ and $\dist(L_k,\mathbf{p}_{\pi_k}(y_k)) \leq \frac{\tilde{s}_k}{t_k}$. 

We claim that $u(y) = Q\llbracket 0 \rrbracket$. In this case, we are unable to exploit the uniform lower bound of the frequency over the flattening intervals, since this only tells us that $I_u$ has a positive lower bound at some scale comparable to $\liminf_{k \to \infty} \frac{\tilde{s}_k}{r_k} > 0$. We instead argue as follows.

Consider the strong Lipschitz approximations $f_k: B_{8r_{L_k}}(p_{L_k},\pi_{L_k}) \to \Acal_Q(\pi_{L_k}^\perp)$ from Theorem~\ref{thm:strongLip} for $T_k$. Take $g_k: \Mcal_k \to \Acal_Q(\R^{m+n})$ to be the reparameterization $N$ from ~\cite[Theorem~5.1]{DLS_multiple_valued} of $f = f_k$ to the center manifolds. We may now rescale and normalize $g_k$ in the same way as the $\Mcal_k$-normal approximations $N^{(k)}$. Namely, we define $\bar{g}_k: \widebar{\Mcal}_k \to \Acal_Q(\R^{m+n})$ and $v_k: \pi_k \supset B_3 \to \Acal_Q(\R^{m+n})$ by
\[
\bar{g}_k(p) \coloneqq \frac{1}{\bar{r}_k} g_k(\bar{r}_k p), \qquad v_k \coloneqq \frac{\bar{g}_k\circ \mathbf{e}_k}{\mathbf{h}_k}.
\]
Then, for $k$ sufficiently large, we have
\begin{align*}
\int_{B_{\frac{3}{2}}} \Gcal(v_k,u_k)^2 &\leq C\mathbf{h}_k^{-2}\int_{\Bcal_{\frac{3}{2}}} \Gcal(\bar{g}_k,\bar{N}^{(k)})^2 \\
&\leq C\mathbf{h}_k^{-2}\bar{r}_k^{-m-2}\int_{\Bcal_{\frac{3}{2}\bar{r}_k}} \Gcal(g_k,N^{(k)})^2 \\
&\leq C\mathbf{h}_k^{-2}\bar{r}_k^{-m-2}\int_{\mathbf{\Phi}_k(B_{\frac{3}{2}\bar{r}_k}\setminus K_k)} \Gcal(g_k,N^{(k)})^2, \\
\end{align*}
where $\mathbf{\Phi}_k$ is the map parameterizing the center manifold $\Mcal_k$ over the plane $\pi_k$, as given in Definition~\ref{def:cm}. Let $E_k \coloneqq \Ebf(T_k, \Cbf_{32r_{L_k}}(p_{L_k}),\pi_{L_k})$. Notice that since $r_{L_k} \simeq \ell(L_k) \simeq \bar{r}_k$, we have $E_k \longrightarrow 0$. Furthermore, since $\Theta(T_k, y_k) \geq Q$, we necessarily have $L_k \in \Wscr^{(k)}_e$. Therefore, the stopping condition~\eqref{itm:excess}, the splitting before tilting in Proposition~\ref{prop:excesssplitting}, and the improved height bound in Lemma~\eqref{lem:Lucaheightbd} tell us that
\begin{equation}\label{eq:excessheightcmp}
    E_k \simeq \boldsymbol{m}_0^{(k)} \left(\frac{r_k}{t_k}\right)^{2-2\delta_2} \simeq \left(\frac{r_k}{t_k}\right)^{-m-2}\int_{\Bcal_{\frac{\sigma r_k}{t_k}}(\mathbf{p}_k(y_k))} |N^{(k)}|^2 \simeq \mathbf{h}_k^2,
\end{equation}
for some choice of $\sigma > 0$ that is independent of $k$. We may now use the $\Lrm^\infty$-estimates~\eqref{eq:oscbd} and~\eqref{eq:normalLip}, the height bounds~\eqref{eq:badcubesbds}, the $\Crm^{3,\kappa}$-estimate~\eqref{itm:cmcurv} of Proposition~\ref{prop:centermfld}, and the bound in Theorem~\ref{thm:strongLip} for the sets $K_k$ over which $T_k\mres\Cbf_{32r_{L_k}}(p_{L_k},\pi_{L_k})$ is not graphical, to achieve the estimate
\[
    \int_{\mathbf{\Phi}_k(B_{\frac{3}{2}\bar{r}_k}\setminus K_k)} \Gcal(g_k,N^{(k)})^2 \leq C\bar{r}_k^{2+2\beta_2}E_k^{\frac{1}{m}}\big|B_{\frac{3}{2}\bar{r}_k}\setminus K_k\big| \leq  C\bar{r}_k^{m+2+2\beta_2}E_k^{\frac{1}{m} + 1 + \gamma_1}.
\]
Combining this with the above control on the $\Lrm^2$-distance between $v_k$ and $u_k$, the comparability~\eqref{eq:excessheightcmp} of the $\Lrm^2$ height and the excess, and the excess decay $E_k \longrightarrow 0$, we deduce that
\[
    \int_{B_{\frac{3}{2}}} \Gcal(v_k,u_k)^2 \leq C\bar{r}_k^{2\beta_2}E_k^{\frac{1}{m} + \gamma_1} \longrightarrow 0.
\]
Due to the convergence of $u_k$ to $u$, this further tells us that $v_k$ converges in $\Lrm^2(B_{\frac{3}{2}};\Acal_Q)$ on to the same Dir-minimizer $u$. It remains to obtain persistence of $Q$-points for the sequence of maps $v_k$. Indeed, since $\Theta(T_k, y_k) \geq Q$, Theorem~\ref{thm:Qptspersistf}\eqref{eq:Qpts2} tells us that there exists $\gamma > 0$ for which
\[
    \int_{B_{s\bar{r}_k}(\mathbf{p}_{\pi_{L_k}}(y_k))} \Gcal(f_k, Q\llbracket \boldsymbol{\eta}\circ f_k\rrbracket)^2 \leq C s^{m+\gamma}\bar{r}_k^{m+2} E_k,
\]
for every $s > 0$ sufficiently small. Applying the estimates in~\cite[Theorem~5.1]{DLS_multiple_valued}, rescaling, making use of Theorem~\ref{thm:strongLip}\eqref{eq:oscbd} and~\eqref{eq:excessheightcmp}, and using the control~\cite[Proposition~4.1(iv)]{DLS16centermfld} on the tilting between the planes $\pi_{L_k}$ and $\pi_k$, we have
\begin{align*}
    \int_{B_s(\mathbf{p}_{\pi_k}(y_k))} |v_k|^2 &\leq C\mathbf{h}_k^{-2}\bar{r}_k^{-m-2}\int_{B_{s\bar{r}_k}(\mathbf{p}_{\pi_{k}}(y_k))}|g_k\circ\mathbf{\Phi}_k|^2 \\
&\leq C\mathbf{h}_k^{-2}\bar{r}_k^{-m-2}\int_{B_{s\bar{r}_k}(\mathbf{p}_{\pi_{L_k}}(y_k))}\Gcal(f_k(p), Q\llbracket \boldsymbol{\vphi}_k(p)\rrbracket)^2 \\
&\leq C\mathbf{h}_k^{-2}\bar{r}_k^{-m-2}\int_{B_{s\bar{r}_k}(\mathbf{p}_{\pi_{L_k}}(y_k))}\Gcal(f_k(p), Q\llbracket \boldsymbol{\eta}\circ f_k\rrbracket)^2  \\
&\qquad + C\mathbf{h}_k^{-2}s^m\bar{r}_k^{-2}\Lip(f_k)^2\|\boldsymbol{\vphi}_k\|^2_{\Crm^1} \|g_k\circ \mathbf{\Phi}_k\|^2_{\Lrm^\infty(B_{\bar{r}_k}(\mathbf{p}_{\pi_k}(y_k)))} \\
&\leq Cs^{m+\gamma} + Cs^m\bar{r}_k^{2\beta_2} E_k^{2\gamma_1 + \frac{1}{m}},
\end{align*}
for every $s < s_0$, where $s_0$ is independent of $k$.
Taking $k \to \infty$, again using the excess decay, and combining with the strong $\Lrm^2$-convergence of $v_k$ to $v$, we have
\[
\int_{B_s(y)} |v|^2 \leq Cs^{m+\gamma}.
\]
Thus, applying Proposition~\ref{lem:freqdecay}, we are able to conclude that $I_v(0) \geq \frac{\gamma}{2}$ and that $v(y) = Q\llbracket 0 \rrbracket$ as expected.

\section{Proof of the results in Section~\ref{sct:prelim}}\label{sct:workraccoon}
Let us prove the preliminary results contained in Section~\ref{sct:prelim}. We begin with the proof of Lemma~\ref{lem:workraccoon}. The proof of all but~\eqref{eq:Minkowskicpct} appears in~\cite{White86}, but nevertheless we repeat the argument here.

Before we begin with the proof, let us make the following important observation. Assuming the validity of Lemma~\ref{lem:workraccoon},
    \begin{equation}\label{eq:cpctMinkowski}
    \Ascr(C) = \set{\alpha \geq 0}{\lim_{r \todown 0}\sup_{K \in \Cscr} r^\alpha N(K,r) = 0}.
    \end{equation}
    
    Indeed, if $\alpha \in \Ascr(\Cscr)$, then the Work Raccoon Lemma~\ref{lem:workraccoon} allows us to choose $\beta \in (\alpha_0,\alpha)$, so we have
    \[
    N(K,r)r^\alpha \leq r^{\alpha - \beta} N(K,r) r^\beta \leq C(\beta) r^{\alpha - \beta} \to 0 \qquad \text{as $r \todown 0$.}
    \]
    Conversely, if
    \[
    \lim_{r \todown 0}\sup_{K \in \Cscr} r^\alpha N(K,r) = 0,
    \]
    then given any $K \in \Cscr$, for any $\delta > 0$ we can find a covering $\Uscr_\delta$ of $K$ by $N(K,\delta)$ open $(m+n)$-dimensional balls of radius $\delta$. Then
    \[
    \sum_{\Bbf \in \Uscr_\delta} (\diam \Bbf)^\alpha = 2\delta^\alpha N(K,\delta).
    \]
    Taking $\delta \todown 0$, the claim follows from the definition of the $\alpha$-dimensional Hausdorff measure.
    
    In view of the above, it is therefore crucial to prove that the half-line is \emph{open} to make the identification~\eqref{eq:cpctMinkowski}. Hence, in order to estimate the upper Minkowski content uniformly over the compact family $\Cscr$, it suffices to estimate the Hausdorff dimension uniformly.

\begin{proof}[Proof of Lemma~\ref{lem:workraccoon}]
    The fact that $\Ascr(\Cscr)$ is an upper half-line is clear, since it is the intersection of upper half-lines, so to establish the first statement of the theorem, we just needs to show that it is open. Fix $\beta \in \Ascr(\Cscr)$. We may assume that $\beta > 0$, since $0 \notin \Ascr(\Cscr)$. Indeed, the closure under rescalings about all possible centers guarantees that $\Hcal^0(K) = 0$ for every $K \in \Cscr$ only if $\Cscr = \emptyset$.

Given any $K \in \Cscr$ (recall that $K$ is compact), we can then find a finite cover $\Uscr(K) \coloneqq \{\Bbf_{r_i}(x_i)\}_{i=1}^N$ of $K$ by open $(m+n)$-dimensional balls such that
\[
\sum_{i=1}^N r_i^\beta < \frac{1}{2}.
\]
Letting
\[
\Wcal(K) \coloneqq \bigcup_{i=1}^N \Bbf_{r_i}(x_i),
\]
we know that we can find a finite collection of sets $K_1,...,K_M$ such that for any $K \in \Cscr$, there exists $j \in \{1,...,M\}$ such that $K \subset \Wcal(K_j)$. This is due to the compactness of $\Cscr$ with respect to the Hausdorff topology.

Now write 
\[
\Uscr(K_j) = \{\Bbf_{r_{j,i}}(z_{j,i})\}_{i=1}^{N(K_j)}
\]
to denote the balls from $\Wcal(K_j)$. Since our collection is finite, we can certainly find $\alpha < \beta$ such that
\[
\sum_{i=1}^{N(K_j)} r_{j,i}^\alpha \leq \frac{1}{2}
\]
for each $j$.

We will show that $\alpha \in \Ascr(\Cscr)$ by inductively replacing each cover $\Uscr(K_j)$ by improved covers with reduced overlaps, which satisfy an $\alpha$-dimensional packing condition. At each stage we will replace a ball in the cover by a collection of smaller balls that cover $K$ more efficiently. We will later see how to further improve this efficient covering to one where the balls have comparable radii at each stage, in order to deduce that the stronger property~\eqref{eq:Minkowskicpct} holds for the exponent $\alpha$.

Fix any $K \in \Cscr$, and select $K_j$ such that $K \subset \Wcal(K_j)$. At the first stage of our subdivision, let $\Uscr_1 = \Uscr(K_j)$.

Assume that we have constructed $\Uscr_\ell$. Take any $\Bbf_r(x) \in \Uscr_\ell$. Choose $j = j(x,r)$ such that the blow-up $K_{x,r}$ satisfies
\[
    K_{x,r} \cap \Bbf_1 \subset \Wcal(K_{j(x,r)}),
\]
with corresponding family
\[
\Uscr(K_{j(x,r)}) = \{\Bbf_{r_{j(x,r),i}}(z_{j(x,r),i})\}
\]
of balls in $\Wcal(K_{j(x,r)})$.
Now we scale these balls back down to the correct size for $K$ itself, to get a family
\[
\Sscr(x,r) \coloneqq \{\Bbf_{r r_{j(x,r),i}}(x+rz_{j(x,r),i})\}_i.
\]
Then we have
\begin{equation}\label{eq:efficientcover1}
\sum_{\Bbf_\rho(w)\in\Sscr(x,r)} \rho^\alpha = \sum_i r^\alpha r_{j(x,r),i}^\alpha \leq \frac{1}{2}r^\alpha.
\end{equation}
So indeed, this provides a more efficient covering of the portion of $K$ inside $\Bbf_r(x)$, than $\Bbf_r(x)$ itself. Doing this for each ball within the family $\Uscr_\ell$, we may thus naturally define
\[
\Uscr_{\ell + 1} \coloneqq \bigcup_{\Bbf_r(x) \in \Uscr_\ell} \Sscr(x,r).
\]
Summing the packing estimate~\eqref{eq:efficientcover1} over the entire family $\Uscr_{\ell + 1}$ yields
\[
\sum_{\Bbf_\tau(x) \in \Uscr_{\ell+1}} \tau^\alpha = \sum_{\Bbf_r(x) \in \Uscr_\ell} \sum_{B_\rho(w) \in \Sscr(x,r)} \rho^\alpha \leq \frac{1}{2} \sum_{\Bbf_r(x) \in \Uscr_\ell} r^\alpha.
\]
Hence, by our choice of $\Uscr_1$,
\[
\sum_{\Bbf_\tau(x) \in \Uscr_{\ell+1}} \tau^\alpha \leq \frac{1}{2^{\ell + 1}}.
\]
We can continue refining our cover in this way indefinitely, taking $\ell \uparrow \infty$, which tells us that $\Hcal^\alpha(K) = 0$.

It remains to show the uniform boundedness of the upper Minkowski contents over the compact family $\Cscr$. We will use the above efficient covering for the sets in our family and will amend it appropriately using compactness to further ensure that the radii of the balls are comparable at each stage. This additional improvement will be precisely what we need in order to obtain the strengthened packing estimate~\eqref{eq:Minkowskicpct}. 

First of all, find the sizes of the smallest and largest balls among all the covers for all of the sets $K_j$, $j = 1,...,M$:
\begin{align*}
\bar{r} &\coloneqq \min\set{r_{j,i}}{\Bbf_{r_{j,i}}(z_{j,i}) \in \Uscr(K_j)}_{i,j}, \\
\bar{R} &\coloneqq \max\set{ r_{j,i}}{\Bbf_{r_{j,i}}(z_{j,i}) \in \Uscr(K_j)}_{i,j}.
\end{align*}

Notice that by construction, we necessarily have $\bar{R} \leq 2^{-\frac{1}{\alpha}} < 1$. For $K \in \Cscr$ fixed as before, let $\rho(\ell)$, $R(\ell)$ denote the radii of the smallest and largest balls in $\Uscr_\ell$ respectively, we have
\begin{equation}\label{eq:radiidecay}
\rho(\ell) \geq \bar{r}\rho(\ell - 1),
\end{equation}
with equality whenever we refine one of the balls in our covering with a rescaled cover that includes a rescaled ball of minimal size (among all the $\Uscr(K_j)$), between stage $\ell -1$ and $\ell$.

We want to ensure that $R(\ell)$ is comparable to $\rho(\ell)$. If this is not the case, namely if
\[
\frac{R(\ell)}{\rho(\ell)} > \bar{r}^{-1},
\]
then we subdivide all the large balls $\Bbf_\tau (w) \in \Uscr_\ell$ with $\tau > \bar{r}^{-1} \rho(\ell)$, several times if necessary, to obtain a new covering $\Uscr'_\ell$ such that
\begin{itemize}
    \item $\sum_{\Bbf_\tau(w) \in \Uscr'_\ell} \tau^\alpha \leq \frac{1}{2^\ell},$
    \item the radius of the largest ball is at most $\bar{R}^k R(\ell) \leq \bar{r}^{-1}\rho(\ell)$ for $k \in \N$ sufficiently large.
\end{itemize}
Note that the radius of the smallest ball in the cover remains the same in this procedure, since for any ball that we subdivide, the new balls will have radius at least $\rho(\ell)\bar{r}^{-1} > \rho(\ell)$.

Thus (not relabelling the new covers), we may assume that
\[
R(\ell) < \bar{r}^{-1}\rho(\ell) \qquad \text{for each $\ell$.}
\]
Combining this with~\eqref{eq:radiidecay}, we further have
\[
\rho(\ell) \in [\bar{r}^\ell,\bar{r}^{\ell-1}).
\]
We claim that this new even more tightly packed family of coverings gives the packing condition for our arbitrary fixed compact set $K \in \Cscr$. Choose any $r \in (0,1]$. Find the smallest $\ell \in \N$ such that $R(\ell) < r$, namely,
\[
R(\ell) < r \leq R(\ell - 1).
\]
We want to consider a covering by balls of radius exactly $r$, so we may enlarge all the balls in our cover $\Uscr_\ell$ to concentric ones with radius $r$ that still cover $K$. This in particular tells us that 
\[
N(K,r) \leq  \#\Uscr_\ell.
\]
We may now exploit the comparability of the sizes of all of these balls. Namely, we obtain
\begin{align*}
N(K,r)r^\alpha &\leq N(K,r)R(\ell-1)^\alpha \\
&\leq N(K,r) \bar{r}^{-\alpha}\rho(\ell - 1)^\alpha \\
&\leq \bar{r}^{-\alpha - 1}N(K,r)\rho(\ell)^\alpha \\
&\leq C(\alpha) \sum_{\Bbf_\tau(w) \in \Uscr_\ell} \tau^\alpha \\
&\leq C(\alpha).
\end{align*}
This concludes the proof.
\end{proof}

Let us now prove the statement of Proposition~\ref{prop:compactfam}, which allows us to apply the Work Raccoon Lemma~\ref{lem:workraccoon} to the family $\Cscr(\eps,\tilde{K})$.
\begin{proof}[Proof of Proposition~\ref{prop:compactfam}]
    For ease of notation, we will omit dependencies on $\eps$ and $\tilde K$. Clearly $\Cscr$ contains $K^0$, since one can take $x_k \equiv 0$ and $r_k \equiv 1$. The closure under rescaling property is trivial to check by simply rescaling the convergent sequence accordingly. The closure under Hausdorff convergence is proven as follows:
    
    Suppose that $\{K_j\}_j \subset \Cscr$, with $K_j^\infty \supset K_j \overset{d_H}{\longrightarrow} K$. For each $j$, we obtain a sequence $K^0_{x_{j,k},r_{j,k}}$ with
    \[
    K^0_{x_{j,k},r_{j,k}} \cap \cl{\Bbf}_1 \overset{d_H}{\longrightarrow} K^\infty_j.
    \]
    But then for each $j$, we can find $k(j)$ large enough and a subset $\tilde{K}^0_j \subset K^0_{x_{j,k(j)},r_{j,k(j)}}$ such that
    \[
    d_H(\tilde{K}^0_j \cap \cl{\Bbf}_1, K_j) \leq \frac{1}{2^j}.
    \]
    Combining this with the convergence of $K_j$ to $K$, we conclude that
    \[
    d_H(\tilde{K}^0_j \cap \cl{\Bbf}_1, K) \longrightarrow 0 \qquad \text{as $j \to \infty$.}
    \]
    In particular, the diagonal subsequence $\{K^0_{x_{j,k(j)},r_{j,k(j)}}\}_j$ of blow-ups can be taken for $K$, in order to see that it indeed lies in $\Cscr(\eps,\tilde{K})$.
\end{proof}

\section{Non-integer multiplicity points}\label{sct:nonintdensity}
We give a proof of the preliminary results  $(m-3)$-rectifiability and (local) upper Minkowski content bound for those singular points of $T$ that have non-integer multiplicities. Before we begin the proof of Lemma~\ref{lem:nonintsing}, let us recall some notation from~\cite{naber2017singular}. The (interior) $k$-th stratum $\Scal^k(T)$ (more generally defined for any integral varifold) is given by
\[
    \Scal^k(T) \coloneqq \set{x \in \spt T\setminus\spt\partial T}{\text{no tangent cone to $T$ at $x$ is $(k+1)$-symmetric}}.
\]
Observe that this definition coincides with the definition~\eqref{eq:stratum}. For fixed $\rho > 0$ and any $\delta > 0$, the quantitative \emph{$k$-th $\delta$-stratum $\Scal^k_{\delta}(T,\rho)$}  at scale $r$ is given by
\[
     \Scal^k_{\delta}(T,\rho) \coloneqq \set{x \in \spt T\setminus \spt\partial T}{\substack{\textup{\normalsize $\Fbf\big((\mathbf{e}_x)_\sharp(T_{x,s}) \mres B_1, V\mres B_1\big) \geq \delta s$  \ $\forall s \in (0,\rho)$ and} \\ \text{\normalsize any area-min. $(k+1)$-symm. cone $V \in \Ical^m(T_x\Sigma)$}}}
\]
where $\mathbf{e}_x$ is the exponential map $\mathbf{e}_x : T_x\Sigma \to \Sigma$, $B_1 \subset T_x\Sigma$. Recall that we call a varifold $V \in \Ical^m(\R^{m+\bar{n}})$ a \emph{$k$-symmetric cone} if $(\iota_{0,r})_\sharp V = V$ for every $r > 0$ and if there exists a $k$-plane $\pi \subset \R^{m+\bar{n}}$ such that $(\tau_y)_\sharp V = V$ for each $y \in \pi$.

Note that for area minimizing currents without boundary, one can replace the varifold metric in the above definition with the flat metric, since the two are then equivalent.

The result~\cite[Theorem~1.4]{naber2017singular} then states that given any $\rho, \delta > 0$ the $k$-dimensional upper Minkowski content of the stratum $\Scal^k_\delta(T,\rho)$ is locally bounded and $\Scal^k_\delta(T,\rho)$ is $k$-rectifiable.

\begin{proof}[Proof of Lemma~\ref{lem:nonintsing}]
    Fix $\eps > 0$, $Q \in \Nbb\setminus \{0\}$ and any $\Omega$ as in the statement of the lemma. Let
    \[
        E \coloneqq \Sing_{\geq Q+\eps} \cap \cl{\Omega}.
    \]
    Clearly $E$ is compact, due to the upper-semicontinuity of the density. We will proceed to show that
    \[
        E \subset \Scal^{m-3}_{\delta}(T,\rho) \qquad \text{for some scale $\rho$ and proximity threshold $\delta > 0$.}
    \]
    The idea is to show that if $T$ is sufficiently $\Fbf$-close to an area minimizing $(m-2)$-symmetric cone $V \in \Ical^m(T_x\Sigma)$ somewhere, this forces the mass ratio of $T$ to remain sufficiently close to an integer value (since the density of $V$ is always an integer). This will be inconsistent with our assumption that we are restricted to points of non-integer density. More specifically, we claim that there exists a scale $\rho > 0$ and a parameter $\eta \in (0,1)$ for which
    \begin{equation}\label{eq:massratiobd}
        \|T\|(\Bbf_s(p)) \leq (Q+1-\eta)\omega_m s^m \qquad \text{for every $p \in E$ and every $s \in (0,\rho)$.}
    \end{equation}
    Indeed, if this is not the case then we can extract a sequence of centers $p_k \in E$ and scales $s_k, \eta_k \todown 0$ for  which this fails, namely
    \[
        \|T\|(\Bbf_{s_k}(p_k)) > (Q+1-\eta_k)\omega_m s_k^m.
    \]
    Up to subsequence (not relabelled), we may assume that $p_k$ converges pointwise to some singular $Q$-point $p \in E$. The monotonicity formula and the convergence of the masses $\|T\|(\Bbf_r(p_k)) \to \|T\|(\Bbf_r(p))$ for all but a countable number of radii $r>0$ then tells us that for any such $r > 0$,
    \[
        \frac{\| T\|(\Bbf_r(p))}{\omega_mr^m} =\lim_{k \to \infty} \frac{\|T\|(\Bbf_r(p_k))}{\omega_m r^m} \geq \liminf_{k \to \infty}e^{C\Abf (s_k - r)}\frac{\| T\|(\Bbf_{s_k}(p_k))}{\omega_m s_k^m} \geq e^{-C\Abf r}(Q+1).
    \]
    In particular, $\Theta(T,p) \geq Q+1$, which contradicts the fact that $p \in E$.
    
    Now we may combine~\eqref{eq:massratiobd} with the properties of points in $\Scal^{m-3}_\delta(T,\rho)$ to reach a contradiction as follows. Observe that the following criteria holds:
    \begin{enumerate}[(a)]
        \item\label{itm:conedensity} If $V \in \Ical^m(T_x\Sigma)$ is an $(m-2)$-symmetric area minimizing cone, then necessarily $\Theta(S,0) = \frac{\|S\|(\Bbf_1)}{\omega_m}$ is a positive integer; \\
        \item\label{itm:Fclosemassclose} One can choose $\delta > 0$ sufficiently small such that if $\Fbf((\mathbf{e}_x)_\sharp(T_{x,s})\mres B_1, V\mres B_1) < \delta s$, then 
        \[
            \big|\|T\|(\Bbf_s(x)) - \|S\|(B_1)\big| < \min\Big\{\frac{\eta}{2},\frac{\eps}{4}\Big\}s^m \qquad \text{for any $x \in \R^{m+n}$;}
        \]
        \item By the monotonicity formula, for any $x \in E$ and any $s < \rho$ sufficiently small we have 
        \[
            \|T\|(\Bbf_s(x)) \geq e^{-C\Abf s} \Theta(T,x)\omega_m s^m \geq \big(Q+\frac{\eps}{2}\big)\omega_m s^m.
        \]
    \end{enumerate}
    The second property follows from the fact that
    \[
        \Fbf(V_1,V_2) = \Fcal(V_1,V_2) + \Fbf(|V_1|,|V_2|) \qquad \text{for any $V_1,V_2 \in \Ical^m$.}
    \]
    Thus, choosing $\delta > 0$ sufficiently small as in~\eqref{itm:Fclosemassclose}, if there is a point $x \in E$ for which $\Fbf((\mathbf{e}_x)_\sharp(T_{x,s})\mres B_1, V\mres B_1) < \delta s$ for some $s < \rho$ and some $(m-2)$-symmetric area minimizing cone $V \in \Ical^m(T_x\Sigma)$, then
    \[
        Q+\frac{\eps}{4} \leq \frac{\|S\|(\Bbf_s)}{\omega_ms^m}\leq Q+1-\frac{\eta}{2}.
    \]
    This, however, contradicts~\eqref{itm:conedensity}, so the claimed inclusion $E \subset \Scal^{m-2}_\delta(T,\rho)$ indeed holds true for $\rho > 0$ as in~\eqref{eq:massratiobd} and $\delta > 0$ as in~\eqref{itm:Fclosemassclose}.
\end{proof}




\end{document}